\documentclass[12pt]{article}
\usepackage{amsfonts}
\usepackage{amsthm}
\usepackage{amsmath}    % need for subequations
\usepackage{graphicx}   % need for figures
\usepackage{verbatim}   % useful for program listings
\usepackage{color}      % use if color is used in text
\usepackage{subfigure}  % use for side-by-side figures
\usepackage{multirow}
\usepackage{epstopdf}
\usepackage{hyperref}
\usepackage{algorithm}
\usepackage{algpseudocode}
\newtheorem{theorem}{Theorem}[section]

\theoremstyle{definition}
\newtheorem{remark}{Remark}[section]
\newtheorem{example}[remark]{Example}
\newcommand{\norm}[1]{\left\Vert#1\right\Vert}

\newcommand{\eps}{\varepsilon}

\def\half{\frac 1 2}

\providecommand{\jump}[1]{ [\mspace{-2.5mu}[ #1 ]\mspace{-2.5mu}] }
\allowdisplaybreaks

\begin{document}
	
	\title{The jump filter in the discontinuous Galerkin method for hyperbolic conservation laws}
	\author{Lei Wei\thanks{School of Mathematical Sciences, University of Science and Technology of China, Hefei, Anhui 230026, P.R. China. E-mail: weilei1@mail.ustc.edu.cn.}
		,~~
		Lingling Zhou\thanks{School of Mathematics and Information Science,
			Henan Polytechnic University, Jiaozuo, Henan 454000, P.R. China. E-mail: lingzhou@hpu.edu.cn. The research was partially supported by NSFC grant No. 12001171.}
		,~~and~
		Yinhua Xia\thanks{Corresponding author. School of Mathematical Sciences, University of Science and Technology of China, Hefei, Anhui 230026, P.R. China. E-mail: yhxia@ustc.edu.cn. The research was partially supported by National Key R\&D Program of China No. 2022YFA1005202/ 2022YFA1005200 and NSFC grant No. 12271498. }
	}
	\date{}
	\maketitle

	\begin{abstract}
		When simulating hyperbolic conservation laws with discontinuous solutions, high-order linear numerical schemes often produce undesirable spurious oscillations.   In this paper, we propose a jump filter within the discontinuous Galerkin (DG) method to mitigate these oscillations. This filter operates locally based on jump information at cell interfaces, targeting high-order polynomial modes within each cell.
		Besides its localized nature, our proposed filter preserves key attributes of the DG method, including conservation, $L^2$ stability, and high-order accuracy. We also explore its compatibility with other damping techniques, and demonstrate its seamless integration into a hybrid limiter. In scenarios featuring strong shock waves, this hybrid approach, incorporating this jump filter as the low-order limiter, effectively suppresses numerical oscillations while exhibiting low numerical dissipation.	
		Additionally, the proposed jump filter maintains the compactness of the DG scheme, which greatly aids in efficient parallel computing. Moreover, it boasts an impressively low computational cost, given that no characteristic decomposition is required and all computations are confined to physical space.
		Numerical experiments validate the effectiveness and performance of our proposed scheme, confirming its accuracy and shock-capturing capabilities.
		\smallskip

		\bigskip
		
		
	\end{abstract}

	\section{Introduction}
	Developing robust, accurate, and efficient high-order schemes for numerically solving hyperbolic nonlinear systems of conservation laws is crucial. Hyperbolic conservation laws are fundamental partial differential equations (PDEs) that bridge mathematics and mechanics, reflecting natural wave phenomena and finding wide applications in fluid dynamics, magnetohydrodynamics, multiphase flow, population dynamics, traffic flow, etc. Despite smooth initial and boundary conditions, nonlinear hyperbolic systems may exhibit flow-field discontinuities, such as shocks and contact discontinuities. The lack of robustness, especially concerning nonlinear instabilities, presents a significant challenge in applying high-order numerical schemes to technical problems. High-order approximations of these discontinuities may result in unbounded oscillations, leading to solver divergence. Hence, the development of shock-capturing techniques has received considerable attention in recent years to tackle this issue.
	
	There has been an abundance of work on the extension of classical shock-capturing methodologies to higher-order methods over the past three decades. Some popular higher-order shock-capturing methods include essentially non-oscillatory (ENO) schemes \cite{harten1987uniformly}, weighted ENO (WENO) methods \cite{liu1994weighted,jiang1996efficient}, subcell finite volume shock-capturing methods \cite{MR3608339,MR4583858}, and Runge-Kutta discontinuous Galerkin (RKDG) methods, just to name a few. Unlike WENO-type methods, DG methods do not require any reconstruction procedure since the numerical solution in each cell is already a polynomial of suitable degree. Since complete discontinuous basis functions are used, DG methods also have some advantages different from classical finite element methods, such as the allowance of arbitrary unstructured meshes with hanging nodes, and easy $h$-$p$ adaptivity. It's worth noting that DG methods exhibit extremely local data communications which demonstrate excellent parallel efficiency. The DG method is also very friendly to the GPU environment \cite{klockner2009nodal}. For the history and development of the DG method, we refer to the survey paper \cite{shu2014discontinuous}. Even though the DG method satisfies a cell-entropy inequality which implies an $L^2$ or energy stability \cite{hou2007solutions}, this stability is not strong enough to prevent spurious oscillations or even blow-ups of the numerical solution in the presence of strong discontinuities. Therefore, nonlinear limiters are often needed to control such spurious oscillations. The classic limiters mainly include minmod type total variation bounded (TVB) slope limiters \cite{cockburn1989tvb1,cockburn1989tvb2,cockburn1998runge}, moment limiters \cite{biswas1994parallel,krivodonova2007limiters}, and many improved version \cite{burbeau2001problem}. In recent years, another class of limiters that have been actively researched are the WENO-type limiters \cite{qiu2005rungelimiter}, including Hermite WENO limiters \cite{zhu2008runge,LuoHWENO2007}, simple WENO limiters \cite{zhong2013simple} and multi-resolution WENO limiters \cite{zhu2020high}. Additionally, a widely used limiter introduced by Kuzmin for DG methods is the vertex-based limiter \cite{kuzmin2010vertex,kuzmin2013slope,kuzmin2014hierarchical}.
	These limiters excel in controlling numerical oscillations; however, they can be computationally expensive and challenging to implement on unstructured or curved meshes, especially in three-dimensional scenarios.

	A different approach to shock capturing, dating back to the work of Von Neumann and Richtmeyer \cite{MR0037613} in the 1950s, is the use of artificial viscosity (AV). Jameson et al. \cite{jameson1981numerical,jameson1993artificial,jameson2001perspective} proposed a combination of a finite volume discretization method with dissipative terms and a Runge-Kutta time stepping method to yield an effective method for solving the Euler equations in arbitrary geometric domains. The historical
	drawback of the artificial viscosity approach is that the added terms are frequently too dissipative in certain
	regions of the flow. To minimize undesirable dissipation, Tadmor applied spectral viscosity (SV) \cite{tadmor1989convergence,MR1305618,maday1993legendre} with spectral methods to stabilize the calculation, introducing diffusion-like terms that depend on the spatial resolution and distinguish different frequencies of the numerical solution. This concept has successfully been applied to large-eddy simulations of high Reynolds number flows in \cite{karamanos2000spectral}. For higher order finite difference methods, Cook and Cabot \cite{cook2004high,cook2005hyperviscosity}
	incorporated artificial viscosity by scaling the physical viscosity terms
	such as dynamic viscosity, bulk viscosity, and shear viscosity. By adjusting the order of the polyharmonic operator, the viscosities can be strongly weighted toward high wavenumbers, thus imparting spectral-like behavior the dissipation. Subsequently, Fiorina and Lele \cite{fiorina2007artificial} further developed this approach to address entropy gradients arising from temperature and species discontinuities. They introduced a nonlinear artificial diffusivity to mitigate these effects. Moreover, Kawai and Lele \cite{kawai2008localized} extended the method to accommodate non-uniform and curvilinear meshes, enhancing its applicability. This approach was later adopted by other authors in the context of compressible turbulence simulations, and by Premasuthan et al. \cite{premasuthan2014computation,premasuthan2014computation2} for spectral difference method \cite{liu2006spectral} using only the dilation part of the shock sensor proposed by Bhagatwala and Lele \cite{bhagatwala2009modified}. The development of this type of method has seen numerous advancements. References \cite{fernandez2018physics,lodato2019characteristic,messai2024artificial,moro2016dilation,MR3209963} provide a wealth of new developments in this field.  As an application of the artificial viscosity and hyperviscosity methods for stabilizing radial basis function based algorithms refer to \cite{MR4135390,MR4517287}.
	
	Artificial viscosity has also been employed with DG methods to capture shocks. Persson and Peraire \cite{persson2006sub} introduced a sub-cell shock-capturing artificial viscosity approach for high-order DG methods which combines a highly selective spectral sensor, based on orthogonal polynomials, with a consistently discretized element-wise constant artificial viscosity added to the equations. Building on work by Persson and Peraire, Klockner et al. \cite{klockner2009nodal} thoroughly justified the detector’s design and further explained the scaling and smoothing steps necessary to turn the output of the detector into a local, artificial viscosity. Later, Persson \cite{persson2013shock} studied the application of sensor-based artificial viscosity to transient flow problems, and showed how different levels of smoothness in the sensor affect the solution.
	Barter and Darmofal \cite{barter2007shock} proposed a PDE-based artificial viscosity model appended to
	the system of governing equations to obtain a smoother artificial viscosity coefficient at the expense
	of solving an additional PDE. Besides, Lv et al. \cite{MR3534873} proposed an artificial viscosity scheme that is based on the eigen-spectrum of the discretization operator to minimize the stiffness addition caused by explicit calculations. The advantage of the formulation presented in \cite{MR3534873} is that it provides a rigorous expression for different orders of bases, rather than relying solely on a simple scaling argument as in \cite{persson2006sub}.
	Another approach is the entropy-based viscosity method \cite{burman2007nonlinear,guermond2008entropy} which was proposed in the DG framework in \cite{zingan2013implementation} and developed for other methods in \cite{hiltebrand2014entropy,chaudhuri2017explicit,tonicello2020entropy}. The underlying heuristics is that in regions where the entropy production is large, the entropy viscosity is large as well and the effective viscosity is limited to be first-order, thus making the method monotone in these regions. There are also approaches that base the shock sensor on the element residuals \cite{hartmann2002adaptive,hartmann2006adaptive,hartmann2013higher,bassi2009high}. However, these residual-based methods tend to present some robustness issues when used with high-order methods and strong shocks. Other models can refer \cite{meister2012application,moro2016dilation}. In addition, the computational efficacy of applying the aforementioned models within the framework of DG was evaluated in \cite{yu2020study}.
	
	In this paper, we present a novel local viscosity in the DG method, utilizing the jump information at cell interfaces, motivated by the oscillation-free discontinuous Galerkin (OFDG) method \cite{lu2021oscillation,liu2022essentially}, which incorporates jump information into the cellwise $L^2$-projection damping term. The concept of jump is of paramount importance as it allows for an adaptive determination of viscosity levels in each cell based on its smoothness characteristics. In regions of smooth flow, where minimal jumps occur between adjacent computational cells, the added viscosity is commensurately small, thereby significantly preserving the accuracy of the DG scheme. Conversely, in areas adjacent to shock waves, where discontinuities are more prominent, a larger viscosity is imparted to these cells, effectively mitigating numerical oscillations. Theoretical analysis confirms that the proposed DG scheme preserves several desirable properties, including conservation, $L^2$ stability, and optimal convergence rate. However, a direct implementation of this scheme may entail significant computational expenses and introduce additional stiffness, which can restrict the time step size for stability. To overcome this challenge, the introduction of SV can be seamlessly integrated within the spectral filtering framework, as demonstrated in \cite{gottlieb2001spectral,hesthaven2007nodal,meister2012application,hesthaven2017numerical}, resulting in a highly efficient computational implementation. The approximation properties of spectrally filtered Fourier and Legendre expansions have been rigorously analyzed in \cite{vandeven1991family,hesthaven2008filtering}, while spectral filtering has emerged as a widely applicable stabilization technique, as documented in \cite{don1994numerical,don2003multidomain}. Using this time-splitting method, the jump filter only needs to be applied after the standard DG schemes for each time step, and only the multiplication
	of the coefficients of the polynomials expanded from the numerical solution by a precomputed factor is necessary. This approach showcases an exceptional ability to automatically discern the magnitude of discontinuities and effectively mitigate numerical oscillations in their vicinity, without relying on problem-specific parameters. A splitting approach has also been applied to the cellwise $L^2$-projection damping term direclty, giving rise to the OEDG method \cite{peng2023oedg} with a specially chosen damping parameter to ensure scale and evolution invariance. Despite different derivations, both approaches yield similar filtering effects on high-order modes.
	Besides, the jump filter is seamlessly integrated into the hybrid limiter framework \cite{MR4673601}, resulting in a significant reduction of numerical dissipation while effectively suppressing numerical oscillations. Unlike traditional limiters such as minmod type slope limiters \cite{cockburn1989tvb1,cockburn1998runge} or WENO limiters \cite{qiu2005rungelimiter,zhu2020high}, the jump filter is computed directly in physical space, leading to substantial computational cost reductions. Additionally, the proposed viscosity is exclusively dependent on local cells, preserving the compactness of the DG scheme and facilitating parallel computations. The accuracy and robustness of our algorithm are demonstrated through a comprehensive set of numerical examples.

	This paper is structured as follows: In Section 2, we introduce the concept of local viscosity and the corresponding jump filter in the DG method for one- and two-dimensional hyperbolic conservation laws. We demonstrate the conservative, stable, and accurate nature of the DG scheme with this jump filter. Section 3 presents a comprehensive array of numerical examples, encompassing one-dimensional and two-dimensional cases, with a primary focus on the Euler equations. These examples cover various scenarios, including strong shock waves, shock reflection, interactions between shock waves and vortices, interactions with bubbles, and Kelvin-Helmholtz instability. Finally, Section 4 provides concluding remarks.
	\section{Scheme formulation}
	\subsection{One-dimensional scalar conservation laws}\label{sub:1dscalaralg} Consider the scalar conservation laws in one-dimensional form
	\begin{align}\label{eqn:scl}
		\begin{cases}
			u_t + f(u)_x =0, \; x\in \Omega = [a,b],\\
			u(x,0) = u_0(x),
		\end{cases}
	\end{align}
	with periodic or compactly supported boundary conditions.
	Let $\mathcal{T}_h$ denote the partition of $\Omega$ with the mesh denoted by $K_j=[x_{j-\half},x_{j+\half}]$ for $j=1,\cdots,N$. The center of the cell is $x_j=\half(x_{j-\half}+x_{j+\half})$ and the mesh size is denoted by $h_j=x_{j+\half}-x_{j-\half}$ with $h=\max\limits_{1\leq j \leq N}h_j$ being the maximum cell size. The mesh is assumed to be regular, which means the ratio between the maximum and minimum mesh sizes keeps bounded in the mesh refinements.
	
	The DG approximation space $V_h^k$  is  the space of all piecewise polynomial functions defined on  $\Omega$:
	\begin{align}
		V_h^k = \{v \in L^2([a,b]): \; v|_{K_j} \in P^k(K_j), \; j = 1,\cdots, N\}.
	\end{align}
	We also denote the inner product over the interval $K_j$ and the associated norm by
	\begin{align}
		(w, r)_{K_j} = \int_{K_j} \, w \, r \, \mathrm{d}x,\quad \norm{w}_{K_j} = \sqrt{(w, w)_{K_j}}.
	\end{align}
	It results in the semi-discrete DG scheme as follows: Find $u_h \in V_h^k$,  such that
	\begin{align} \label{scheme:DG}
		(\partial_t u_h, v)_{K_j} + <\widehat{f(u_h)}, v>_{\partial K_j} - (f(u_h), v_x)_{K_j} = 0, \; \forall v \in V_h^k \mbox{ and }  j = 1,\cdots, N,
	\end{align}
	where $<\widehat{f(u_h)}, v>_{\partial K_j} = \hat{f}_{j+\half}v(x_{j+\half}^{-}) - \hat{f}_{j-\half}v(x_{j-\half}^{+})$. Here, $\hat{f}_{j+\half}$ is the numerical flux defined on the cell interface $\partial K_j$ and in general depends on the values of $u_h$ from both sides of the interface $\hat{f}_{j+\half} = \hat{f}(u_h(x_{j+\half}^{-}), u_h(x_{j+\half}^{+}))$. The Godunov flux or the Lax-Friedrichs flux is often employed when solving scalar equations.
	
	In the linear case of  Eqn. \eqref{eqn:scl}, the smoothness of the solution depends only on the initial conditions. However,  for general nonlinear cases, the discontinuous solutions, known as shocks, may develop even if the initial data are smooth. The high order DG scheme would generate spurious oscillations when the solution contains shocks, so-called the Gibbs phenomenon. These spurious oscillations may not only generate some artificial structures, but also cause many overshoots and undershoots that make the numerical scheme less robust. To overcome this issue, the limiter approach was proposed, such as the TVB limiters \cite{cockburn1989tvb1,cockburn1989tvb2,cockburn1998runge}, the WENO limiters \cite{qiu2005rungelimiter,zhu2008runge}, and the vertex-based limiters \cite{kuzmin2010vertex,kuzmin2013slope,kuzmin2014hierarchical}, which is processed after each time stage. Another strategy, known as the entropy viscosity approach, is a predominant approach in the continuous Galerkin (CG) community for solving hyperbolic conservation laws \cite{MR0759810,MR2787948,MR3008293,guermond2008entropy}, which is based on the viscosity solution \cite{leveque1992numerical}.
	The viscosity solution of the Eqn. \eqref{eqn:scl} is defined as the solution to the equation
	\begin{align}
		u^\eps_t + f(u^\eps)_x = \eps u^\eps_{xx},
	\end{align}
	with $u^\eps(x,0) = u(x,0)$. The physical solution $u(x,t)$ is obtained as the limit viscosity solution $u(x,t) = \lim\limits_{\eps\rightarrow 0^+} u^\eps(x,t)$.
	
	We propose the following DG scheme with local viscosity: Find $u_h \in V_h^k$,  such that
	\begin{align}\label{eq:scavisscheme}
		(\partial_t u_h, v)_{K_j} + <\widehat{f(u_h)}, v>_{\partial K_j} - (f(u_h), v_x)_{K_j}  + (\nu_j (u_h)_x, v_x)_{K_j} = 0, \; \forall v \in V_h^k,
	\end{align}
	where $\nu_j$ is the local viscosity coefficient and will be determined later.
	In the DG scheme \eqref{eq:scavisscheme}, the introduction of this local viscosity term aims to mitigate spurious oscillations in shock solutions without compromising the high-order accuracy of smooth solutions. To achieve this, we define the local viscosity function as follows:
	\begin{align} \label{def_nu}
		\nu_j = \sigma_j (1-\xi^2),
	\end{align}
	where $\xi = \frac{2(x-x_j)}{h_j}$ and the jump viscosity coefficient $\sigma_j$ is given by
	\begin{align}
		\sigma_j = c_f(h_{j}\norm{\jump{u_h}}_{\partial K_j} + \sum\limits_{l=1}^k l(l+1) h_{j}^{l+1}\lVert\jump{\partial_{x}^{l} u_h}\rVert_{\partial K_j}),
	\end{align}
	with $c_f\geq0$ as a carefully designed free parameter. Here, $\norm{\jump{w}}_{\partial K_j} =| w(x_{j-\half}^{+}) -w(x_{j-\half}^{-})| + | w(x_{j+\half}^{+}) -w(x_{j+\half}^{-})|$ denotes the jump of $w$ at cell interface $\partial K_j$, and $\partial^l_x u_h$ represents the $l$-th derivative of $u_h$ with respect to $x$.
	
	Assuming $\sigma_j$ is constant in $K_j$, which can be estimated by the numerical solution at the current time stage and then through integration by parts, the local viscosity term is equivalent to:
	\begin{align}
		(\nu_j (u_h)_x, v_x)_{K_j}  &= -\sigma_j( \partial_x( (1-\xi^2) \partial_x u_h ), v )_{K_j}, \\
		&= -\sigma_j \frac{2}{h_j} ( \partial_\xi ( (1-\xi^2) \partial_\xi \tilde{u}_h ), \tilde{v})_{\widehat{K}},
	\end{align}
	where the reference cell $\widehat{K} = [-1,1]$.  The numerical solution on $K_j$ can be expressed by the Gauss-Legendre basis
	\begin{align}\label{eq:GLbasisequation}
		u_h|_{K_j} = \sum\limits_{l=0}^{k} u_j^l(t) P_l(x) = \sum\limits_{l=0}^{k} u_j^l(t) \tilde{P}_l(\xi).
	\end{align}
	Since $\tilde{P}_l(\xi)$ satisfies the Sturm-Liouville equation
	\begin{align}\label{eqn:SL}
		\frac{d}{d\xi} ( (1-\xi^2)\frac{d}{d\xi}\tilde{P}_l(\xi))  + l(l+1) \tilde{P}_l(\xi) = 0,
	\end{align}
	we obtain
	\begin{align}
		(\nu_j (u_h)_x, v_x)_{K_j}  &= \sigma_j \frac{2}{h_j} \sum\limits_{l=0}^{k} u_j^l (l(l+1)\tilde{P}_l(\xi),  \tilde{v})_{\widehat{K}}.
	\end{align}
	
	If we were to directly implement this term in \eqref{eq:scavisscheme}, it would affect the maximum allowable time step when using an explicit time discretization. Alternatively, this term can be included in a time-splitting fashion \cite{gottlieb2001spectral,hesthaven2007nodal}. In this approach, we first advance the original DG scheme \eqref{scheme:DG} by one time stage, followed by incorporating the diffusion term
	\begin{align} \label{scheme:split}
		(\partial_t u_h, v)_{K_j}  + (\nu_j (u_h)_x, v_x)_{K_j} = 0,
	\end{align}
	which is equivalent to
	\begin{align}
		(\partial_t \tilde{u}_h, \tilde{v})_{\widehat{K}}  +  \sigma_j (\frac{2}{h_j})^2 \sum\limits_{l=0}^{k} u_j^l (l(l+1)\tilde{P}_l(\xi),  \tilde{v})_{\widehat{K}} = 0.
	\end{align}
	By taking $\tilde{v} = \tilde{P}_l(\xi)$, it results that
	\begin{align}
		\frac{d}{dt} u_j^l + \sigma_j^l u_j^l = 0,\; l=0,\cdots, k,
	\end{align}
	where
	\begin{align}\label{para:damp}
		\sigma_j^l = \sigma_j (\frac{2}{h_j})^2  l(l+1).
	\end{align}
	For one time step with size $\Delta t$, this  ODE system can be solved as follows:
	\begin{align}\label{timespli:1dscalar}
		u_j^l(t+\Delta t)  = \exp(-\sigma_j^l \Delta t)u_j^l(t), \; l=0,\cdots, k.
	\end{align}
	Notice that for $l=0$, $u_j^0$ is the cell average and $\sigma_j^0 = 0 $, thus the local cell conservation of the DG scheme is preserved.

	In \eqref{para:damp}, to further minimize numerical dissipation, we opt for a smaller mode damping parameter $\sigma_j^l$ defined as follows:
	\begin{align}\label{para:damp2}
		\sigma_j^l = \sigma_{j,l} (\frac{2}{h_j})^2 l(l+1),\\
		\sigma_{j,l} =  c_f^l(h_{j}\norm{\jump{u_h}}_{\partial K_j} + \sum\limits_{m=1}^l m(m+1) h_{j}^{m+1}\lVert\jump{\partial_{x}^{m} u_h}\rVert_{\partial K_j}).\label{para:damp21}
	\end{align}
	Here, $\sigma_{j,l}$ replaces $\sigma_j$ in equation \eqref{para:damp}, focusing solely on jumps within the mode number $l$. Compared to equation \eqref{para:damp}, this damping technique eliminates the need to consider jump information from higher-order derivatives when filtering low-level moments.  Consequently, it enhances the preservation of original polynomial data, thereby maintaining the accuracy of the standard DG scheme more effectively. Besides, the free parameter $c_f^l = \frac{\beta_j }{4l(l+1)}$ where $\beta_j = |f^{'}(\bar{u}_j)|$, and $\bar{u}_j$ denotes the averages of the conservative variable of cell $K_j$.

	The idea of using jump information were explored in \cite{lu2021oscillation,liu2022essentially}, wherein an additional damping source term is introduced into the DG scheme:
	\begin{align}\label{eq:ofscheme}
		(\partial_t u_h, v)_{K_j} + <\widehat{f(u_h)}, v>_{\partial K_j} - (f(u_h), v_x)_{K_j}  +\sum\limits_{l=0}^k \frac{\tilde{\sigma}_j^l}{h_j}(u_h-P_h^{l-1}u_h, v)_{K_j} = 0,
	\end{align}
	where
	\begin{align} \label{eqn_damping_coef_1d}
		\tilde{\sigma}_j^{l} = \frac{2(2 l+1)}{2k-1} \frac{(h_j)^{l}}{l !}
		\left( \jump{ \partial_{x}^{l} u_h}_{j+\half}^2+\jump {\partial_{x}^{l} u_h }_{j-\half}^2\right)^{\frac 12},
	\end{align}
	and $P_h^l$ denotes the standard $L^2$ projection into $P^l$ polynomial space and $P_h^{-1} = P_h^0$. Here,  $\jump{w}_{j+\half}=w(x_{j+\half}^{+})-w(x_{j+\half}^{-})$. 
	
	Recently, a novel approach known as the oscillation-eliminating DG (OEDG) method
	has been introduced for hyperbolic conservation laws \cite{peng2023oedg}, building upon the foundations of the OFDG method. In the OEDG method, the damping term is also conceptualized as a modal filter, leveraging time-splitting techniques. The jump filter approach is rooted in the local viscosity term in \eqref{eq:scavisscheme}, whereas the OFDG/OEDG approach is based on the damping term in \eqref{eq:ofscheme}. For ease of comparison, under our notation the damping effect of the OEDG approach can be written as 
	\begin{align*}
		\sigma_j^{l,\text{OE}} = \sigma_{j,l}^{\text{OE}} (\frac{1}{h_j})^2,\\
		\sigma_{j,l}^{\text{OE}} =  c_f^{\text{OE}}(h_{j}\norm{\jump{u_h}}_{\partial K_j}+ \sum\limits_{m=1}^l \frac{2m+1}{m!}h_j^{m+1}\lVert\jump{\partial_{x}^{m} u_h}\rVert_{\partial K_j} ).
	\end{align*}
	where \begin{align*}c_f^{\text{OE}} =  
		\begin{cases} 
			0, & \text{if } u_h=\text{avg}(u_h), \\
			\frac{\beta_j}{2(2k-1)\norm{u_h-\text{avg}(u_h)}_{L^{\infty}(\Omega)}}, & \text{otherwise},
		\end{cases}
	\end{align*}
	and $\text{avg}(u_h)$ denotes the domain average of $u_h$.  Although the derivations differ, both strategies ultimately result in similar damping and filtering effects for high-order modes in numerical solutions. Both approaches can be easily adapted to multidimensional scenarios and unstructured meshes due to their local properties in the weak form of the scheme. Additionally, both the OEDG and jump filter methods are independent of the characteristic decomposition. The OEDG method is designed to be scale/evolution invariance through a careful choice of parameters, while such invariance is not explicitly guaranteed in the original OFDG formulation. The jump filter adopts the same parameter $\beta_j$ as used in the OEDG approach for dissipation control. However, it does not possess scale invariance.
	
	\remark{In the local viscosity function \eqref{def_nu}, we choose the weight function \(1 - \xi^2\) due to the the Sturm-Liouville equation \eqref{eqn:SL} for the Gauss-Legendre basis functions over reference cell $\hat{K}=[-1,1]$. It is important to note that using a more general Jacobi basis is equally applicable to the entire filtering process, requiring only a modification of the corresponding weight function. The Jacobi polynomials, denoted by $J^{(\alpha,\beta)}_l(\xi)$, are orthogonal with respect to the weight function $\omega^{(\alpha,\beta)}(\xi) := (1-\xi)^\alpha (1+\xi)^\beta$ over $\hat{K}$, as follows:
		\begin{equation}\label{eq:orth}
			\int_{-1}^{1} J^{(\alpha,\beta)}_{l_1}(\xi) J^{(\alpha,\beta)}_{l_2}(\xi) \omega^{(\alpha,\beta)}(\xi) \, d\xi = \gamma^{\alpha,\beta}_{l_1} \delta_{l_1l_2},  
		\end{equation}
		where $\gamma^{\alpha,\beta}_{l_1} = ||J^{\alpha,\beta}_{l1}||^2 _{\omega_{\alpha,\beta}}$, and $\delta_{l_1l_2} = 
		\begin{cases} 
			1, & \text{if } l_1=l_2 , \\
			0 ,& \text{otherwise}.
		\end{cases}$. The Jacobi polynomials satisfy the following Sturm-Liouville problem:
		\begin{equation}\label{eq:generalSLeq}
			\partial_\xi( \omega^{(\alpha+1,\beta+1)}(\xi)\partial_\xi J^{(\alpha,\beta)}_{l}(\xi))+ l(l+\alpha+\beta +1)\omega^{(\alpha,\beta)}(\xi)J^{(\alpha,\beta)}_l(\xi) =0. 
		\end{equation}
		Recall the special case of \( J^{(0,0)}_l(\xi) \), known as the Gauss-Legendre polynomials $\tilde{P}_l(\xi)$. The numerical solution on $K_j$ can be expressed using the Jacobi basis:
		Jacobi basis
		\begin{equation}\label{eq:generalJacobi}
			u_h|_{K_j} = \sum_{l=0}^{k}u_j^l(t)J^{(\alpha,\beta)}_l.
		\end{equation}
		Thanks to the Eqn. \eqref{eq:generalSLeq}, we have  
		\begin{align}
			(\nu_j (u_h)_x, v_x)_{K_j}  &= \sigma_j \frac{2}{h_j} \sum\limits_{l=0}^{k} u_j^l (l(l+\alpha+\beta+1)J^{(\alpha,\beta)}_l(\xi),  \tilde{v})_{\widehat{K}}.
		\end{align}
		This term can also be utilized with time-splitting techniques. The only distinction lies in the parameter selection, where the coefficients associated with \( l(l+1) \) are replaced by \( l(l+\alpha+\beta+1) \).

		\subsection{Stability and error estimates}
		The conservation property of the scheme can be readily obtained by setting $v=1$ in \eqref{eq:scavisscheme}, leading to
		\begin{align}
			\dfrac{d}{ d t} \int_{K_{j}} u_h \mathrm{d}x = -\hat{f}_{j+\half}+\hat{f}_{j-\half}.
		\end{align}
		Summing over $j$, with either periodic or compactly supported boundary conditions, we achieve global conservation.
		The $L^2$ stability of the semi-discrete scheme is achieved by taking $v=u_h$ in \eqref{eq:scavisscheme}. Summing it over $j$, we obtain
		\begin{align}
			\dfrac{1}{2}\dfrac{d}{ dt} \norm{u_h}^2 = -\sum_{j}\Theta_{j+\half}  - \sum_{j}\norm{\sqrt{\nu_j}(u_h)_x}^2_{K_j},
		\end{align}
		where $\Theta_{j+\half} =  \int_{u_h(x_{j+\half}^{-})}^{u_h(x_{j+\half}^{+})} \left( f(y) - \hat{f}(u_h(x_{j+\half}^{-}), u_h(x_{j+\half}^{+}) ) \right) \, dy\geq0$ as for the cell entropy inequality for \eqref{scheme:DG}  in \cite{jiang1994cell}, and $\norm{u_h(\cdot,t)}$ denotes the global $L^2$ norm. Consequently, we have the $L^2$ stability $\norm{u_h(\cdot,t)}\leq\norm{u_h(\cdot,0)}$.

		In the following, we will present  stability analysis and the optimal error estimate of the fully discrete DG scheme with the jump filter for the linear hyperbolic conservation laws, i.e., $f(u)=cu$. Without loss of generality, we assume $c=1$. 
		
		Owing to the identical RKDG stages in \eqref{rkdg1} and the similar filtering (damping/local viscosity) steps in \eqref{err_rk1}-\eqref{err_rk2}, the proof can be carried out by closely following the analysis of the OEDG approach in \cite{peng2023oedg}. For the fully discrete scheme RKDG scheme with the filter, the local viscosity provides the stability while the choice of the viscosity coeficient preserves spatial accuracy. Together with the general error analysis of the original RKDG method \cite{MR4125673,MR4161755,MR3977110}, this allows the argument to proceed in a straightforward manner. Therefore, we omit the detailed proof here and refer the reader to \cite{peng2023oedg} for further details.
		
		Take the upwind numerical flux $\widehat{f(u_h)}=(u_h)^{-}$ in (\ref{scheme:DG}) and the DG scheme is rewritten as
		\begin{align}\label{semi_dg1}
			(\partial_t u_h, v)= \mathcal{D}(u_h,v),
		\end{align}
		where $(w,r)=\sum_{j=1}^{N}(w,r)_{K_j}$, $\mathcal{D}(w,r)=\sum_{j=1}^{N}\mathcal{D}_j(w,r)$ and
		\begin{align}
			\mathcal{D}_j(w,r)=(w,r_{x})_{K_j}-w^{-}_{j+\frac{1}{2}}r_{j+\frac{1}{2}}^{-}+w^{-}_{j-\frac{1}{2}}r_{j-\frac{1}{2}}^{+}.\label{DG_Op}
		\end{align}
		Applying the $r$th-order $s$-stage RK method with Shu-Osher representation in the semi-discrete DG scheme (\ref{semi_dg1}), we obtain the fully discrete RKDG scheme
		\begin{align}
			u_h^{n,0}&=u_h^{n},\nonumber\\
			(u_h^{n,l+1},v) &= \sum_{0\leq m\leq l}\biggl\{c_{lm}(u_{h}^{n,m},v)+\tau d_{lm}\mathcal{D}(u_{h}^{n,m},v)\biggr\},\,\,l=0,1,\ldots,s-1,\label{rkdg1}\\
			u_h^{n+1} &= u_h^{n,s},\nonumber
		\end{align}
		where $u_h^{n}$ is the numerical solution at $n$th time level, $\tau$ is the uniform time step and $t_n=n\tau$. Define $u_{\sigma}^{0}=u_h^0$ and $u_{\sigma}^{n,0}=u_\sigma^n$. Then we present the RKDG scheme with a jump filter after each RK stage:
		\begin{align}\label{err_rk1}
			(u_h^{n,l+1},v)=\sum_{0\leq m\leq l}\biggl\{c_{lm}(u_{\sigma}^{n,m},v)+\tau d_{lm}\mathcal{D}(u_{\sigma}^{n,m},v)\biggr\},\,\,l=0,1,\ldots,s-1,
		\end{align}
		and $u_\sigma^{n,l+1}$ is the solution of the following initial value problem
		\begin{align}\label{err_rk2}
			\begin{split}
				&\frac{d}{dt}(u_{\sigma},v)_{K_j}+(\nu_{j}\partial_{x}u_{\sigma},v_{x})_{K_{j}}=0,\,\,\forall v\in V_h^k,\\
				&u_{\sigma}(x,0)=u_h^{n,l+1},
			\end{split}
		\end{align}
		where $u_{\sigma}(x,t)\in V_h^k$ and $\nu_{j}$ is defined by (\ref{def_nu}). Finally, $u_{\sigma}^{n+1}=u_\sigma^{n,s}$.
		
		\begin{theorem}\label{thm:stabilityanalysis}
			Suppose that without the jump filter, the fully discrete RKDG scheme is stable under the CFL condition $\tau\leq\gamma h^{\kappa}$, where $\kappa=1+1/r $ for $r\equiv 1$ (mod $4$), $\kappa=1+1/(r+1)$ for $r\equiv 2$ (mod $4$), and $\kappa=1$, $C_{*}=0$ for $r\equiv3$ (mod $4$). Let $u_{\sigma}^{n}$ be the solution of the fully discrete RKDG scheme with the jump filter in (\ref{err_rk1})-(\ref{err_rk2}). Then we have for any $n$ that
			\begin{align*}
				\norm{u_{\sigma}^{n+1}}^{2}\leq(1+C\tau)\norm{u_{\sigma}^{n}}^{2}
			\end{align*}
			under the same CFL condition $\tau \leq \gamma h^\kappa$. Here $C$ is a positive constant independent of $\tau$, $h$ and $n$.
		\end{theorem}

		\begin{theorem}\label{thm:erroraccuracy1}
			Suppose that without the jump filter, the fully discrete RKDG scheme is stable under the CFL condition $\tau\leq\gamma h^{\kappa}$, which is stated in Theorem \ref{thm:stabilityanalysis}. If the exact solution $u(x,t)$ is sufficiently smooth, $u_{\sigma}^{n}$ is the solution of the fully discrete RKDG scheme with the jump filter in (\ref{err_rk1})-(\ref{err_rk2}) and $\norm{u_\sigma^{0}- u(\cdot, 0) }\leq Ch^{k+1}$ with $k\geq1$, then we have the optimal error estimate
			\begin{align*}
				\norm{u_{\sigma}^{n}-u(\cdot,t_n)}\leq C( h^{k+1}+\tau^{r}).
			\end{align*}
		\end{theorem}
		
		\subsection{Two-dimensional scalar conservation laws}\label{sub:2dscalaralg}
		In the two-dimensional case, the situation is similar. Consider the two-dimensional scalar conservation laws
		\begin{align}\label{eqn:2sdscl}
			\begin{cases}
				u_t + \nabla\cdot\boldsymbol{F}(u)=0, \; \\
				u(\boldsymbol{x},0) = u_0(\boldsymbol{x}),
			\end{cases}
		\end{align}
		with suitable boundary conditions on the rectangular domain $\Omega$. After we have the regular Cartesian grid $\mathcal{T}_h$ of $\Omega$, the cell $K_{i,j}\in\mathcal{T}_h$ is denoted by $K_{i,j}=[x_{i-\half},x_{i+\half}]\times [y_{j-\half},y_{j+\half}]$ for $i=1,\cdots,N_x$; $j=1,\cdots,N_y$. The center of the cell $K_{i,j}$ is $(x_i,y_j)=(\half(x_{i-\half}+x_{i+\half}),\half(y_{j-\half}+y_{j+\half}))$ and the mesh size is denoted by $h_{x_i}=x_{i+\half}-x_{i-\half}$, $h_{y_j}=y_{j+\half}-y_{j-\half}$.
		
		The DG scheme with local viscosity on the rectangular mesh is: Find $u_h \in V_h^k = \{v \in L^2(\Omega): \; v|_{K_{i,j}} \in P^k(K_{i,j}), \; i= 1,\cdots, N_x; \; j = 1,\cdots, N_y\}$, such that for any $v \in V_h^k$, we have
		\begin{equation}
			\label{eq:semiloc_rec}
			(\partial_tu_h , v)_{{K_{i,j}}}  - (\boldsymbol{F}(u_h), \nabla v)_{{K_{i,j}}} +<\hat{\boldsymbol{F}}_h^{\boldsymbol{n}},v>_{\partial K_{i,j}}  +   (	\boldsymbol{\nu}_{i,j}\odot\nabla u_h ,  \nabla v)_{K_{i,j}}= 0 ,    \;
		\end{equation}
		where $<\hat{\boldsymbol{F}}_h^{\boldsymbol{n}},v>_{\partial K_{i,j}} = \int_{\partial K_{i,j}}\hat{\boldsymbol{F}}_h \cdot \boldsymbol{n} v \mathrm{d}s$ and $\odot$ denotes Hadamard product. Here, $ \boldsymbol{n} = (n_x, n_y)^T$ is the outward unit normal of the cell boundary $\partial_{K_{i,j}}$ and $\hat{\boldsymbol{F}}_h$ is a two point numerical flux. Following the one-dimensional case, we can choose the local viscosity function as
		\begin{align}
			\boldsymbol{\nu}_{i,j} =(\boldsymbol{\nu}_{i,j}^{(1)},\boldsymbol{\nu}_{i,j}^{(2)}) =  \sigma_{i,j} (\frac{h_{x_i}}{h_{y_j}}(1-\xi^2),\frac{h_{y_j}}{h_{x_i}}(1-\eta^2)),
		\end{align}
		where  $\xi = \frac{2(x-x_{i})}{h_{x_{i}}}$, $\eta = \frac{2(y-y_{j})}{h_{y_{j}}}$.
		
		By integration by parts, the local viscosity term is equivalent to
		\begin{equation}
			\label{eq:local2d}
			(\boldsymbol{\nu}_{i,j}\odot\nabla u_h ,  \nabla v)_{K_{i,j}} = -\sigma_{i,j}(\partial_\xi((1-\xi^2)\partial_\xi  \tilde{u}_h) + \partial_{\eta}((1-\eta^2)\partial_{\eta} \tilde{u}_h),\tilde{v})_{\widehat{K}},
		\end{equation}
		where the reference cell $\widehat{K} = [-1,1]^2$.
		
		The numerical solution on $K_{i,j}$ can be expressed by the Gauss-Legendre basis
		\begin{align}
			u_h|_{K_{i,j}} = \sum\limits_{l=0}^k\sum\limits_{p+q=l} u_{i,j}^{p,q}(t) P_{p,q}(x,y) =\sum\limits_{l=0}^k \sum\limits_{p+q=l} u_{i,j}^{p,q}(t) \tilde{P}_{p,q}(\xi,\eta).
		\end{align}
		Since $\tilde{P}_{p,q}(\xi,\eta)$ satisfies the Sturm-Liouville equation
		\begin{align}
			\frac{\partial}{\partial\xi} ( (1-\xi^2)\frac{\partial}{\partial\xi}\tilde{P}_{p,q}(\xi,\eta)) + \frac{\partial}{\partial\eta} ( (1-\eta^2)\frac{\partial}{\partial\eta}\tilde{P}_{p,q}(\xi,\eta))+ \lambda_{p,q} \tilde{P}_{p,q}(\xi,\eta) = 0,
		\end{align}
		where\begin{align}
			\lambda_{p,q}  = p(p+1)+q(q+1).
		\end{align}
		Then we have
		\begin{align}
			(\boldsymbol{\nu}_{i,j}\odot\nabla u_h ,  \nabla v)_{K_{i,j}}=\sigma_{i,j} \sum\limits_{l=0}^k\sum\limits_{p+q=l} u_{i,j}^{p,q}(t) ( \lambda_{p,q}\tilde{P}_{p,q}, \tilde{v})_{\widehat{K}}.
		\end{align}
		
		We also use the time splitting fashion to deal with this term as in the one-dimensional case. By taking $\tilde{v} = \tilde{P}_{p,q}(\xi,\eta)$, we obtain the following ODE system
		\begin{align}
			\frac{d}{dt} u_{i,j}^{p,q} + \sigma_{i,j}^l u_{i,j}^{p,q} = 0,\; p+q=l=0,\cdots, k,
		\end{align}
		where
		\begin{align}\label{para:2ddamp}
			\sigma_{i,j}^l = \sigma_{i,j} \frac{2}{h_{x_{i}}}\frac{2}{h_{y_{j}}}   \lambda_{p,q} .
		\end{align}
		For one time step $\Delta t$, it can be solved as follows:
		\begin{align}
			u_{i,j}^{p,q}(t+\Delta t)  = \exp(-\sigma_{i,j}^l \Delta t)u_{i,j}^{p,q}(t), \; p+q=l=0,\cdots, k.
		\end{align}
		The jump viscosity coefficient $\sigma_{i,j}$ is given by
		\begin{align}
			\sigma_{i,j} = c_{f,x}^{p,q}(h_{x_{i}}\norm{\jump{u_h}}_{\partial K_{i,j}^{\text{vertical}}} + \sum\limits_{l=1}^k l(l+1) h_{x_{i}}^{l+1}\sum_{|\boldsymbol{\alpha}|=l}\norm{\jump{\partial^{\boldsymbol{\alpha}}u_h}}_{\partial K_{i,j}^{\text{vertical}}}) \nonumber  \\
			+ c_{f,y}^{p,q}(h_{y_{j}}\norm{\jump{u_h}}_{\partial K_{i,j}^{\text{horizontal}}} + \sum\limits_{l=1}^k l(l+1) h_{y_{j}}^{l+1}\sum_{|\boldsymbol{\alpha}|=l} \norm{\jump{\partial^{\boldsymbol{\alpha}}u_h}}_{\partial K_{i,j}^{\text{horizontal}}}), \label{eq:2dalljumpvis1}
		\end{align}
		where $c_{f,x}^{p,q}$ and $c_{f,y}^{p,q}$ are  a free parameter. Here, the vector $\boldsymbol{\alpha} = (\alpha_1, \alpha_2)$ is the multi-index with $|\boldsymbol{\alpha}| = \alpha_1+ \alpha_2$, and $\partial^{\boldsymbol{\alpha}} w$ is defined as $\partial^{\boldsymbol{\alpha}}{w}= \partial^{\alpha_1}_{x}\partial^{\alpha_2}_{y}{w}$. Besides, $\norm{\jump{w}}_{\partial K_{i,j}^{\text{vertical}}}= \frac{1}{h_{y_j}}\int_{y_{j-\half}}^{y_{j+\half}}  \lVert\jump{w}_{i+\half,j}\rVert+\lVert\jump{w}_{i-\half,j}\rVert\mathrm{d}s$, where $\lVert\jump{w}_{i+\half,j}\rVert = |w(x_{i+\half,j}^{-},y) - w(x_{i+\half,j}^{+},y)|$ represents the absolute value of the jump of $w$ across the right interface of an element $K_{i,j}$. A similar formula can be used to express $\norm{\jump{\partial^{\boldsymbol{\alpha}}u_h}}_{\partial K_{i,j}^{\text{horizontal}}}$. The trapezoidal rule is employed to approximate this damping term along each edge of $K_{i,j}$ in all our two-dimensional tests. Just as in the one-dimensional case, to further minimize numerical dissipation, we can also adjust the mode damping parameter $\sigma_{i,j}^l$ as follows:
		\begin{align}\label{para:damp22d}
			\sigma_{i,j}^l = \sigma_{i,j,l} (\frac{2}{h_{x_{i}}}) (\frac{2}{h_{y_{j}}})  \lambda_{p,q},\\
			\sigma_{i,j,l} =  c_{f,x}^{p,q}( h_{x_{i}}\norm{\jump{u_h}}_{\partial K_{i,j}^{\text{vertical}}} + \sum\limits_{m=1}^l m(m+1) h_{x_{i}}^{m+1}\sum_{|\boldsymbol{\alpha}|=m}\norm{\jump{\partial^{\boldsymbol{\alpha}}u_h}}_{\partial K_{i,j}^{\text{vertical}}})  \nonumber \\
			+ c_{f,y}^{p,q}(h_{y_{j}}\norm{\jump{u_h}}_{\partial K_{i,j}^{\text{horizontal}}} + \sum\limits_{m=1}^l m(m+1) h_{y_{j}}^{m+1}\sum_{|\boldsymbol{\alpha}|=m}\norm{\jump{\partial^{\boldsymbol{\alpha}}u_h}}_{\partial K_{i,j}^{\text{horizontal}}}). \label{para:damp22d1}
		\end{align}
		We replace $\sigma_{i,j}$ by $\sigma_{i,j,l}$ in \eqref{para:2ddamp},  which considers jumps only within the mode number $l$. Additionally, this damping term \eqref{para:damp22d} serves to determine the magnitude of the viscosity term throughout this paper, allowing for adaptive damping of various moments based on the jump size. Again, it leads to reduced numerical dissipation compared with \eqref{eq:2dalljumpvis1}.
		
		Finally, it's worth noting that there are no fundamental obstacles to extending the jump filter to Cartesian meshes in higher dimensions.

		\subsection{Systems of conservation laws}\label{sub:sysalg}
		Since the system of conservation laws can be regarded as multiple scalar equations coupled together, extending the jump filter to systems does not pose any fundamental difficulty. Considering one-dimensional system $\boldsymbol{u}_t+\boldsymbol{f}(\boldsymbol{u})_x=\boldsymbol{0}$ with suitable boundary conditions where $\boldsymbol{u} = (u^{(1)}, \cdots , u^{(\mathcal{N})})$, and $\mathcal{N}$ represents the number of equations in the system. By employing the time-splitting method, we derive the following system of ODEs:
		\begin{align}\label{eq:1dsystemstime}
			\frac{d}{dt} \boldsymbol{u}_j^l + \sigma_j^l \boldsymbol{u}_j^l = \boldsymbol{0},\; l=0,\cdots, k,
		\end{align}
		where
		\begin{align}\label{para:dampsystems1d}
			\sigma_j^l = \max\limits_{1\leq s\leq \mathcal{N} }\sigma_{j,l}(u_{h}^{(s)}) (\frac{2}{h_j})^2 l(l+1).
		\end{align}
		Here, $\sigma_{j,l}(u_{h}^{(s)})$ denotes the jump viscosity coefficient based on the $s$-th variable of $\boldsymbol{u}_h=(u_{h}^{(1)},\cdots,u_{h}^{(\mathcal{N})})$, which is computed by \eqref{para:damp21}.  Note that the selection of $c_f^l$ in \eqref{para:damp21} is same for all components in our numerical simulations. 
		
		Each component $u_{h}^{(s)}$ can be represented on cell $K_j$ as $u_h^{(s)}|_{K_j} = \sum\limits_{l=0}^{k} u_j^{l,(s)}(t) P_l(x)$, and $ \boldsymbol{u}_j^l = (u_j^{l,(1)},\cdots,u_j^{l,(\mathcal{N})} )$. 
		For the one time step size $\Delta t$, this  ODE system can be solved as follows:
		\begin{align}\label{timespli:sys}
			\boldsymbol{u}_j^l(t+\Delta t)  = \exp(-\sigma_j^l \Delta t)\boldsymbol{u}_j^l(t), \; l=0,\cdots, k.
		\end{align}
		When dealing with systems of conservation laws, our approach is to compute the jump viscosity coefficient for each component of the system and select the maximum value to ensure the stability of the scheme when solving problems with strong shock waves.		
		The damping coefficients can also be chosen component-wise. However, based on our numerical tests, the component-wise selection tends to be less stable than the method described in \eqref{para:dampsystems1d}-\eqref{timespli:sys}, especially when applied to problems involving strong shock waves, such as Examples 3.6 and 3.7. Consequently, we use the method in \eqref{para:dampsystems1d}-\eqref{timespli:sys} for our numerical simulations.

		In this paper, we focus primarily on the Euler equations. In the 1D system, the parameter $c_f^l = \frac{\beta_j}{4l(l+1)} \frac{1}{\bar{H}_j}$ appears in \eqref{para:damp21}, while in 2D, $(c_{f,x}^{p,q},c_{f,y}^{p,q}) = \frac{1}{4\lambda_{p,q}} \frac{1}{\bar{H}_{i,j}}( h_{x_i}\beta_{i,j}^x,h_{y_j}\beta_{i,j}^y)$ where the additional factors $h_{x_i}$ and $h_{y_j}$ are introduced to control the dissipation associated with each spatial dimension. Here, $\beta_j$ is the spectral radius of the Jacobian matrix $\frac{\partial f}{\partial u}(\bar{u}_j)$, and in 2D, $\beta_{i,j}^x$ and $\beta_{i,j}^y$ are estimates of the local maximum wave speeds in the $x$- and $y$-directions, respectively.  $\bar{H}_j$ or $\bar{H}_{i,j}$ represents the average enthalpy within cell $K_j$ or $K_{i,j}$, respectively.

		The selection of these parameters is primarily based on numerical performance, with a focus on ensuring that appropriate numerical dissipation is introduced to control oscillations without causing excessive or insufficient dissipation. Our numerical tests in this paper demonstrate that this parameter selection effectively captures shock waves.  For more specific requirements, such as achieving scale/evolution invariance or making the parameters dimensionless, these can be addressed by adjusting the choice of $c_f^l$ or $(c_{f,x}^{p,q},c_{f,y}^{p,q})$, as indicated in \cite{peng2023oedg} for the OEDG method, which is designed with scale/evolution invariance in mind.

		We also explored the use of empirical parameters, such as setting $c_f^l = \frac{\beta_j}{4l(l+1)} \frac{1}{\bar{R}_j}$ or $c_f^{l,(s)} = \frac{\beta_j}{4l(l+1)} \frac{1}{|\bar{u}_h^{(s)}|}$ in 1D Euler equations.  Similar parameters can be applied in 2D. Here, $\bar{R}_j=\frac{\bar{p}_j}{\bar{\rho}_j^\gamma}$ represents the average entropy within cell $K_j$, and $\bar{u}_{h}^{(s)}$ denotes the average of $s$-th variable of the system. These parameters generally help control numerical oscillations. However, in some numerical examples, particularly in 2D problems, they result in excessive numerical dissipation, which can be mitigated by the hybrid approach  in \cite{MR4673601} outlined in Algorithm \ref{alg:hybridlimiterjumpfilter}.

		We are aware that classical limiters like the TVB limiter or the WENO-type limiter often yield satisfactory outcomes when applied in the characteristic space. In our case, we compute the jump and then apply damping to various moments based on the jump magnitude. However, it's worth noting that the eigenvectors of a matrix are not unique, and selecting different eigenvectors can result in different damping values if we calculate the jump in characteristic space, potentially resulting in varying computed results. To address this challenge, as in \cite{peng2023oedg} we directly compute the jumps in the conservative space and select an appropriate $c_f^l$ in \eqref{para:damp21}, or $c_{f,x}^{p,q}$ and  $c_{f,y}^{p,q}$ in \eqref{para:damp22d1}. For more demanding problems, characteristic decomposition remains applicable. However, our numerical experiments demonstrate that, when applied in the conservative variables, this approach can still effectively mitigate spurious oscillations, even in the presence of strong shocks.

		\section{Numerical experiments}
		This section presents several numerical examples, encompassing both one-dimensional and two-dimensional benchmark cases, with a primary focus on the compressible Euler equations. The aim is to showcase the effectiveness of the DG scheme equipped with a jump filter in maintaining the accuracy of the standard DG scheme. Particularly in scenarios featuring strong shock waves, we observe that the jump filter efficiently controls numerical oscillations, yielding results comparable to those achieved with the TVB limiter.
		
		To further reduce numerical dissipation, we adopt the hybrid approach detailed in \cite{MR4673601}. This approach integrates the jump filter as the lower-order limiter. Specifically, we use the jump indicator introduced in \cite{MR4673601} to identify ``troubled" cells. On these cells, we apply the jump filter to compute lower-order polynomials that control numerical oscillations. Finally, we combine these lower-order polynomials with the higher-order DG polynomials, weighting them appropriately to obtain the mixed polynomials. We summarize this hybrid procedure for 1D scalar problems as an example in Algorithm \ref{alg:hybridlimiterjumpfilter}.
		
		\begin{algorithm}
			\caption{Hybrid Procedure Associated with the Jump Filter}\label{alg:hybridlimiterjumpfilter}
			\begin{algorithmic}[1]
				\State \textbf{Initialize:} For each cell $K_j$, compute the DG solution $u_{h}^{DG}$ from the standard DG scheme.
				\State \textbf{Determine the jump indicator:} For each cell $K_j$, calculate the jump indicator as $\eta_j = \frac{\lVert \jump{u_h^{DG}} \rVert_{\partial K_j}}{2h_j}$.
				\If{$\eta_j > 1$}
				\State \textbf{Calculate the damping parameter:} Determine the damping parameter $\sigma_j^l$ by \eqref{para:damp2}-\eqref{para:damp21}.
				\State \textbf{Obtain the limiting solution:} Compute the limiting solution $u_{h}^{lim}$ by \eqref{timespli:1dscalar}.
				\State \textbf{Compute the hybrid solution:} Calculate the final hybrid solution as $u_{h}^{hyd} = \omega_j u_{h}^{DG} + (1 - \omega_j) u_{h}^{lim}$, where the blending factor $\omega_j$ is given by $\omega_j = \frac{\max_j \eta_j - \eta_j}{\max_j \eta_j - 1}$.
				\Else
				\State \textbf{No limiting strategies are needed:}
				$u_{h}^{hyd}=u_{h}^{DG}$.
				\EndIf
			\end{algorithmic}
		\end{algorithm}

		Numerical experiments will demonstrate that employing this hybrid approach, with the lower-order polynomials computed using the jump filter, yields results comparable to those obtained by using the TVB limiter for the lower-order polynomials. However, the key advantage of the jump filter lies in its avoidance of characteristic decomposition, resulting in relatively lower computational costs.

		In the following tests, we employ the local Lax-Friedrichs (LLF) numerical flux. However, it's worth noting that other numerical fluxes, such as the HLL (Harten-Lax-van Leer) \cite{harten1983upstream} and HLLC (Harten-Lax-van Leer-Contact) numerical fluxes \cite{toro2013riemann}, can also be effectively employed in this context. For time discretization, we utilize the third-order explicit strong-stability-preserving Runge-Kutta (SSP-RK) method \cite{Gottlieb2001Strong}. Furthermore, our jump filter or hybrid limiter, in conjunction with the jump filter, is applied in a post-processing step after each Runge-Kutta substage.	
		\subsection{One dimensional tests}\label{sub:1dexamples}  
		\begin{example}[Burgers equation]\label{exm:burgers}
			Consider the inviscid Burgers equation
			\begin{align}
				u_t + (\frac{u^2}{2})_x = 0, \; x\in (-\pi,\pi),
			\end{align}
			with the initial condition $u(x,0) = \frac{1}{2}  +  \sin(x)$. Table \ref{tab:Burgers_accuracy} presents the errors and orders of numerical solutions with and without the jump filter, at time $T=0.5$ when the solution is still smooth. It is evident that the DG scheme equipped with the jump filter achieves the desired accuracy. Note that the numerical results were obtained using the jump filter without employing the hybrid procedure. Moreover, as the mesh is refined, the error between the scheme with and without the jump filter diminishes significantly.
			
			Additionally, in Fig. \ref{fig:Burgers1d}, numerical solutions and errors after shock formation at time $T=1.5$ are depicted. Pointwise errors reveal controlled spurious oscillations around the shock, with a very small pollution region. Furthermore, both the cell average and the entire polynomial of numerical solutions remain oscillation-free.
			
			Throughout all cases illustrated in Fig. \ref{fig:Burgers1d}, the local jump viscosity coefficient correlates closely with the numerical errors. This suggests that it can serve as a reliable shock indicator, facilitating adaptive filtering.
		\end{example}	
		\begin{table}[htb]\small
			\caption{\label{tab:Burgers_accuracy}Numerical errors and orders
				in Example \ref{exm:burgers}. } \centering
			\resizebox{1.\textwidth}{!}{
				\medskip
				\begin{tabular}{|c|r||cc|cc|cc|cc|}  \hline
					&  &\multicolumn{4}{c|}{results with filter} & \multicolumn{4}{c|}{results without filter} \\ \hline
					& $N$ & $L^2$ error & order & $L^\infty$ error & order & $L^2$ error & order & $L^\infty$ error & order \\ \hline \hline
					\multirow{6}{0.6cm}{$P^1$}
					&   20    & 2.60e-02    &   --    & 3.74e-02    &   --    &  1.01e-02    &   --& 9.92e-03    &   --    \\
					&   40    & 3.96e-03    &   2.71  &5.66e-03    &2.72    &2.85e-03 &1.83 &3.00e-03 &1.72  \\
					&   80   &7.52e-04  &2.40 &9.76e-04 &2.54    &7.73e-04 &1.88 &7.94e-04 &1.92 \\
					&   160  &1.90e-04 &1.98 &2.02e-04 &2.27   &2.04e-04 &1.92 &2.03e-04 &1.96  \\
					&  320  &5.07e-05 &1.91 &4.75e-05 &2.09    &5.29e-05 &1.95 &5.13e-05 &1.99 \\
					&  640   &1.32e-05 &1.94 &1.23e-05 &1.94   &1.35e-05 &1.97 &1.29e-05 &1.99 \\\hline
					\multirow{6}{0.6cm}{$P^2$}
					&   20    &1.10e-01    &   --    &2.37e-01    &   --    &6.70e-04    &   -- &1.39e-03    &   -- \\
					&   40   &1.41e-02 &2.97 &3.32e-02 &2.83  &9.42e-05 &2.83 &2.13e-04 &2.70\\
					&   80   &6.03e-05 &7.87 &1.50e-04 &7.79    &  2.13E-05    &   3.21  &  6.54E-05    &   3.09  \\
					&   160    &2.94e-06 &4.36 &6.75e-06 &4.47   &1.81e-06 &2.89 &3.62e-06 &2.92 \\
					&  320    &2.84e-07 &3.37 &5.66e-07 &3.58&2.37e-07 &2.94 &4.83e-07 &2.91 \\
					&  640   &3.26e-08 &3.12 &6.52e-08 &3.12  &3.02e-08 &2.97 &6.20e-08 &2.96\\ \hline
					\multirow{6}{0.6cm}{$P^3$}
					&   20    & 8.70e-03    &  --   & 1.90e-02    &   --    &  3.76e-05    &   --& 6.42e-05    &   -- \\
					&   40   &1.02e-04 &6.42 &3.65e-04 &5.71    &3.88e-06 &3.27 &9.45e-06 &2.76 \\
					&   80   &1.50e-06 &6.08 &5.34e-06 &6.09      &2.87e-07 &3.76 &5.86e-07 &4.01\\
					&   160   &4.15e-08 &5.18 &1.14e-07 &5.55   &1.99e-08 &3.85 &3.85e-08 &3.93\\
					&  320   &1.70e-09 &4.61 &4.46e-09 &4.67    &1.33e-09 &3.90 &2.46e-09 &3.97  \\
					&  640  &9.22e-11 &4.20 &1.94e-10 &4.53  &8.69e-11 &3.94 &1.55e-10 &3.99 \\ \hline
			\end{tabular}}
		\end{table}
		\begin{figure}[ht]
			\centering
			\begin{minipage}[b]{0.3\linewidth}
				\centering
				\includegraphics[width=\linewidth]{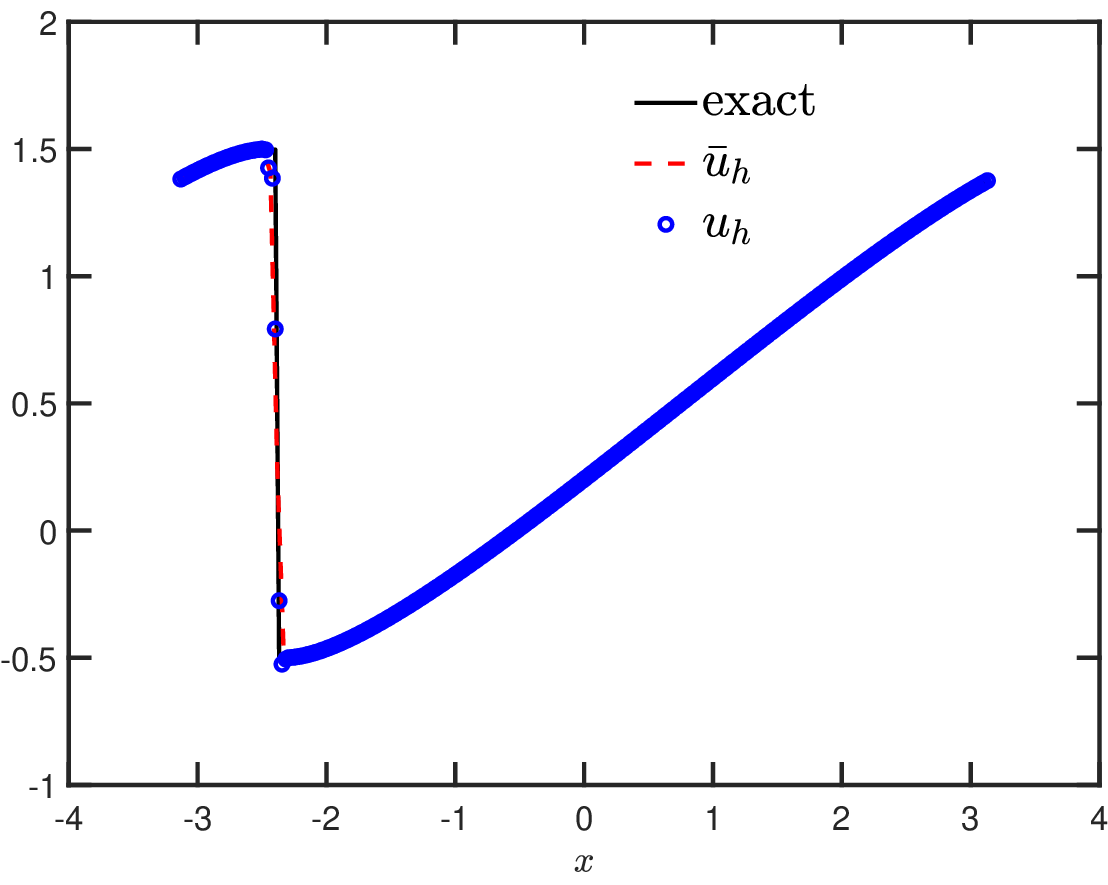}
			\end{minipage}
			\hfill
			\begin{minipage}[b]{0.3\linewidth}
				\centering
				\includegraphics[width=\linewidth]{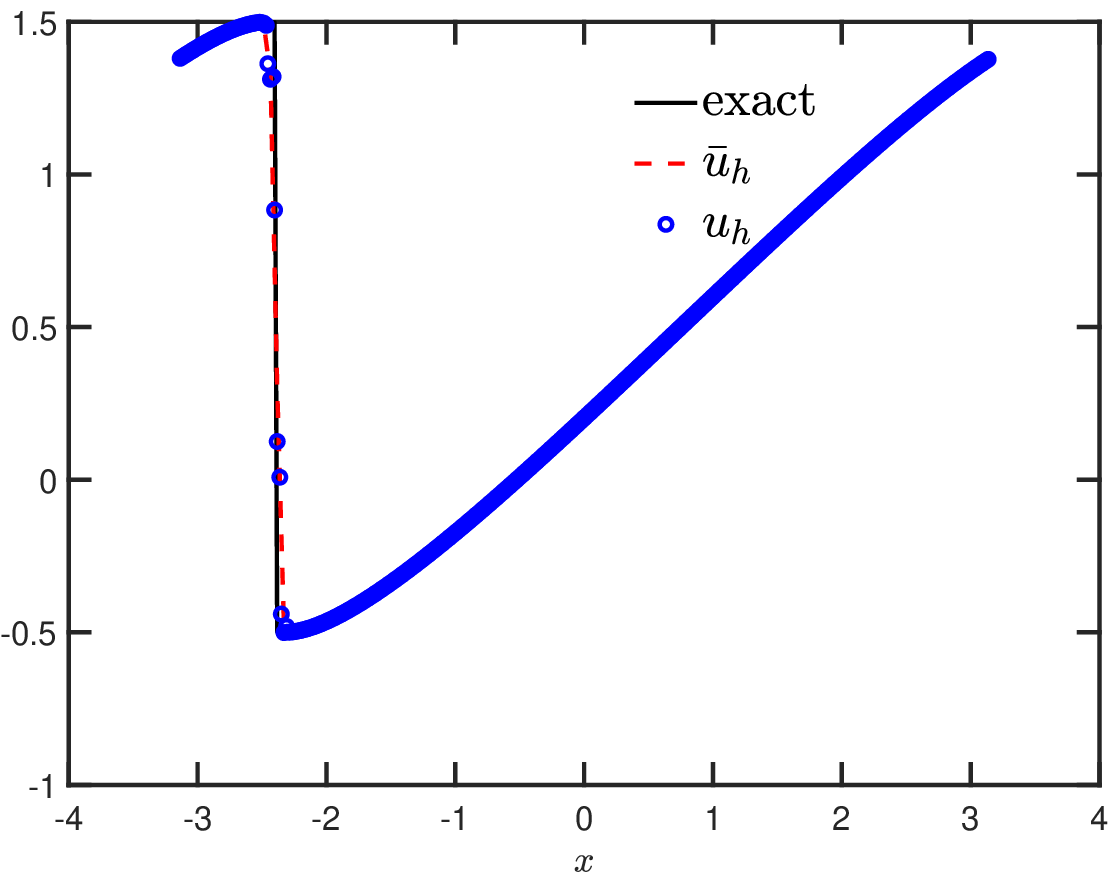}
			\end{minipage}
			\hfill
			\begin{minipage}[b]{0.3\linewidth}
				\centering
				\includegraphics[width=\linewidth]{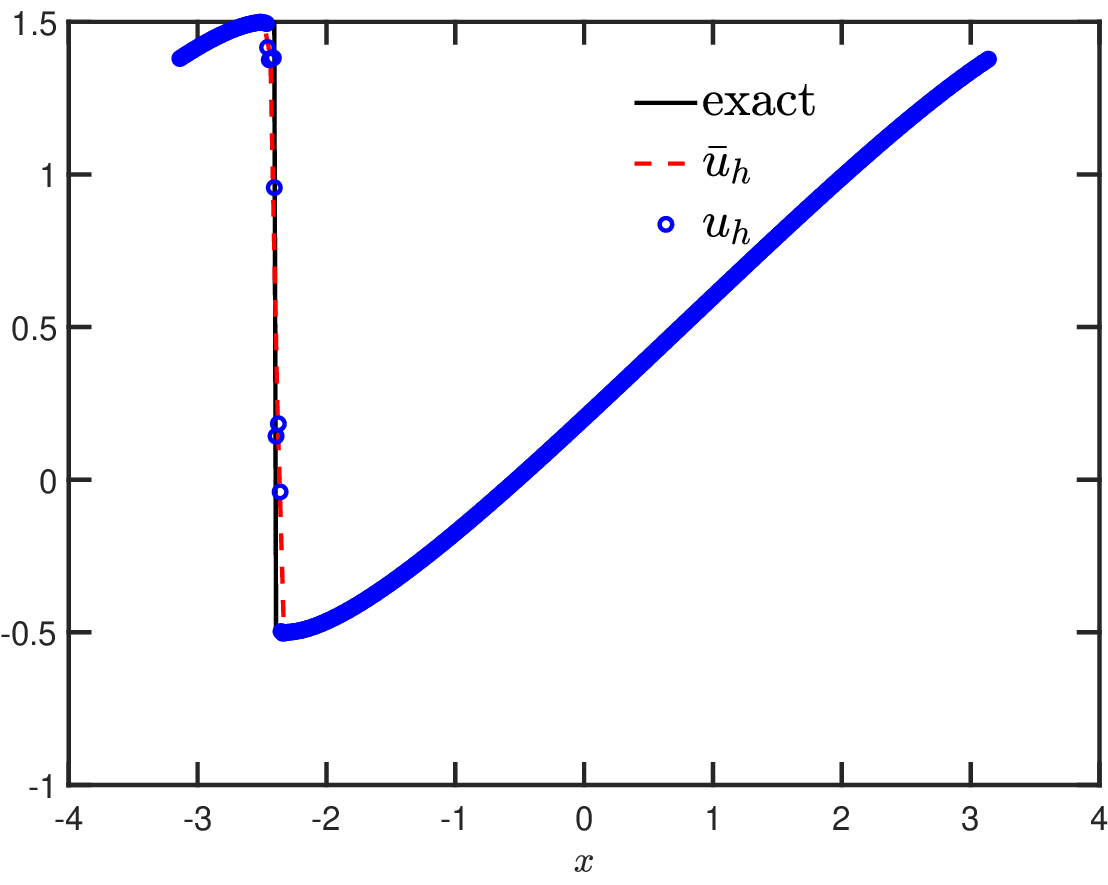}
			\end{minipage}
			\hfill
			\begin{minipage}[b]{0.3\linewidth}
				\centering
				\includegraphics[width=\linewidth]{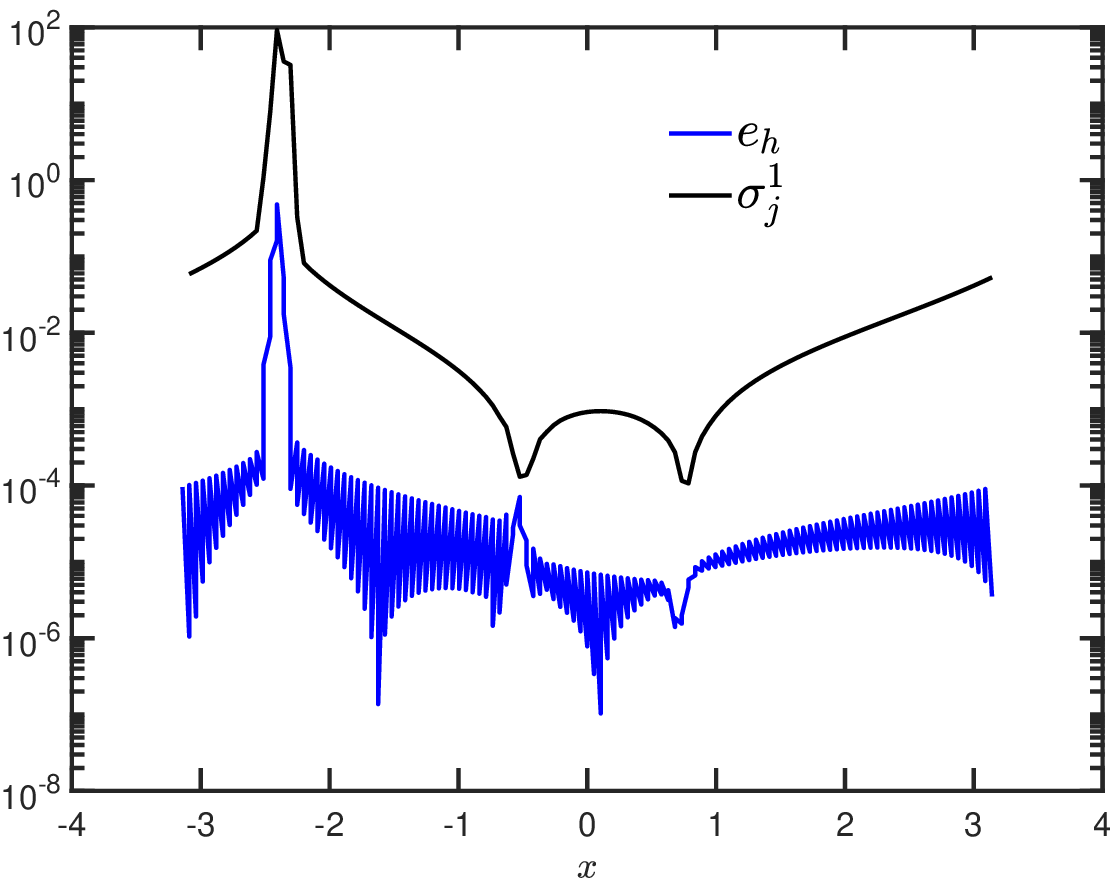}
			\end{minipage}
			\hfill
			\begin{minipage}[b]{0.3\linewidth}
				\centering
				\includegraphics[width=\linewidth]{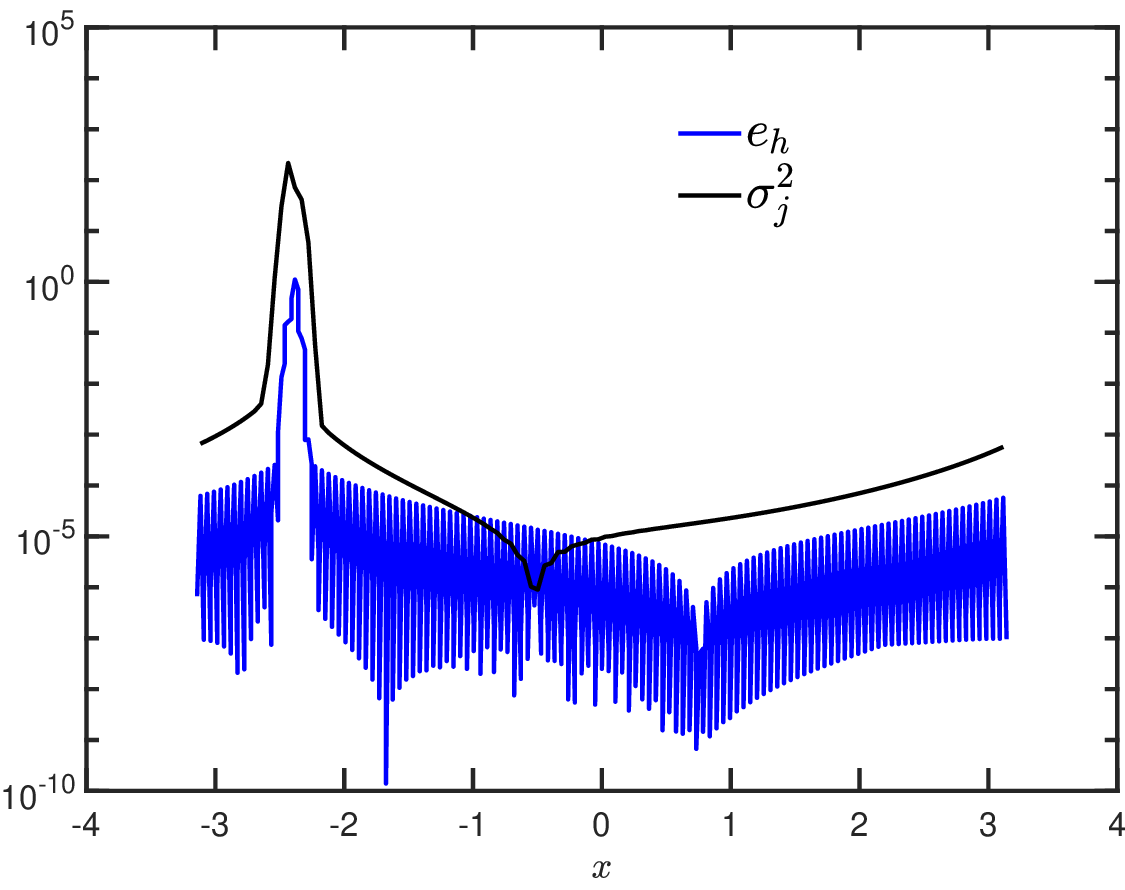}
			\end{minipage}
			\hfill
			\begin{minipage}[b]{0.3\linewidth}
				\centering
				\includegraphics[width=\linewidth]{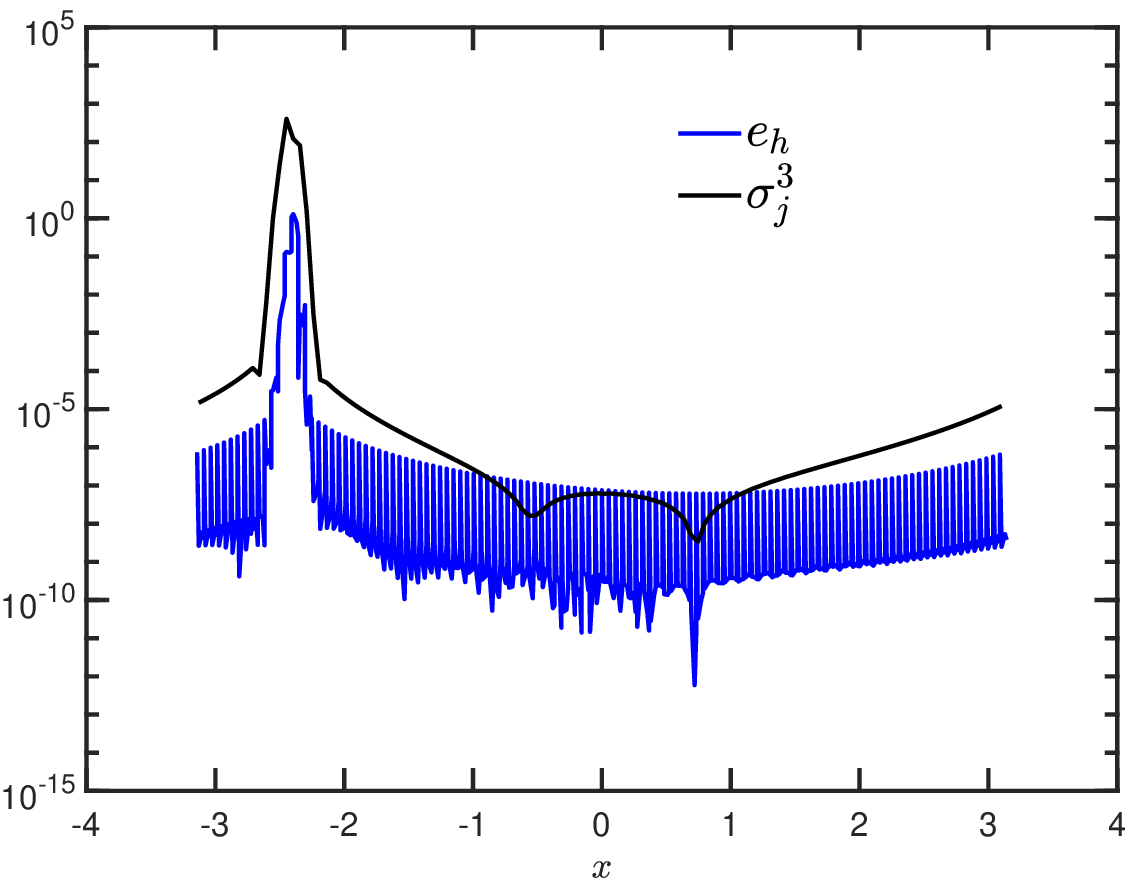}
			\end{minipage}
			\caption{Numerical results of DG scheme with the jump filter for Example \ref{exm:burgers}, with $N=120$ and $P^k$, $k=1,2,3$,  at time $T=1.5$.
				Top: numerical solutions of full polynomials and cell averages. Bottom: pointwise errors and $\sigma_{j,l}$ where $l=k$.}
			\label{fig:Burgers1d}
		\end{figure}
		
		Next, we delve into the Euler equations:
		\begin{equation}	
			\boldsymbol{U}_t +  \boldsymbol{F}(\boldsymbol{U})_x = 0,
		\end{equation}
		with $	\boldsymbol{U} = (\rho , m ,E )^T$, $\boldsymbol{F}(\boldsymbol{U}) = (m ,\frac{m^2}{\rho} + p, \frac{m}{\rho}(E + p))^T$,
		where $m = \rho u$, $E = \frac{1}{2} \rho u^2 + \rho e$, and $p = (\gamma - 1) \rho e$. Here, $\rho$ denotes density, and $u$ and $m$ represent velocity and momentum, respectively. $E$ stands for the total energy, $p$ denotes pressure, and $e$ signifies internal energy. $\gamma$ represents the ratio of specific heats, and $H = (E+p)/\rho$ is the enthalpy. For the Euler equations in 1D, we set $c_f^l = \frac{\beta_j}{4l(l+1)} \frac{1}{\bar{H}_j}$ in \eqref{para:damp21}, where $\bar{H}_j$ represents the average enthalpy within cell $K_j$. Notably, in Euler equations, we observed that the factor $\frac{1}{\bar{H}_j}$ enhances the scheme's capability to capture strong shock waves. In this section, we present plots of the full polynomials of DG solutions.
		\begin{example}[Lax problem]\label{ex:Lax}
			We consider the Lax problem for the Euler equations with the initial condition
			\begin{equation}
				\left(\rho,u,p,\gamma\right)^T=
				\begin{cases}
					\left(0.445,0.698,3.528,1.4\right)^T,  & x\in\left[-1,0\right) \\
					\left(0.5,0,0.571,1.4\right)^T,   & x\in\left[0,1\right].
				\end{cases}
			\end{equation}
			The density $\rho$ calculated with the DG scheme with the jump filter and the associated hybrid limiter, is plotted against the exact solution in Fig. \ref{fig:LaxBlast}. It is evident that the DG scheme with the jump filter maintains sharp, oscillation-free resolution near shocks. Moreover, employing the hybrid limiter associated with the jump filter as the lower-order limiting strategy for this problem also yields oscillation-free results. In particular, this hybrid approach can further reduce computational costs, as discussed in \cite{MR4673601}.
		\end{example}
		\begin{example}[Blast waves]  \label{ex:Blast}
			We then consider the interaction of two blast waves with the initial condition:
			\begin{equation}
				\left(\rho,u,p,\gamma\right)^T=
				\begin{cases}
					\left(1,0,10^3,1.4\right)^T,  & x\in\left[0,0.1\right),\\
					\left(1,0,10^{-2},1.4\right)^T,   & x\in\left[0.1,0.9\right],\\
					\left(1,0,10^2,1.4\right)^T,   & x\in\left[0.9,1\right].
				\end{cases}
			\end{equation}
			Reflecting boundary conditions are applied at both endpoints, and the numerical solutions are depicted in Fig. \ref{fig:LaxBlast}. Here, the ``exact" solution refers to the numerical solution computed by the fifth-order finite difference WENOZ-H scheme \cite{wan2021new} with $20000$ grid points. Initially, it's evident that the DG scheme with the jump filter maintains sharp, oscillation-free resolution near shocks. To further minimize numerical dissipation, we employ the hybrid limiter associated with the jump filter for this problem. Comparing the results, we observe that the hybrid limiter exhibits superior resolution compared to the DG scheme with the jump filter alone.
			\begin{figure}[ht]
				\centering
				\begin{minipage}[b]{0.4\linewidth}
					\centering
					\includegraphics[width=0.8\linewidth]{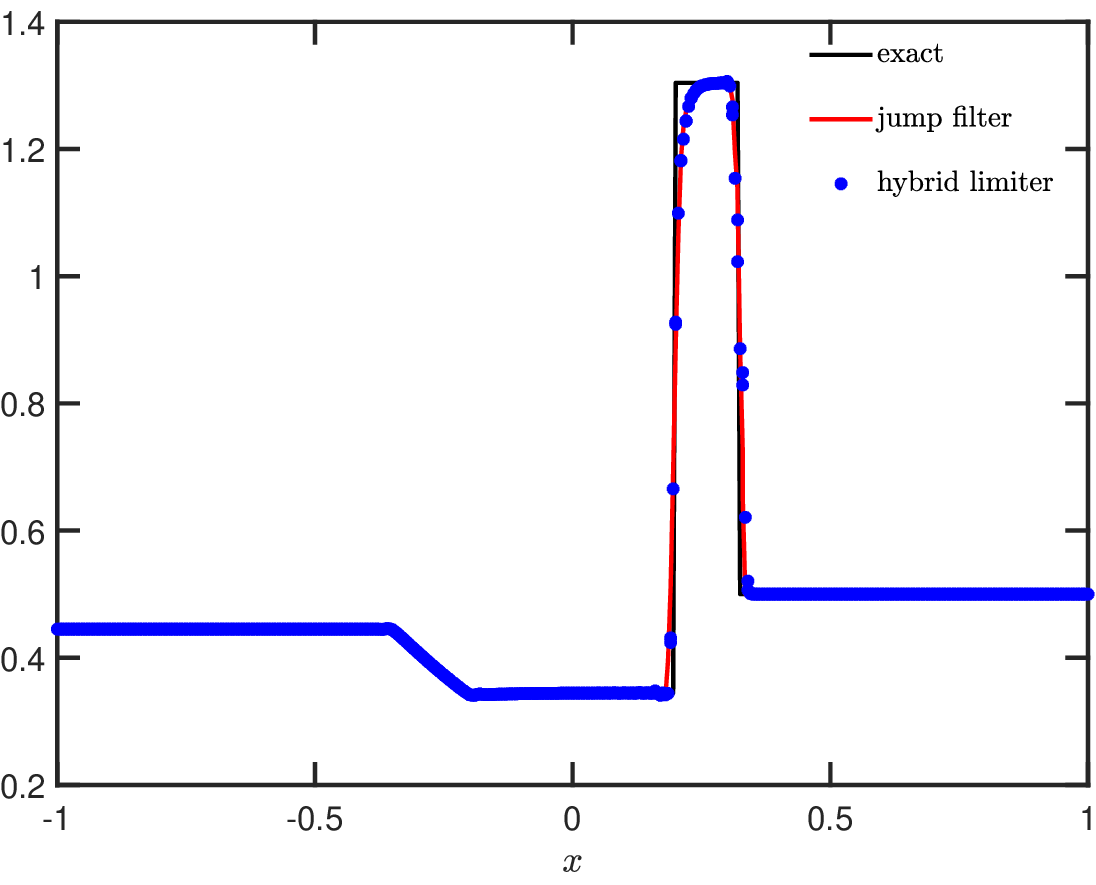}
				\end{minipage}
				%	\hfill
				\begin{minipage}[b]{0.4\linewidth}
					\centering
					\includegraphics[width=0.8\linewidth]{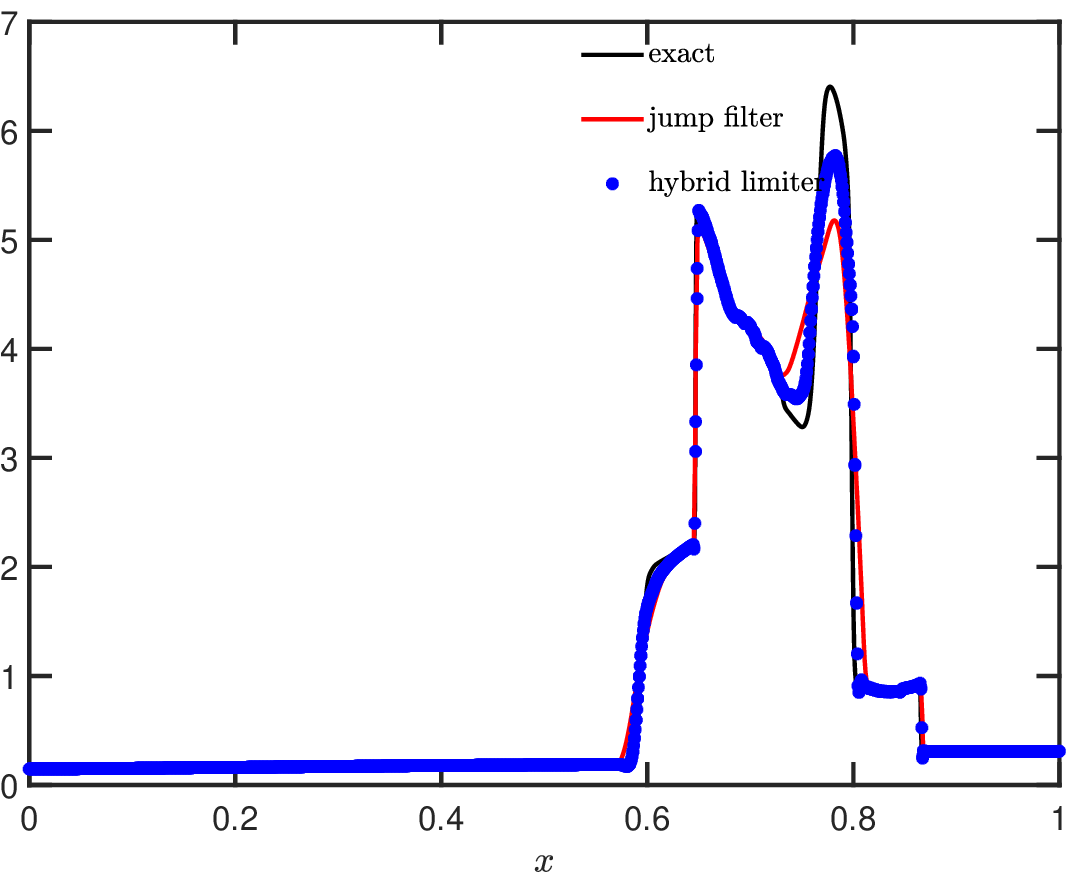}
				\end{minipage}
				\caption{Numerical results of DG scheme with the jump filter and hybrid limiters. Left: Example \ref{ex:Lax} , with $N=200$ and $P^2$,  at time $T= 0.13$. Right: Example \ref{ex:Blast} , with $N=600$ and $P^2$,  at time $T= 0.038$.}
				\label{fig:LaxBlast}
			\end{figure}
		\end{example}
		
		\begin{example}[Shu-Osher shock tube problem]\label{ex:ShuOsher}
			Next, we consider the shock density wave interaction problem with a moving Mach 3 shock interacting with sine waves in the density:
			\begin{equation}
				\left(\rho,u,p,\gamma\right)^T=
				\begin{cases}
					\left(3.857143,2.629369,10.333333,1.4\right)^T,  & x\in\left[-5,-4\right) \\
					\left(1+0.2\sin(5x),0,1,1.4\right)^T,   & x\in\left[-4,5\right].
				\end{cases}
			\end{equation}
			Given that this solution exhibits both shocks and a complex smooth region structure, a higher-order scheme is expected to showcase its advantages. For reference, the ``exact" solution is computed using the fifth-order finite difference WENO scheme \cite{jiang1996efficient} with 2000 grid points. In Fig. \ref{fig:ShuOsher}, the density is plotted using the DG method with the jump filter and the hybrid limiter associated with it. Both schemes effectively capture high-frequency waves, and the hybrid limiter demonstrates its ability to efficiently reduce numerical dissipation. Additionally, higher-order DG methods with the jump filter or the hybrid limiter outperform their lower-order counterparts.
			\begin{figure}[ht]
				\centering
				\begin{minipage}[b]{0.24\linewidth}
					\centering
					\includegraphics[width=0.9\linewidth]{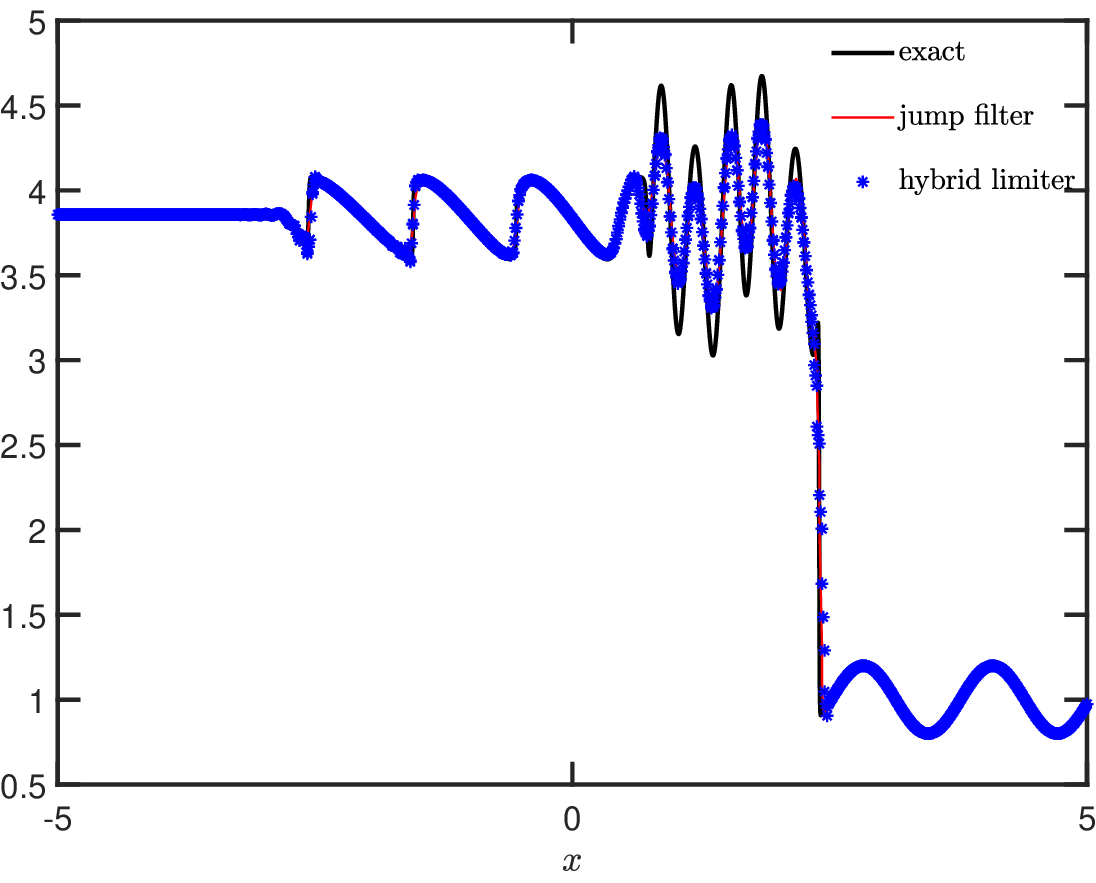}
				\end{minipage}
				\hfill
				\begin{minipage}[b]{0.24\linewidth}
					\centering
					\includegraphics[width=0.9\linewidth]{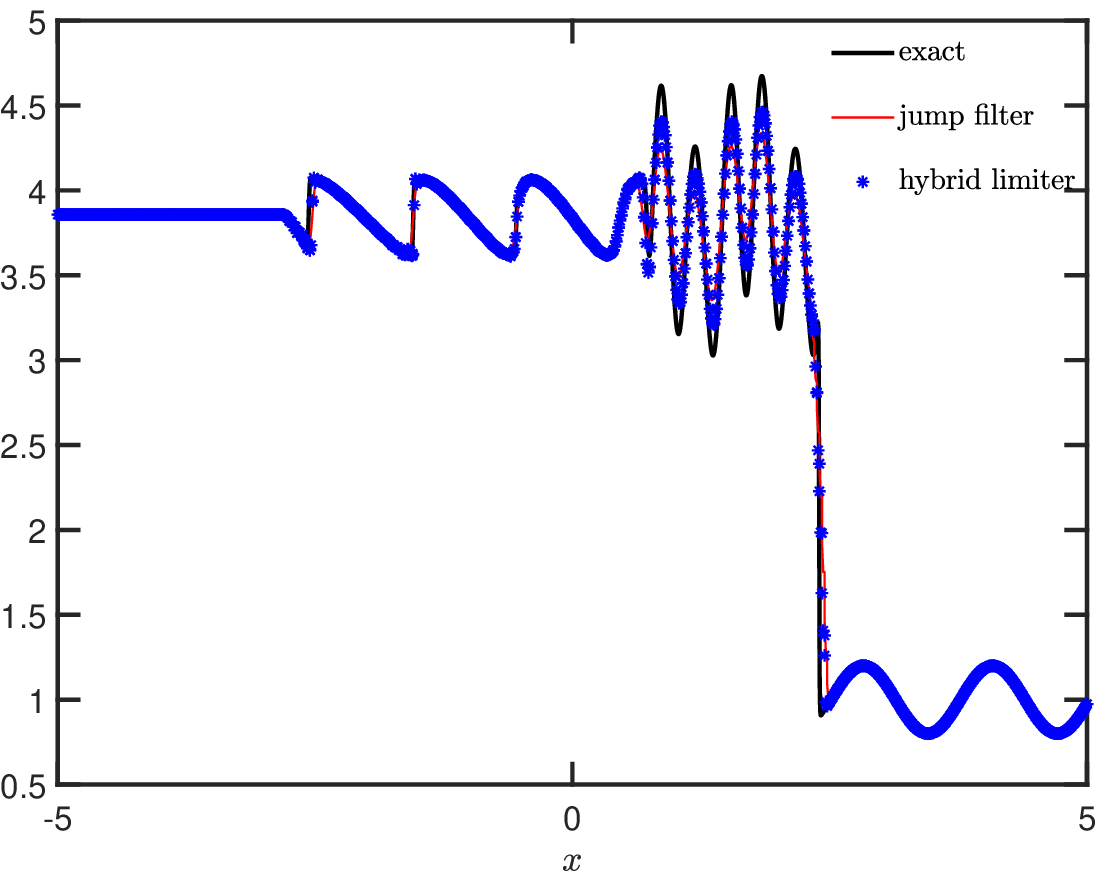}
				\end{minipage}
				\hfill
				\begin{minipage}[b]{0.24\linewidth}
					\centering
					\includegraphics[width=0.9\linewidth]{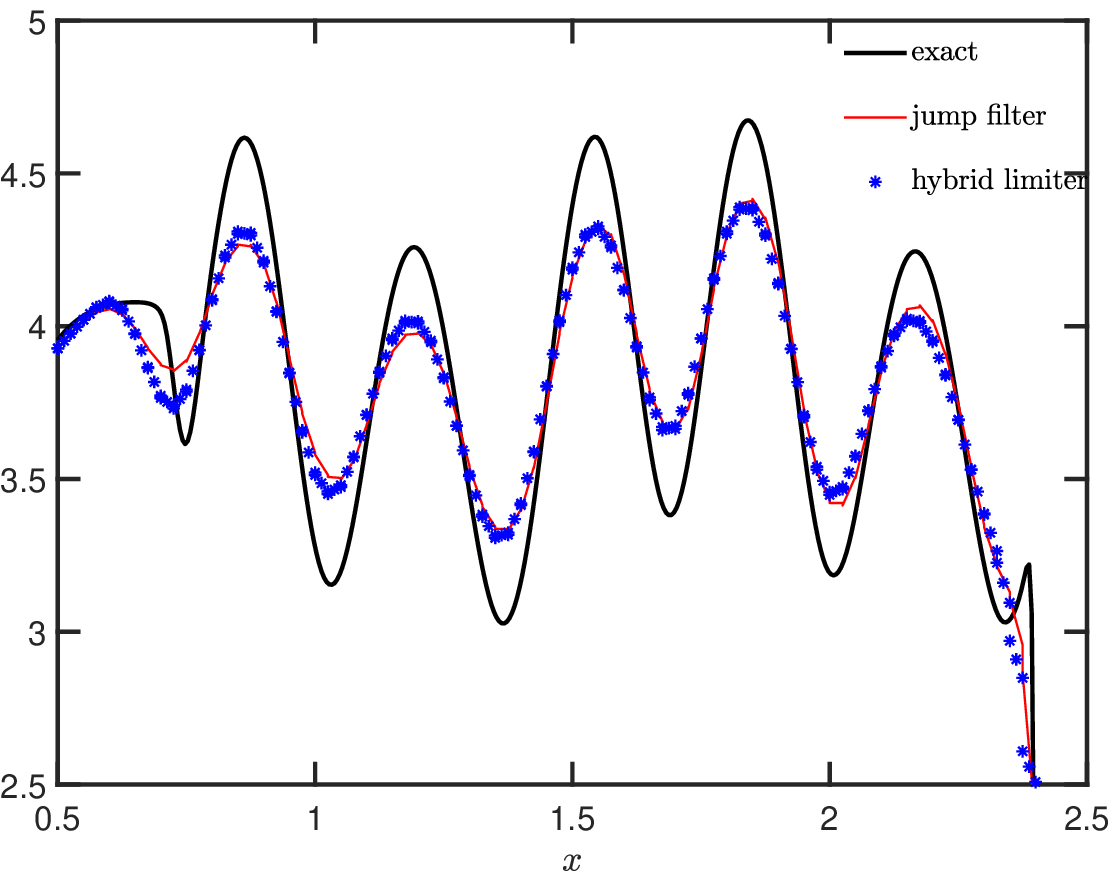}
				\end{minipage}
				\hfill
				\begin{minipage}[b]{0.24\linewidth}
					\centering
					\includegraphics[width=0.9\linewidth]{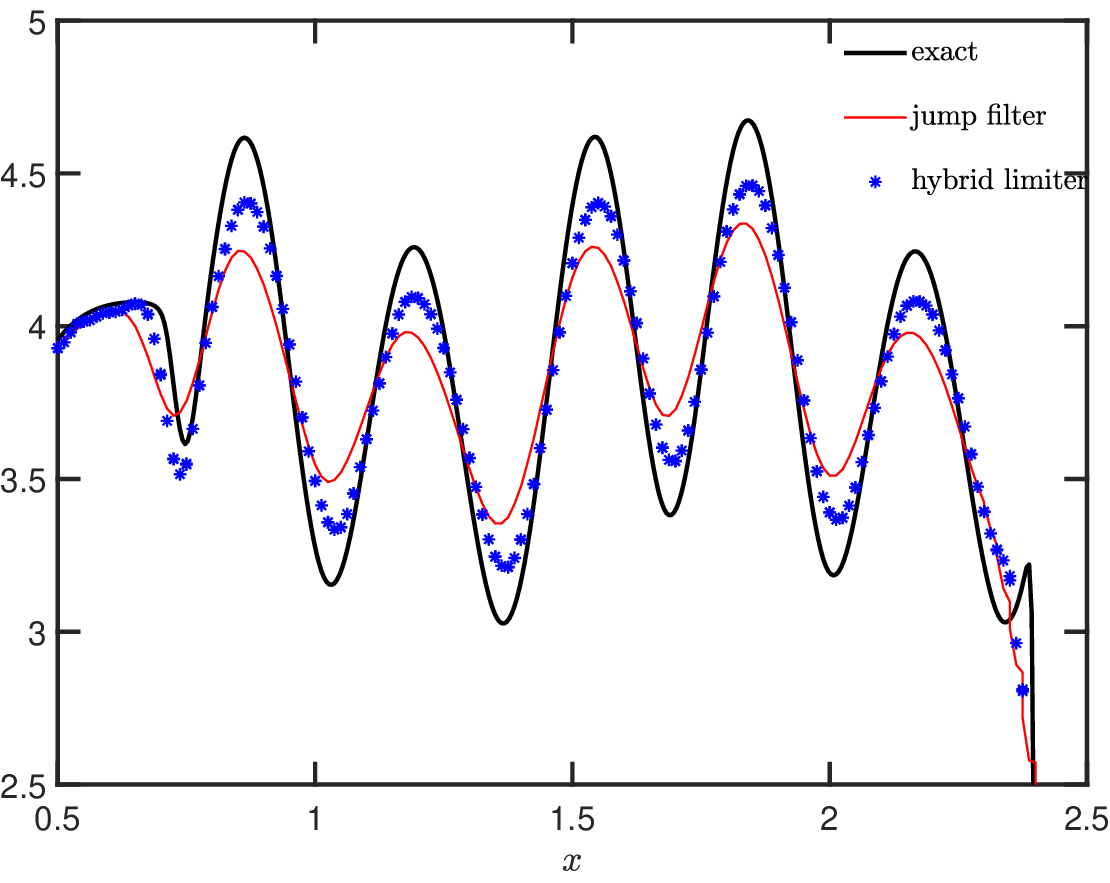}
				\end{minipage}
				\caption{Numerical results of DG scheme with the jump filter and hybrid limiters for Example \ref{ex:ShuOsher}  with $N=400$ at time $T= 1.8$. From left to right: $P^1$; $P^2$; $P^1$ (Zoomed-in); $P^2$ (Zoomed-in).}
				\label{fig:ShuOsher}
			\end{figure}	
		\end{example}
		\subsection{Two dimensional tests}\label{sub:2dexamples}
		The 2D Euler equation can be written as:
		\begin{equation}
			\boldsymbol{U}_t + \nabla \cdot 	\boldsymbol{F}(\boldsymbol{U}) = 0, \mbox{ in } \Omega \subset \mathbb{R}^2,
		\end{equation}	
		with $\boldsymbol{U} = (\rho,\boldsymbol{m},E)^T$, $\boldsymbol{F}(	\boldsymbol{U}) = (\boldsymbol{m}^T , \rho 	\boldsymbol{u} \boldsymbol{u}^T + p	\boldsymbol{I} , (E + p) \	\boldsymbol{u}^T )^T$,
		where $\boldsymbol{m} = \rho \boldsymbol{u}$, $E = \frac{1}{2} \rho \boldsymbol{u}^2 + \rho e$, and $p = (\gamma - 1) \rho e$. Here, $\rho$ represents density, and $\boldsymbol{u}=(u,v)^T$ and $	\boldsymbol{m}$ are velocities and momenta. $E$ is the total energy, $p(\boldsymbol{U})$ is pressure, and $e$ is the internal energy. And $\gamma$ is the ratio of specific heats. Unless otherwise specified, we always take $\gamma=1.4$.  For Euler equations in 2D, the damping term \eqref{para:damp22d} serves to determine the magnitude of the viscosity term with $(c_{f,x}^{p,q},c_{f,y}^{p,q}) = \frac{1}{4\lambda_{p,q}} \frac{1}{\bar{H_j}}( h_{y_j}\beta_{i,j}^x, h_{x_i}\beta_{i,j}^y)$ in \eqref{para:damp22d1}, where $\beta_{i,j}^x$ and $\beta_{i,j}^y$ are the estimates of the local maximum wave speeds in the $x$- and $y$-directions, respectively.
		
		\begin{example}[Vortex evolution problems]\label{ex:2Dvortex}
			We now consider the problem of vortex evolution. An isentropic vortex perturbation centered at $(x_c, y_c)$ is added to the velocity $(u, v)$, temperature ($T = p/\rho$), and entropy ($S = \ln\left(p/\rho^ \gamma\right)$) of the flow, given
			in the following:
			\[
			(\delta u, \delta v) = \frac{\epsilon}{r_c}e^{\alpha(1-r^2)}(\overline{y}, -\overline{x}), \quad \delta T = -\frac{( \gamma - 1)\epsilon^2}{4\alpha\gamma}e^{2\alpha(1-\tau^2)}, \quad \delta S = 0  ,
			\]
			where $(\overline{x}, \overline{y}) = (x - x_c, y - y_c)$, $r^2 = \overline{x}^2 + \overline{y}^2$, and the $\tau=r/r_c$. Here $\epsilon$ is the strength of the vortex, $\alpha$ controls the decay rate of the vortex, and $r_c$ is the critical radius at which the vortex has the maximum strength. We test the following two cases:
			\begin{itemize}
				\item[(a)] (accuracy test) We consider an isentropic
				vortex test case with an exact solution given by $ u = 1 - \beta e^{(1 - r^2)} \frac{y - y_0}{2\pi}$, $v = \beta e^{(1 - r^2)} \frac{x - x_0}{2\pi}$, $\rho  = \left( 1 - \left( \frac{\gamma - 1}{16\gamma\pi^2} \right) \beta^2 e^{2(1 - r^2)} \right)^{\frac{1}{\gamma - 1}}$, $p = \rho^\gamma$, where $r  = \sqrt{(x - T - x_0)^2 + (y - y_0)^2}$,
				$x_0  = 5$, $y_0 = 0$, $\beta = 5$, $\gamma = 1.4$. We compute the numerical solution at time $T = 10$ in the region $[0, 10]\times[-5,5]$ under periodic boundary conditions. Table \ref{tab:2Dvortex} presents the accuracy and convergence rates. Note that the numerical results were obtained using the jump filter without employing the hybrid procedure. It is observed that the DG scheme with the jump filter achieves the expected convergence rate, surpassing $k+1$. This deviation is attributed to the dominance of the viscous term on coarse meshes. However, with mesh refinement, the convergence rate rapidly approaches the theoretically expected $k+1$ order of accuracy. Moreover, comparison with the standard DG scheme reveals negligible differences in errors on fine meshes, suggesting a minimal impact of our jump filter as expected. As the mesh is refined, the jumps between adjacent cells diminish significantly.	
				\item[(b)] (shock-vortex interaction) The model problem characterizes the mutual interference effects of a stationary shock wave and a moving vortex. The computational domain is given by $[0, 2] \times [0, 1]$. A stationary shock wave is located at $x = 0.5$ and is perpendicular to the $x$-axis. The left state of the shock wave is $(\rho, u, v, p)^T = (1, 1.1\sqrt{\gamma}, 0, 1)^T$. A small vortex is imposed on the fluid to the left of the shock wave, with its center at $(x_c, y_c) = (0.25, 0.5)$. We take the same values of these parameters as in \cite{jiang1996efficient} that $\epsilon = 0.3$, $r_c = 0.05$, and $\alpha = 0.204$. The upper and lower boundaries are subjected to reflective boundary conditions. Fig. \ref{fig:shockvortexp3} displays contour plots of pressure at six distinct times, illustrating the shock-vortex interaction dynamics. Employing DG schemes with the jump filter effectively captures these dynamics and accurately depicts intricate wave structures. The obtained results are comparable to those reported in \cite{jiang1996efficient}.
			\end{itemize}

			\begin{table}[htb]\small
				\caption{\label{tab:2Dvortex} Numerical errors and orders
					in Example \ref{ex:2Dvortex}. } \centering
				\resizebox{1.\textwidth}{!}{
					\medskip
					\begin{tabular}{|c|r||cc|cc|cc|cc|}  \hline
						&  &\multicolumn{4}{c|}{results with filter} & \multicolumn{4}{c|}{results without filter} \\ \hline
						& $N_x\times N_y$ & $L^2$ error & order & $L^\infty$ error & order & $L^2$ error & order & $L^\infty$ error & order \\ \hline \hline
						\multirow{6}{0.6cm}{$P^1$}
						&  $ 80\times 80$     & 6.26e-03    &   --    &6.79e-02    &   --    &  1.30e-03    &   --& 1.60e-02    &   --    \\
						&   $ 160\times 160$    & 5.52e-04    &  3.50  &4.97e-03    &3.77    &1.93e-04 &2.75 &2.25e-03 &2.83  \\
						&  $ 320\times 320$     &4.70e-05  &3.55 &4.90e-04 &3.33    &2.84e-05 &2.76 &3.74e-04 &2.58 \\
						&  $ 640\times 640$    &5.70e-06 &3.04 &7.48e-05 &2.71   &4.96e-06 &2.52 &8.37e-05 &2.16  \\
						&  $ 1280\times 1280$   &1.08e-06 &2.39 &1.85e-05 &2.01    &1.06e-06 &2.22 &1.97e-05 &2.08 \\\hline
						\multirow{6}{0.6cm}{$P^2$}
						&  $ 80\times 80$    &1.23e-03    &   --    & 1.32e-02    &   --    &2.06e-05    &   -- & 5.76e-04    &   -- \\
						&  $ 160\times 160$  &3.13e-05 &5.29 &3.55e-04 &5.21
						
						&2.77e-06 &2.89 &7.24e-05 &2.99\\
						
						&   $ 320\times 320$   & 1.01e-06 &4.94 & 1.49e-05 &4.57  &
						
						3.88e-07    &  2.83  &  9.64e-06    &   2.90  \\
						
						&  $ 640\times 640$    &6.21e-08 &4.03 &1.21e-06 &3.61
						
						&5.36e-08 &2.85 &1.33e-06 &2.85 \\
						
						& $ 1280\times 1280$   &7.46e-09 &3.05 &1.92e-07 &2.70
						&7.37e-09 &2.86 &1.97e-07 &2.75 \\ \hline
						\multirow{6}{0.6cm}{$P^3$}
						
						&  $ 80\times 80$  &  1.65e-03   &  --   & 1.83e-02    &   --    &
						
						6.47e-07    &   --&  1.38e-05    &   -- \\
						
						&  $ 160\times 160$  &1.72e-06 &9.90 &3.48e-05 &9.04
						
						&3.33e-08 &4.28 &9.25e-07 &3.90 \\

						&  $ 320\times 320$   &2.13e-08 &6.33 &3.42e-07 &6.66
						
						&1.86e-09 &4.16 &5.45e-08 &4.08\\

						&  $ 640\times 640$    &3.49e-10 &5.93 &7.18e-09 &5.57
						
						&1.10e-10 &4.08 &3.18e-09 &4.10\\

						&  $ 1280\times 1280$   & 8.66e-12 &5.33 &2.37e-10 &4.92
						
						&6.82e-12 &4.01  &1.85e-10 &4.10  \\ \hline
				\end{tabular}}
			\end{table}
		\end{example}
		\begin{figure}[ht]
			\centering
			\begin{minipage}[b]{0.3\linewidth}
				\centering
				\includegraphics[width=\linewidth]{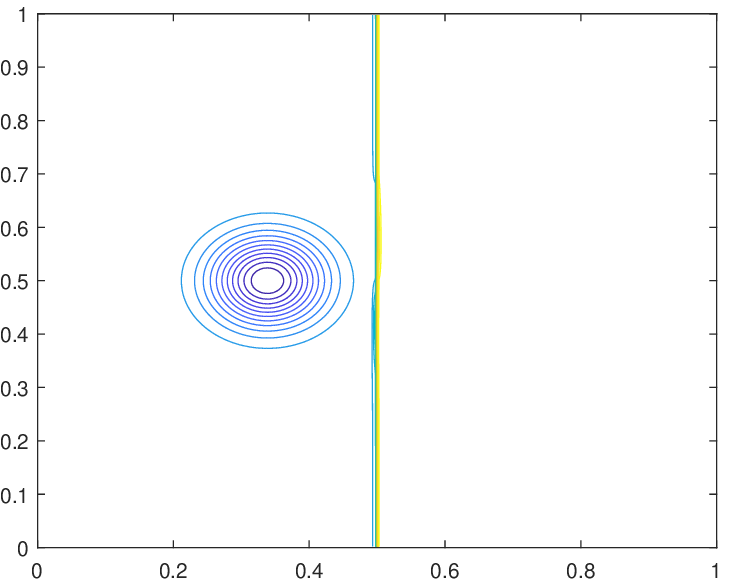}
			\end{minipage}
			\begin{minipage}[b]{0.3\linewidth}
				\centering
				\includegraphics[width=\linewidth]{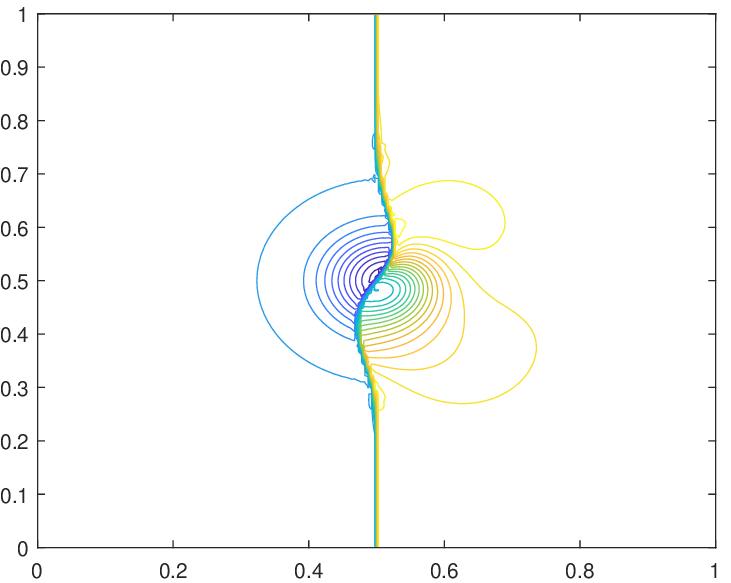}
			\end{minipage}
			\begin{minipage}[b]{0.3\linewidth}
				\centering
				\includegraphics[width=\linewidth]{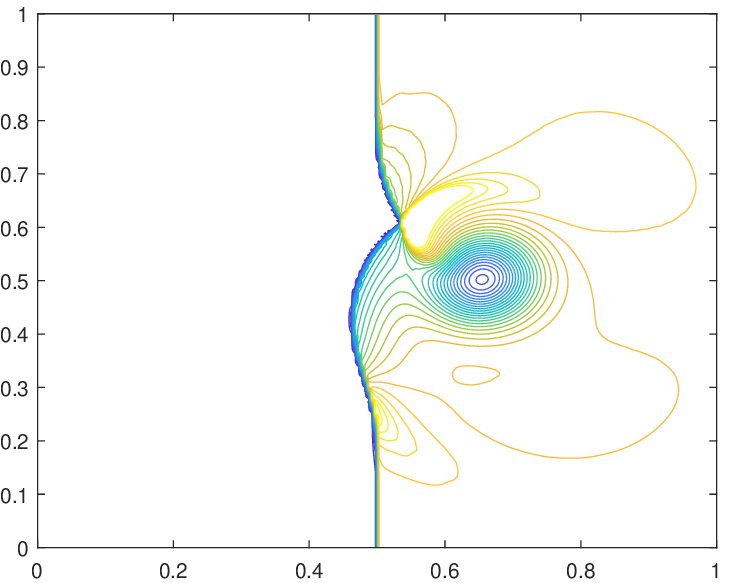}
			\end{minipage}
			\begin{minipage}[b]{0.3\linewidth}
				\centering
				\includegraphics[width=\linewidth]{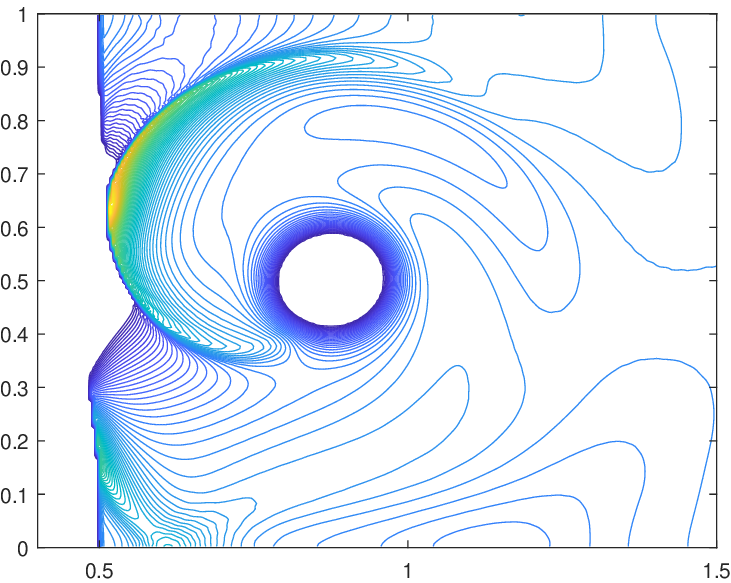}
			\end{minipage}
			\begin{minipage}[b]{0.3\linewidth}
				\centering
				\includegraphics[width=\linewidth]{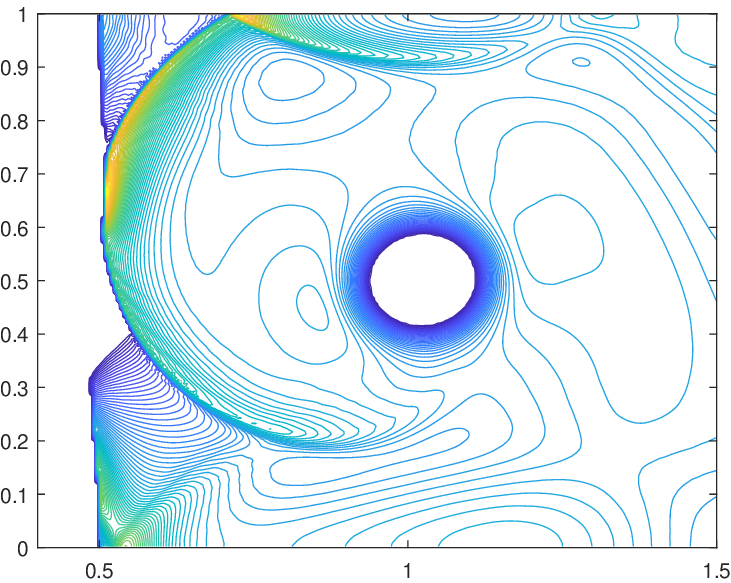}
			\end{minipage}
			\begin{minipage}[b]{0.3\linewidth}
				\centering
				\includegraphics[width=\linewidth]{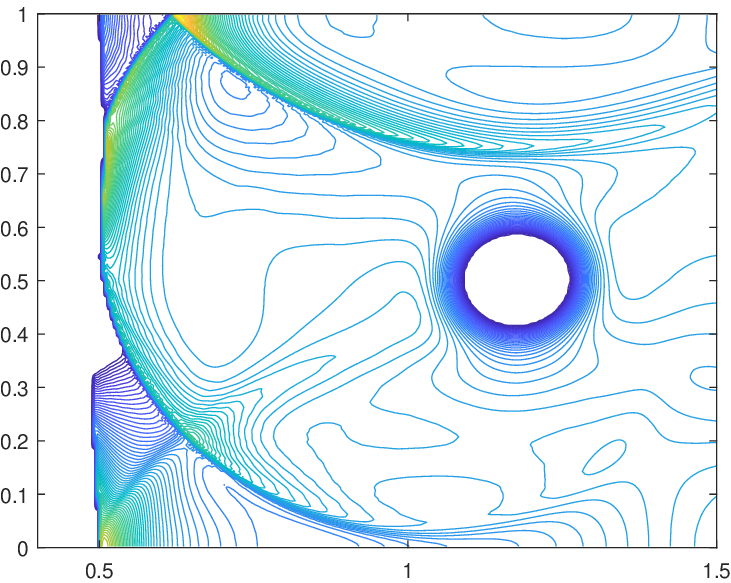}
			\end{minipage}
			\caption{Numerical results of DG scheme with the jump filter for Example \ref{ex:2Dvortex} case (b). From left to right, and top to bottom: $T = 0.068$; $T = 0.203$; $T = 0.330$; $T = 0.529$; $T = 0.662$; $T = 0.8$.  30 equally spaced pressure contours from 0.68 to 1.30 at $T = 0.068$ and $T = 0.203$; $T = 0.330$. 90 equally spaced pressure contours from 1.19 to 1.37 at $T = 0.529$; $T = 0.662$; and $T = 0.8$. $P^3$ basis. $400\times 200$ cells.}
			\label{fig:shockvortexp3}
		\end{figure}

		\begin{example}[Double Mach reflection problem]\label{ex:Double_rec}
			This is a benchmark test for the two-dimensional Euler equations for high-resolution schemes. The computation domain is $\left[0,4\right]$ $\times$ $\left[0,1\right]$. A right-moving Mach $10$ shock is located at $x=1/6$, $y=0$, making a $60^{\circ}$ angle with the horizontal wall. At the top boundary, the flow values are set to describe the exact motion of the Mach 10 shock. The boundary conditions at the left and the right are inflow and outflow respectively.

			The contour plots of density in DG solutions with the jump filter and the hybrid limiter at time $T=0.2$ are presented in Fig. \ref{fig:rec_double_p1_p4}, with zoomed-in views in Fig. \ref{fig:rec_double_local}. Firstly, it's evident that the DG scheme with the jump filter effectively captures shock waves, with resolution improving as accuracy increases. These results mirror the effect of TVB limiters in achieving total variation boundedness in standard DG schemes \cite{cockburn1998runge}, as further explored in \cite{MR4673601}.
			
			On the other hand, when utilizing the hybrid limiter associated with the jump filter as lower-order limiters, the scheme adeptly captures shock waves and showcases the advantages of higher-order schemes more comprehensively. Higher-order schemes exhibit superior resolution, particularly below the Mach stem, as illustrated in Fig. \ref{fig:rec_double_local}. This underscores that employing the jump filter as lower-order polynomials in the hybrid limiter not only facilitates efficient capture of shock waves but also minimizes numerical dissipation.
			
			\begin{figure}[htbp]
				\centering
				\begin{minipage}{0.49\linewidth}
					\centering
					\includegraphics[width=1\linewidth]{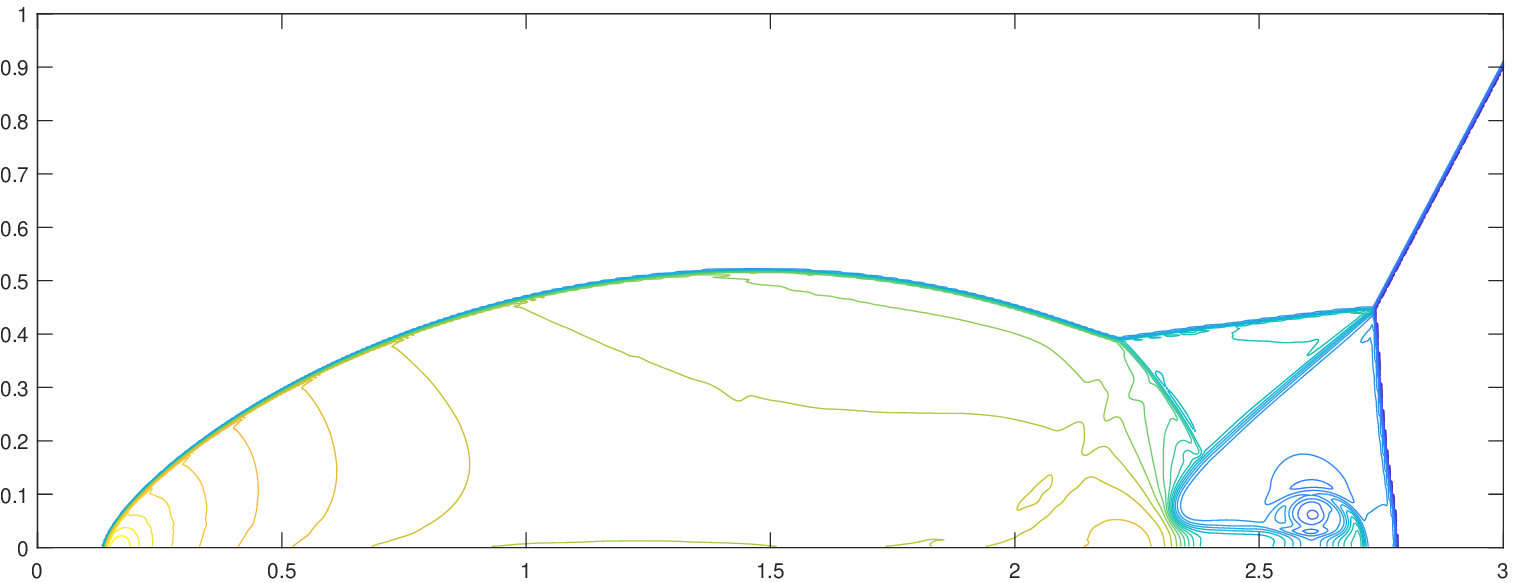}
				\end{minipage}
				\begin{minipage}{0.49\linewidth}
					\centering
					\includegraphics[width=1\linewidth]{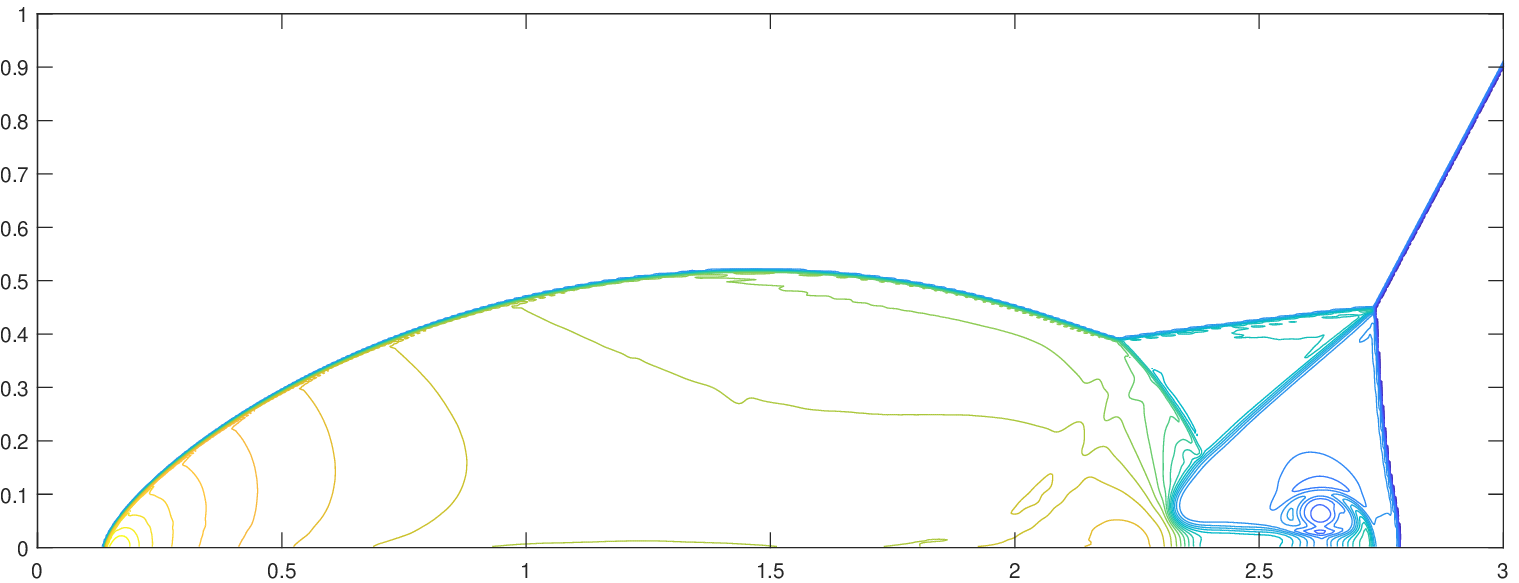}
				\end{minipage}

				\begin{minipage}{0.49\linewidth}
					\centering
					\includegraphics[width=1\linewidth]{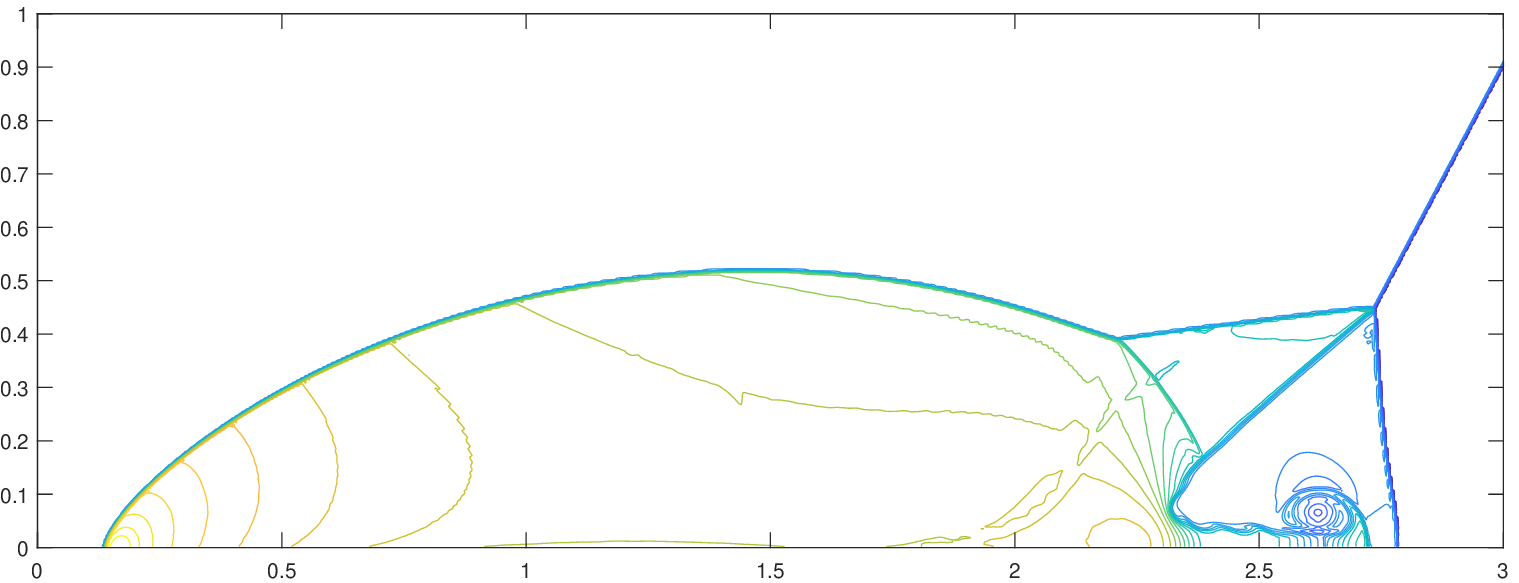}
				\end{minipage}
				\begin{minipage}{0.49\linewidth}
					\centering
					\includegraphics[width=1\linewidth]{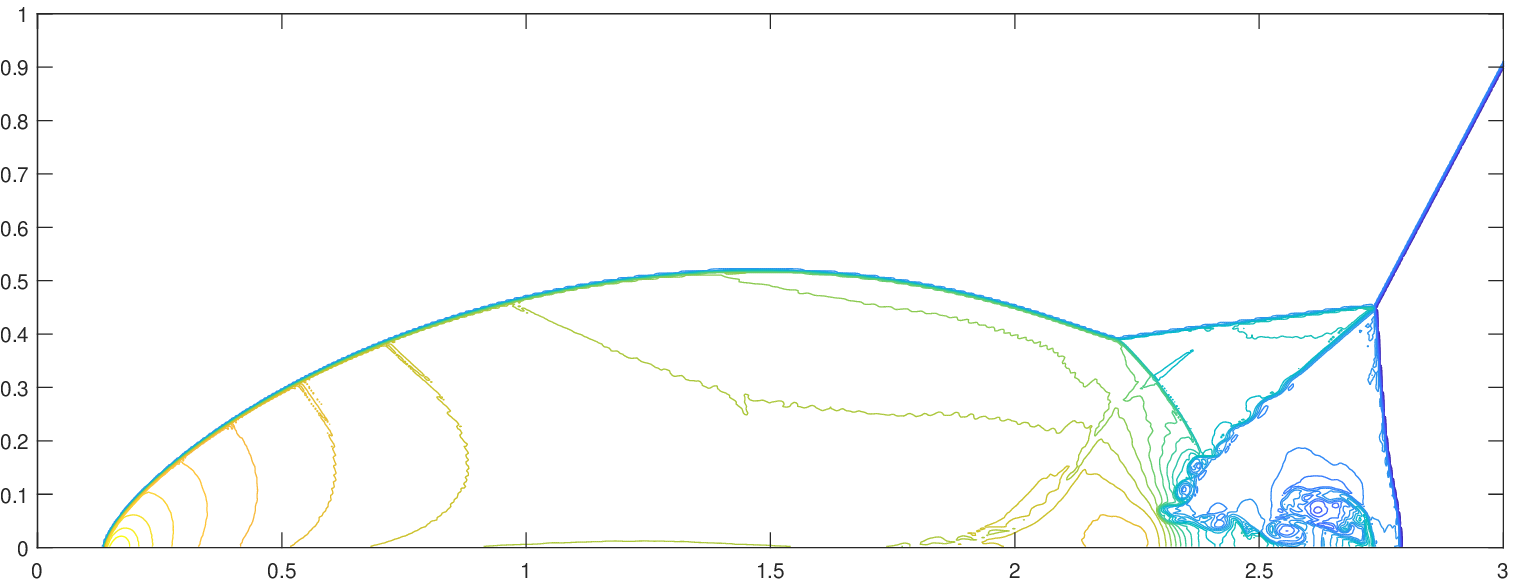}
				\end{minipage}

				\begin{minipage}{0.49\linewidth}
					\centering
					\includegraphics[width=1\linewidth]{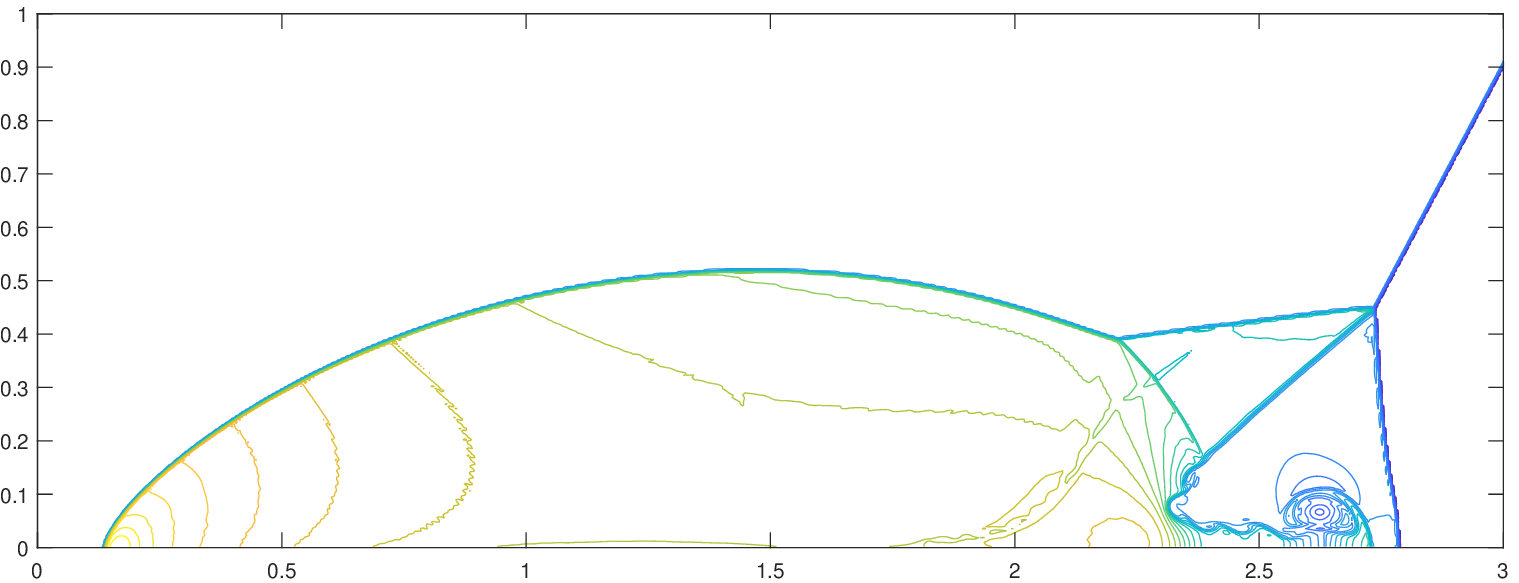}
				\end{minipage}
				\begin{minipage}{0.49\linewidth}
					\centering
					\includegraphics[width=1\linewidth]{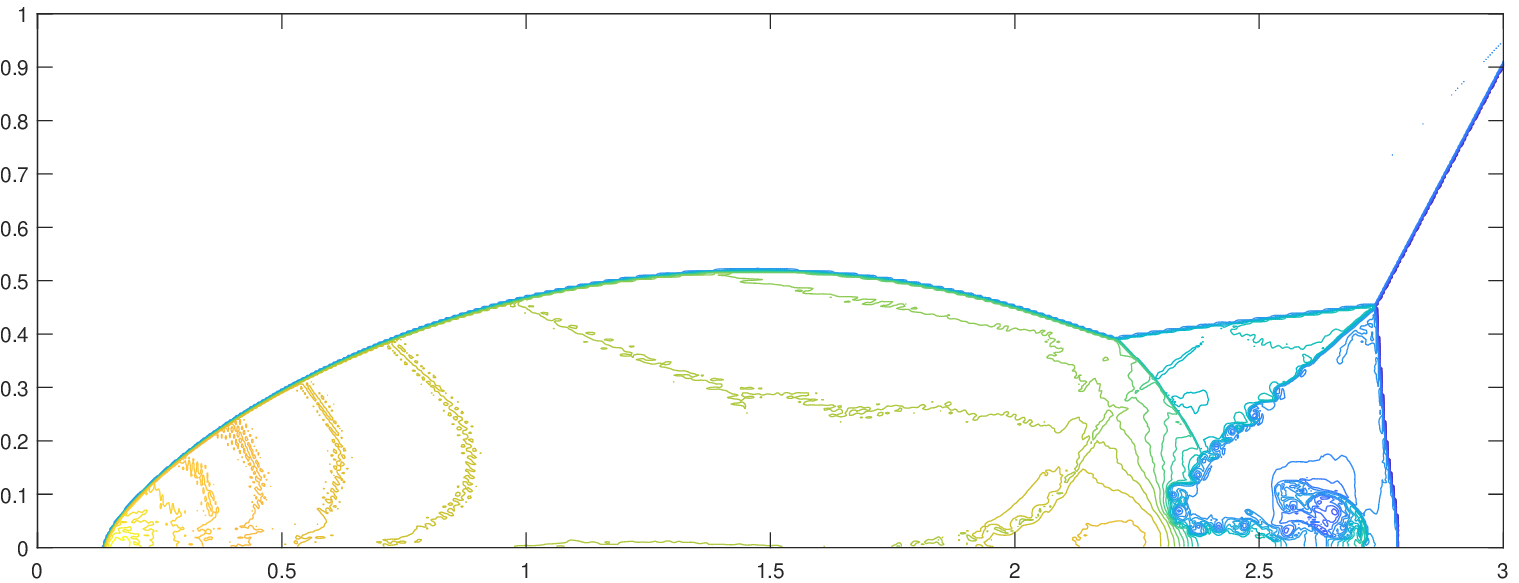}
				\end{minipage}

				\begin{minipage}{0.49\linewidth}
					\centering
					\includegraphics[width=1\linewidth]{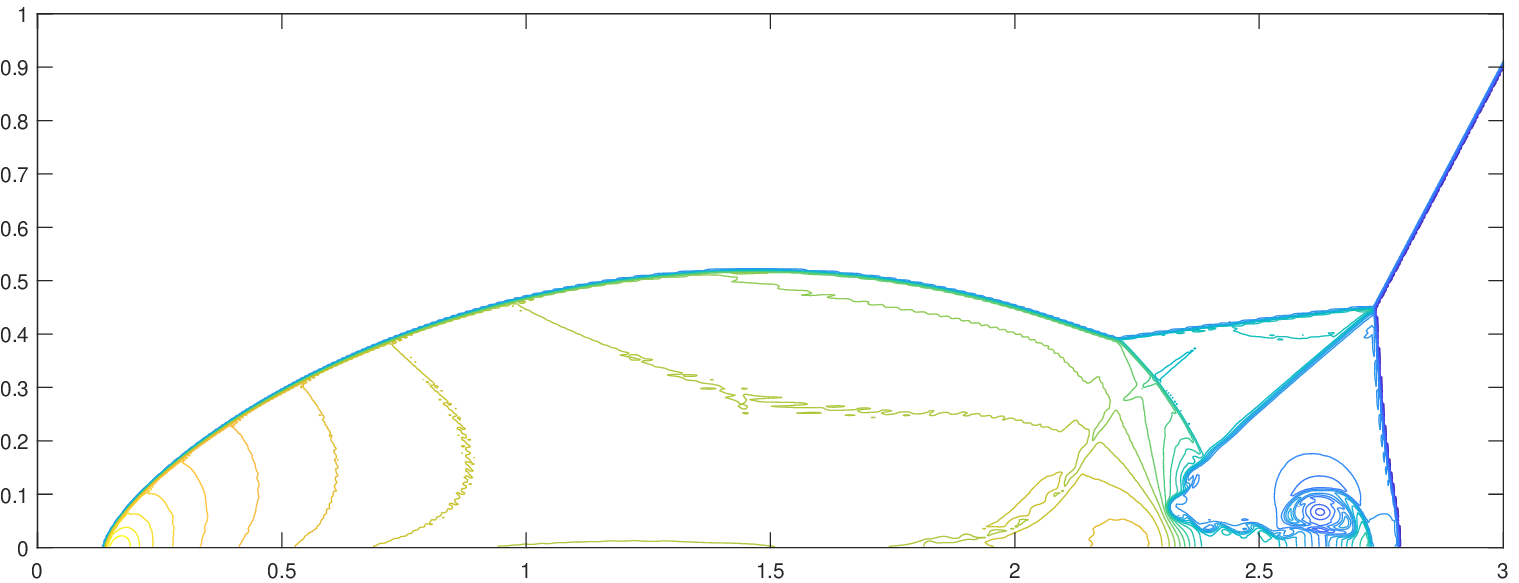}
				\end{minipage}
				\begin{minipage}{0.49\linewidth}
					\centering
					\includegraphics[width=1\linewidth]{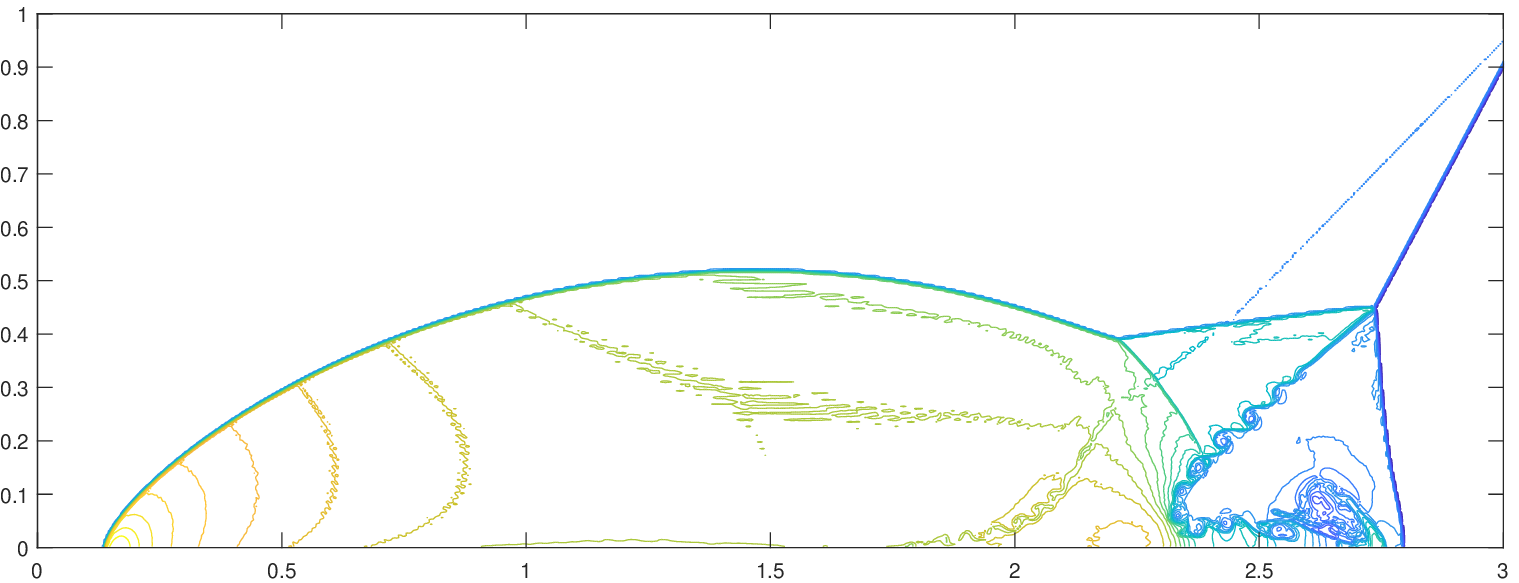}
				\end{minipage}
				\caption{Numerical results for Example \ref{ex:Double_rec}. DG solutions of $k=1,2,3,4$ from top to bottom. Left: the jump filter; Right: the hybrid limiter associated with the jump filter. 30 equally spaced density contours from $1.5$ to $21.5$. $960 \times 240$ cells.}
				\label{fig:rec_double_p1_p4}
			\end{figure}

			\begin{figure}[htbp]
				\centering
				\begin{minipage}{0.24\linewidth}
					\centering
					\includegraphics[width=1\linewidth]{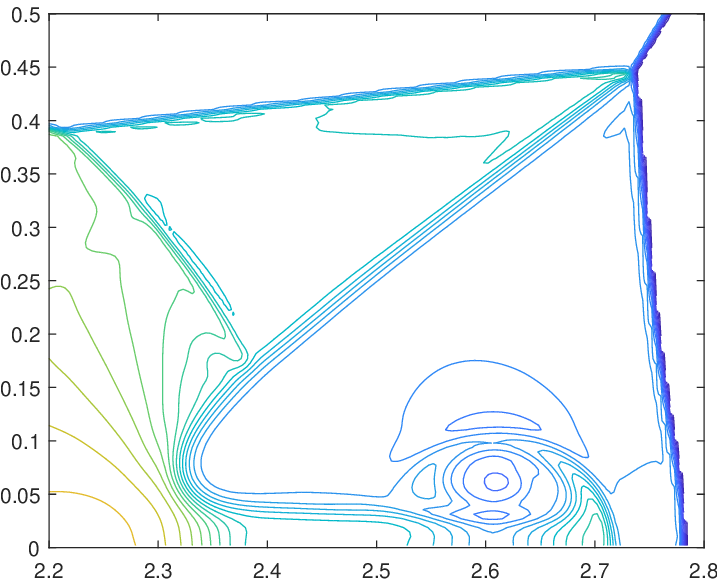}
				\end{minipage}
				\hfill
				\begin{minipage}{0.24\linewidth}
					\centering
					\includegraphics[width=1\linewidth]{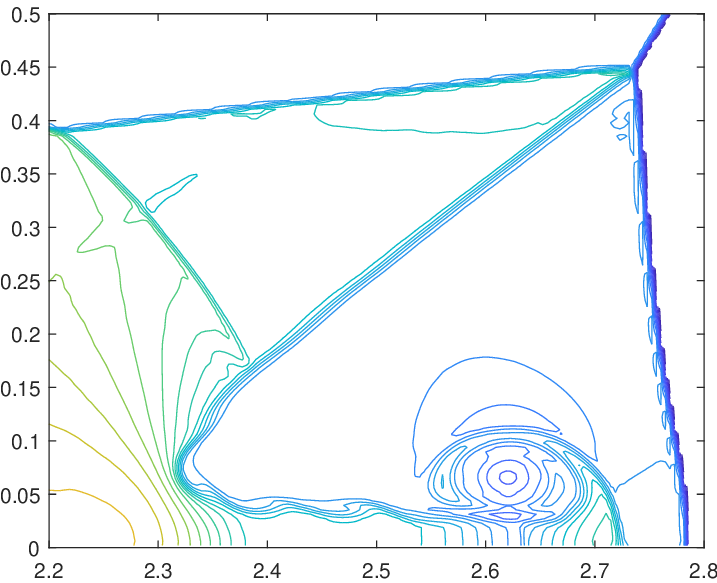}
				\end{minipage}
				\hfill
				\begin{minipage}{0.24\linewidth}
					\centering
					\includegraphics[width=1\linewidth]{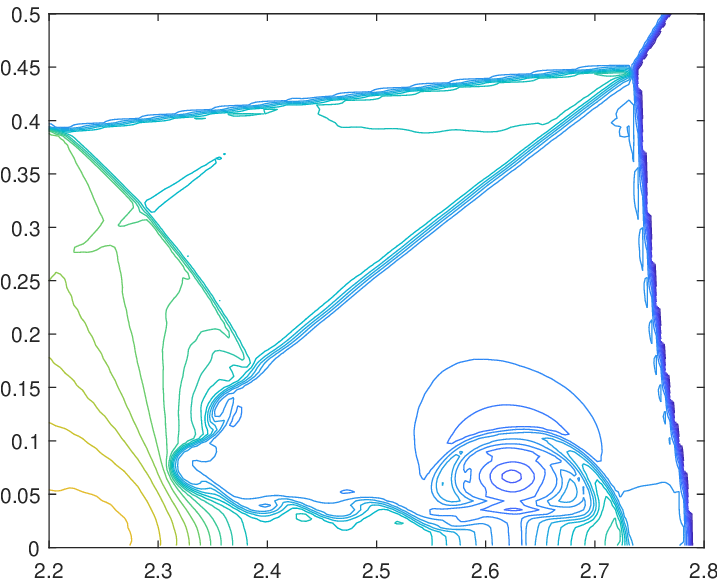}
				\end{minipage}
				\hfill
				\begin{minipage}{0.24\linewidth}
					\centering
					\includegraphics[width=1\linewidth]{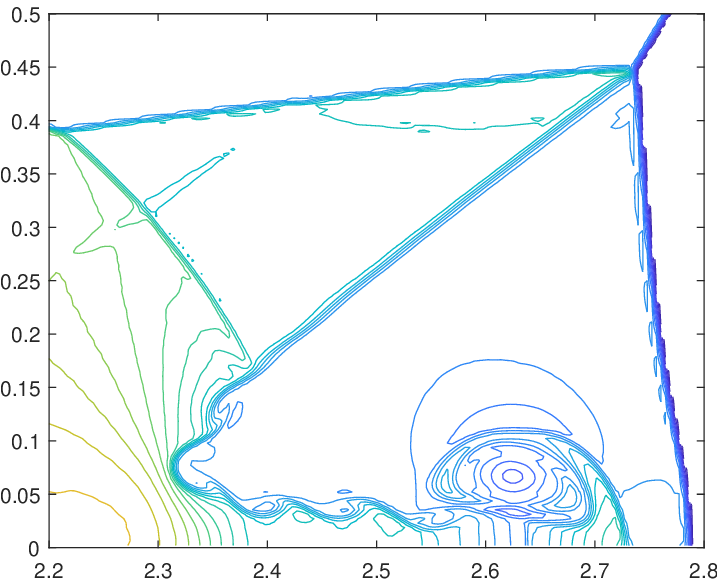}
				\end{minipage}
				\hfill
				\begin{minipage}{0.24\linewidth}
					\centering
					\includegraphics[width=1\linewidth]{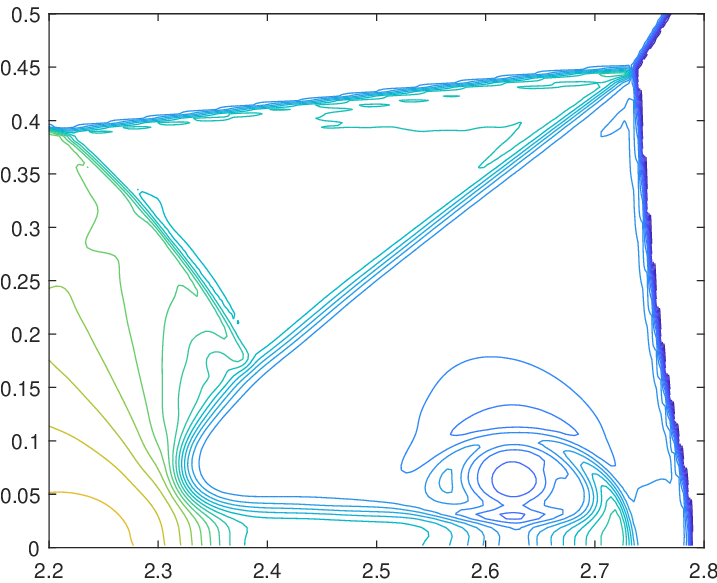}
				\end{minipage}
				\hfill
				\begin{minipage}{0.24\linewidth}
					\centering
					\includegraphics[width=1\linewidth]{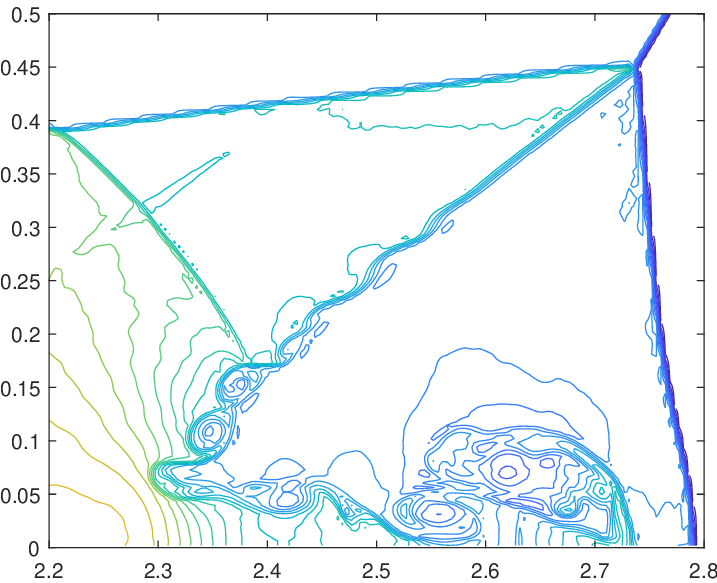}
				\end{minipage}
				\hfill
				\begin{minipage}{0.24\linewidth}
					\centering
					\includegraphics[width=1\linewidth]{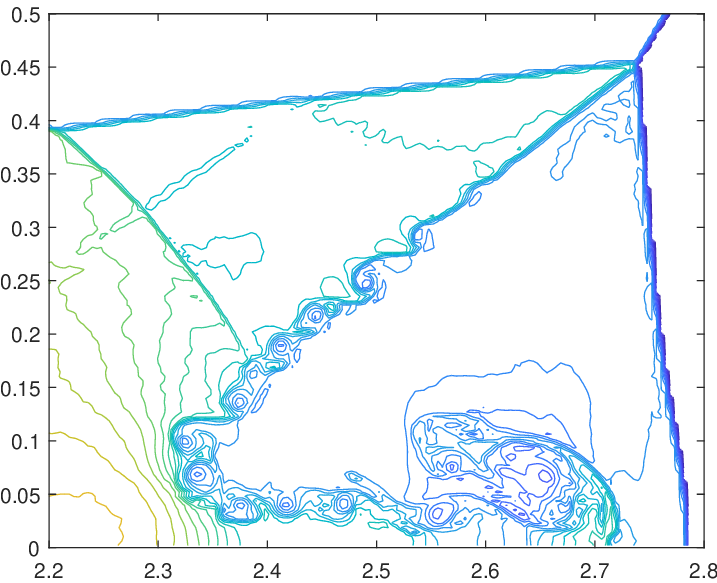}
				\end{minipage}
				\hfill
				\begin{minipage}{0.24\linewidth}
					\centering
					\includegraphics[width=1\linewidth]{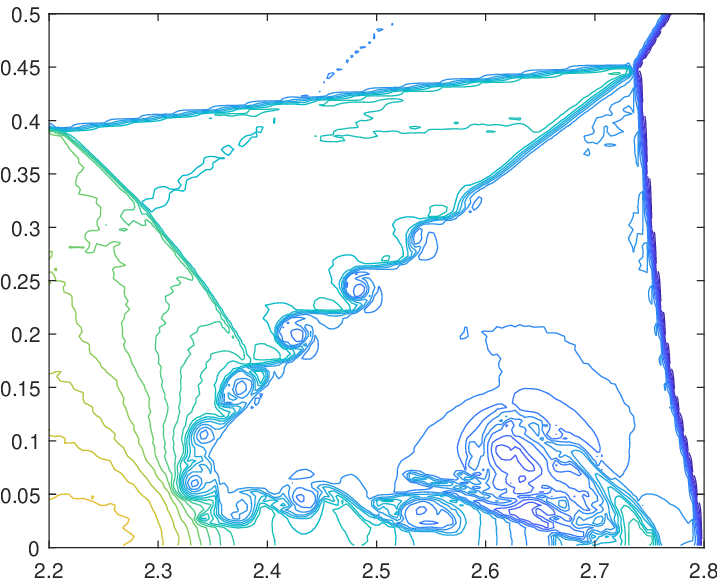}
				\end{minipage}
				\caption{Numerical results for Example \ref{ex:Double_rec}. Zoom-in DG solutions of $k=1,2,3,4$ from left to right. Top: the jump filter; Bottom: the hybrid limiter associated with the jump filter. 30 equally spaced density contours from $1.5$ to $21.5$.  $960 \times 240$ cells.}
				\label{fig:rec_double_local}
			\end{figure}	
		\end{example}

		\begin{example}[Forward facing step problem] \label{ex:forward_rec}
			The wind tunnel is $3$ units in length and $1$ unit in width, and the step is $0.2$ units high, located $0.6$ units from the left end of the tunnel.
			The initial condition is a uniform flow given by $(\rho,u,v,p)^T=(\gamma,3,0,1)^T$. Reflective boundary conditions are applied along the wall boundaries, and inflow and outflow boundary conditions are imposed along the left and right boundaries.
			
			The DG solutions with the jump filter and the associated hybrid limiter are depicted in Fig. \ref{fig:rec_forward_p1_p4} at time $T=4$. Notably, all schemes exhibit essential oscillation-free behavior. Additionally, a discernible difference emerges between higher and lower-order schemes: the former more effectively captures slip lines, showcasing the superiority of higher-order schemes.
			Moreover, employing the hybrid approach yields even lower dissipation. This can be attributed to its ability to effectively capture the intricate flow structure and resolve slip lines better than the jump filter alone.	
			\begin{figure}[htbp]
				\centering
				\begin{minipage}{0.49\linewidth}
					\centering
					\includegraphics[width=1\linewidth]{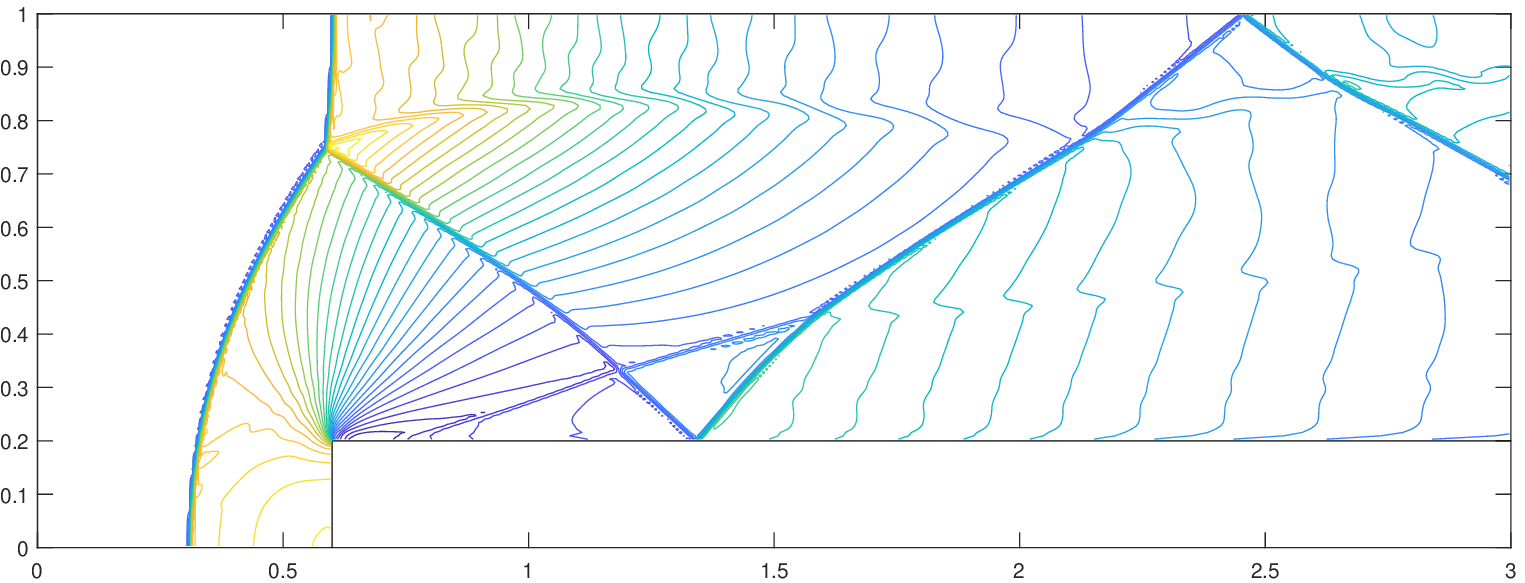}
				\end{minipage}
				\begin{minipage}{0.49\linewidth}
					\centering
					\includegraphics[width=1\linewidth]{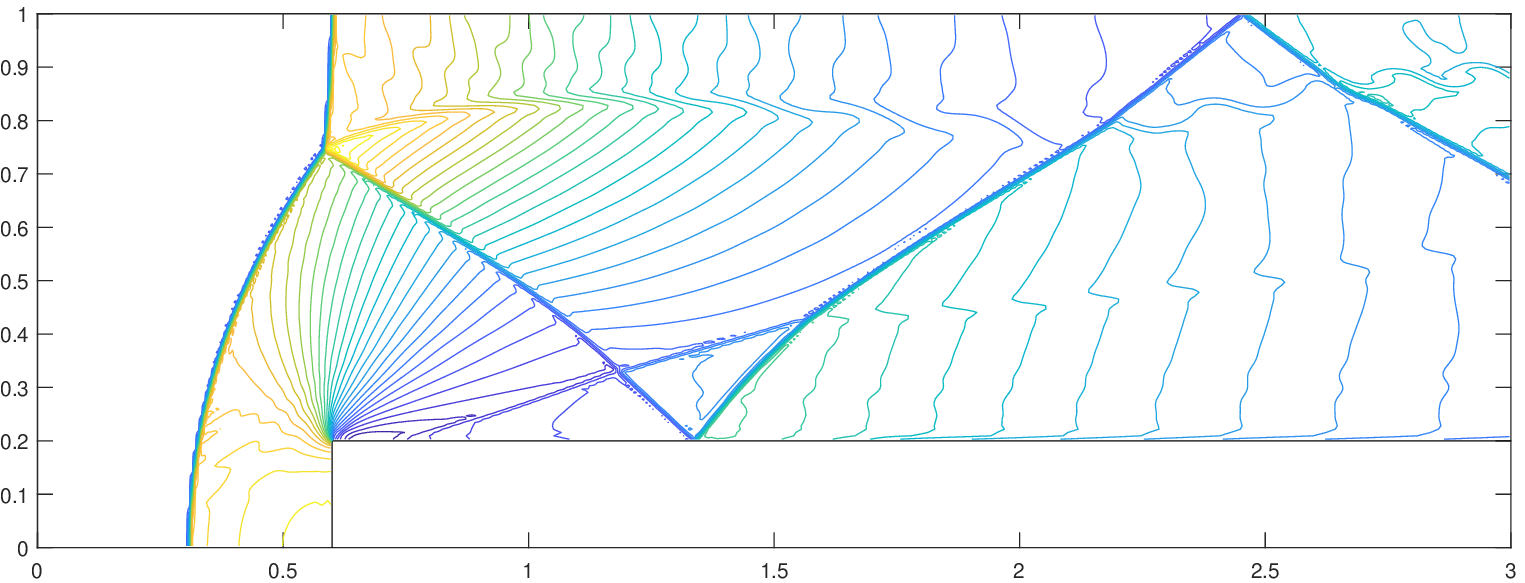}
				\end{minipage}

				\begin{minipage}{0.49\linewidth}
					\centering
					\includegraphics[width=1\linewidth]{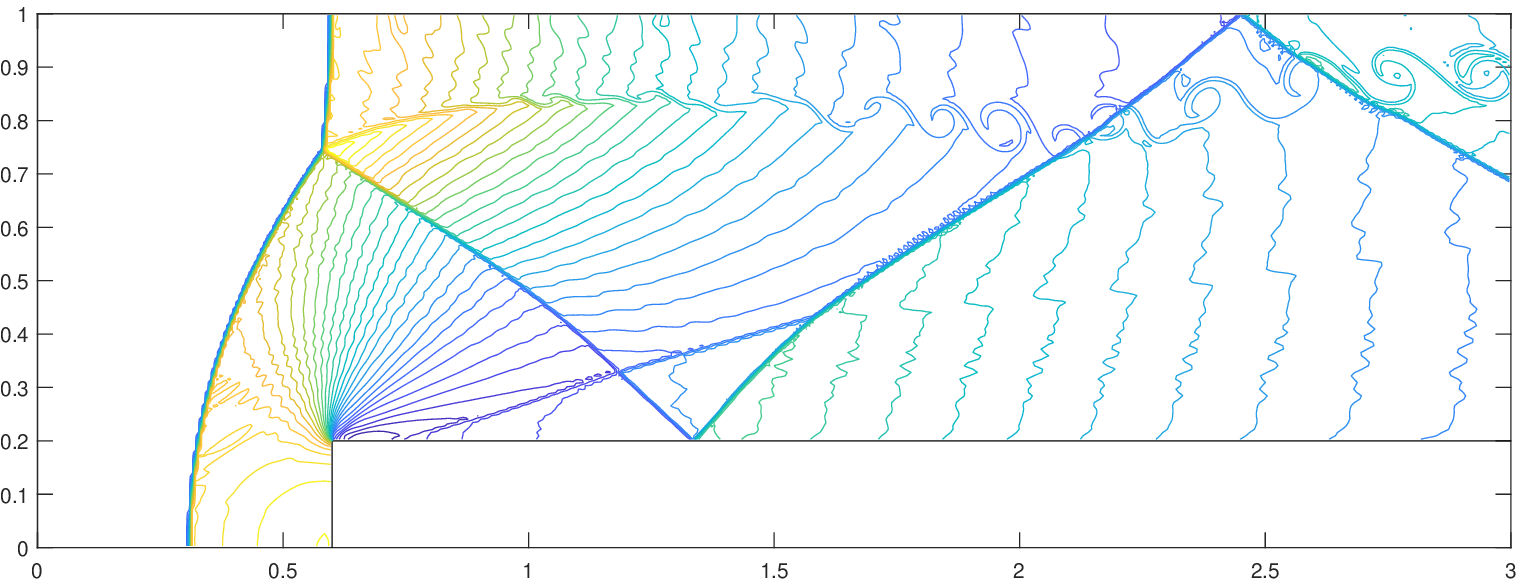}
				\end{minipage}
				\begin{minipage}{0.49\linewidth}
					\centering
					\includegraphics[width=1\linewidth]{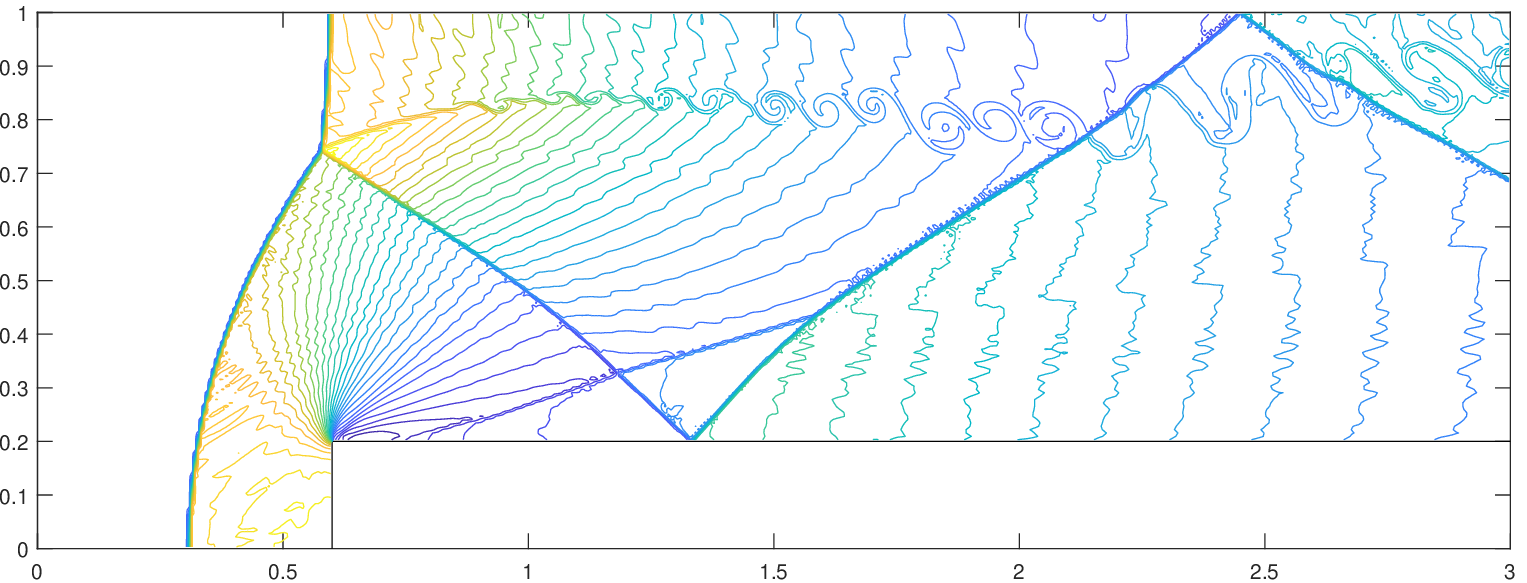}
				\end{minipage}

				\begin{minipage}{0.49\linewidth}
					\centering
					\includegraphics[width=1\linewidth]{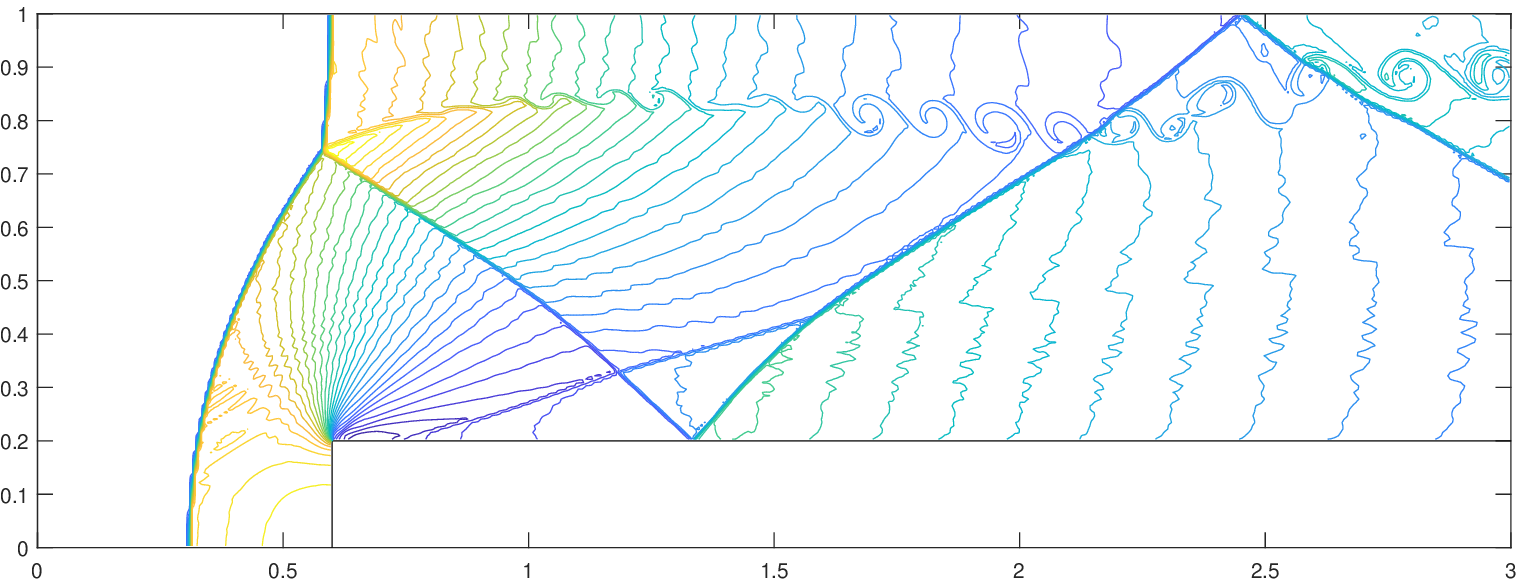}
				\end{minipage}
				\begin{minipage}{0.49\linewidth}
					\centering
					\includegraphics[width=1\linewidth]{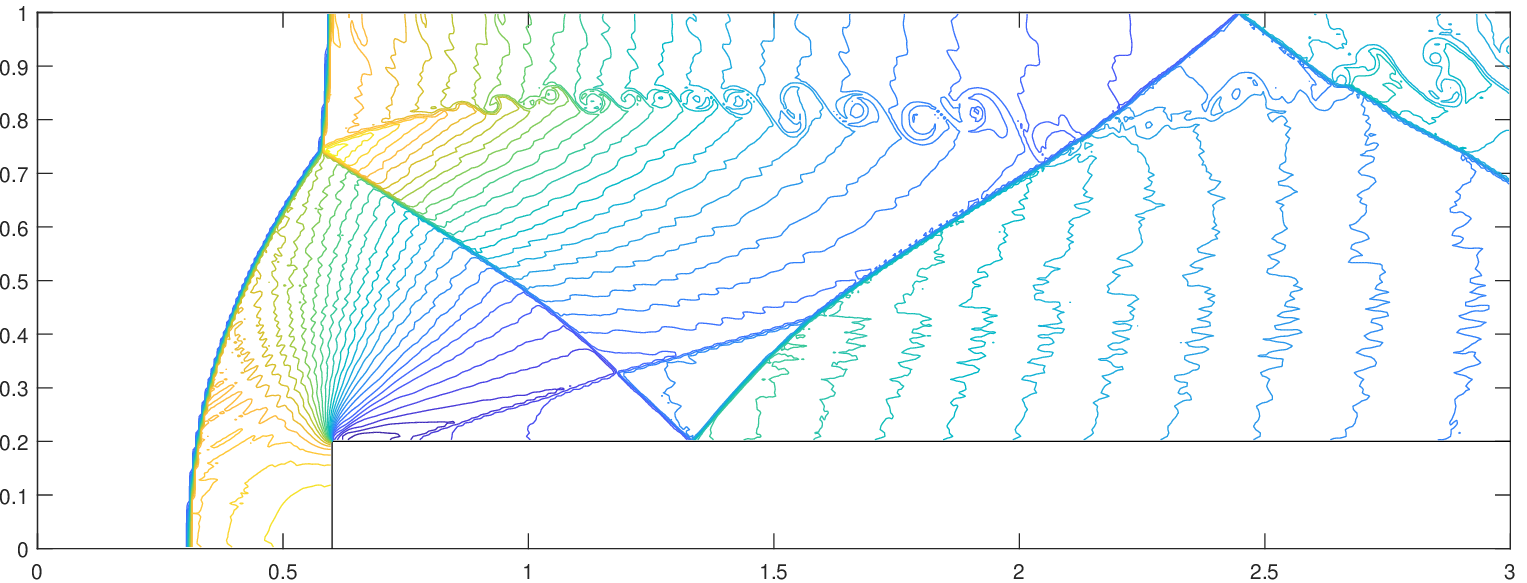}
				\end{minipage}

				\begin{minipage}{0.49\linewidth}
					\centering
					\includegraphics[width=1\linewidth]{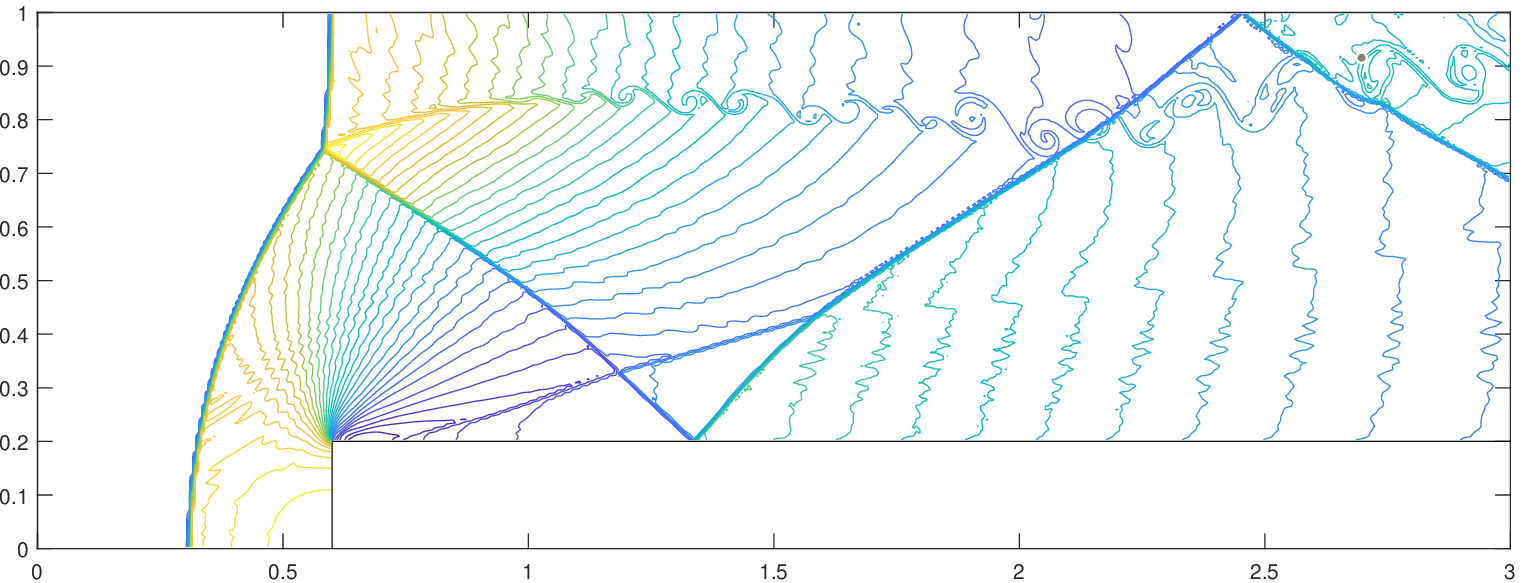}
				\end{minipage}
				\begin{minipage}{0.49\linewidth}
					\centering
					\includegraphics[width=1\linewidth]{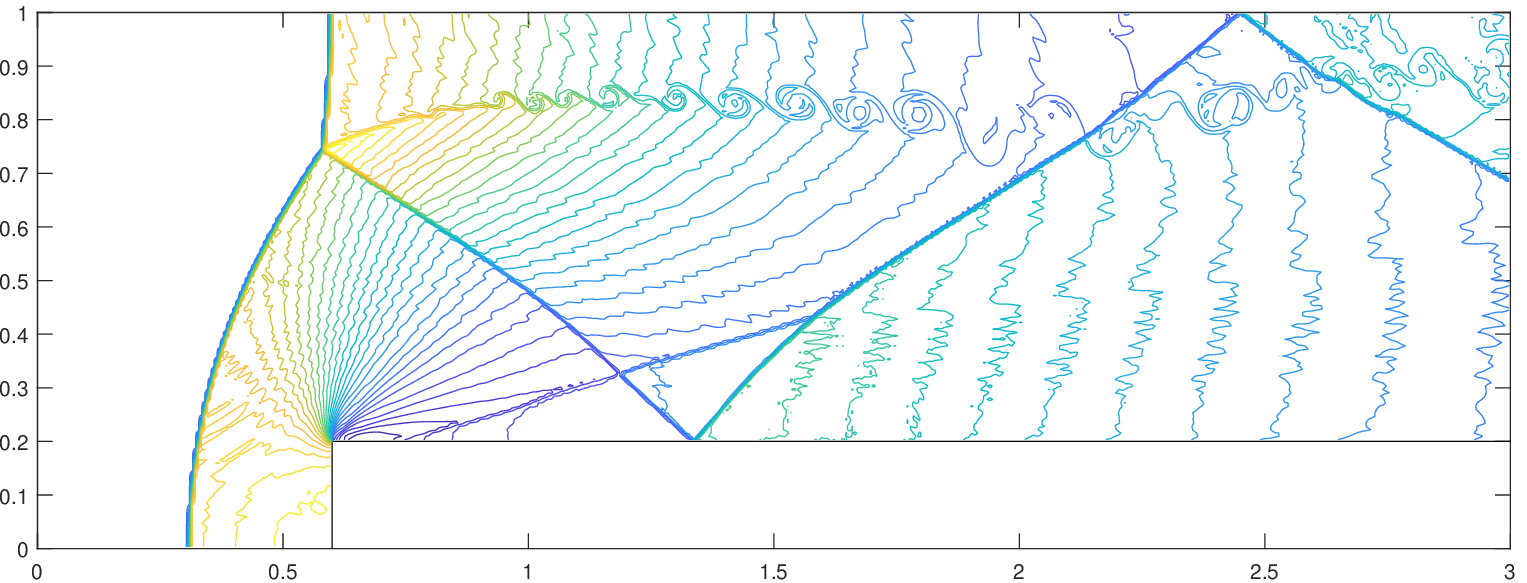}
				\end{minipage}
				\caption{Numerical results for Example \ref{ex:forward_rec}. DG solutions of $k=1,2,3,4$ from top to bottom. Left: the jump filter; Right: the associated hybrid limiter. 30 equally spaced density contours from $0.32$ to $6.15$.  $480 \times 160$ cells.}
				\label{fig:rec_forward_p1_p4}
			\end{figure}
		\end{example}
		
		\begin{example}[Shock passing a backward facing corner]\label{ex:diff_rec}
			Shock diffraction is a very common phenomenon. The computational domain for this problem is $\left[0,1\right]\times\left[6,11\right]$ and $\left[1,13\right]\times\left[0,11\right]$. The
			initial condition is a pure right-moving shock of Mach $= 5.09$, initially located at $x = 0.5$
			and $6 \leq y\leq 11$, moving into undisturbed air ahead of the shock with $(\rho,u,v,p)^T=(1.4,0,0,1)^T$. We use the inflow boundary condition at $\{x = 0, y \in [6, 11]\}$, reflective boundary condition at $\{x \in [0, 1], y = 6\}$ and $\{x = 1, y \in [0, 6]\}$, and outflow boundary condition elsewhere. To ensure the positivity of the density and pressure of the DG solutions, we first apply the jump filter, and then we utilize the positivity-preserving (PP) limiter introduced by Zhang and Shu \cite{ZHANG20108918} in Examples \ref{ex:diff_rec} and \ref{ex:highmach}. Due to the conservation nature of our scheme, integrating the PP limiter is straightforward. Based on our numerical tests, we've found that the DG scheme with the jump filter can maintain positivity even without PP limiters for certain grid configurations. However, this isn't universally true. Fig. \ref{fig:Diff_two} depicts density contour plots at time $T=2.3$. Evidently, all schemes successfully exhibit the essential property of non-oscillation in the numerical solution. Furthermore, the higher-order scheme consistently outperforms the lower-order scheme in accurately capturing the slip line.
			
			\begin{figure}[ht]
				\centering
				\begin{minipage}{0.49\linewidth}
					\centering
					\includegraphics[width=0.9\linewidth]{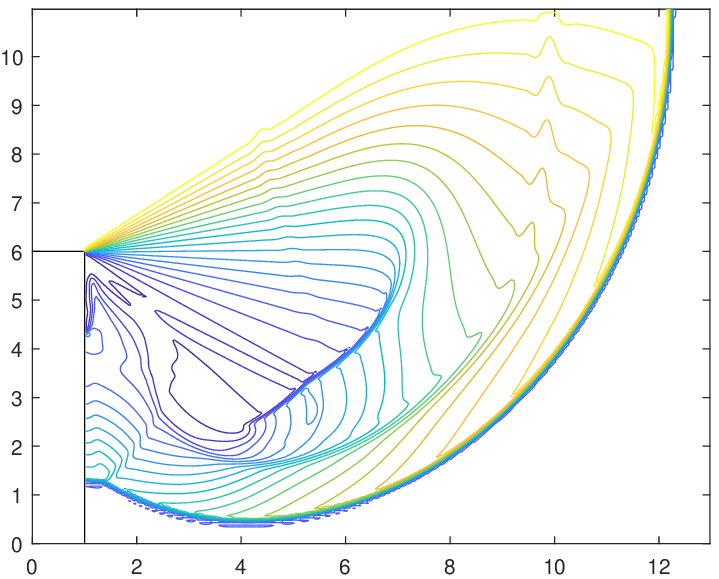}
				\end{minipage}
				\begin{minipage}{0.49\linewidth}
					\centering
					\includegraphics[width=0.9\linewidth]{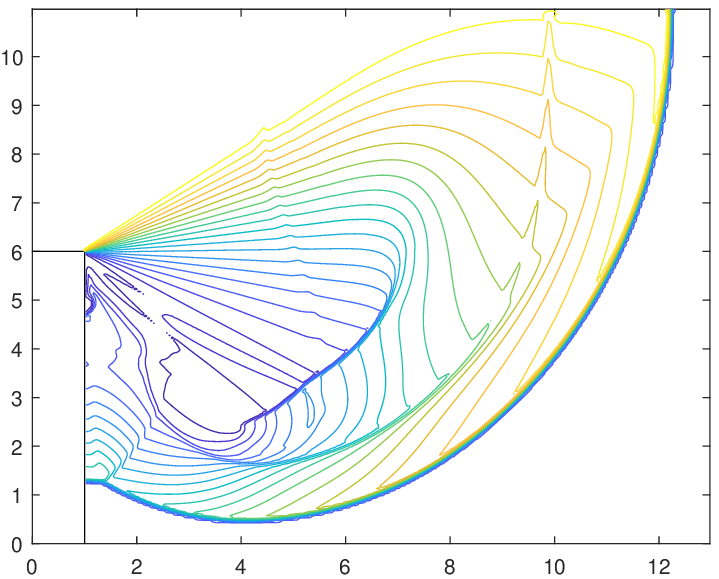}
				\end{minipage}
				\caption{Numerical results of DG scheme with the jump filter for Example \ref{ex:diff_rec}. Left: $P^1$. Right: $P^3$. 20 equally spaced density contours from $0.066227$ to $7.0668$. $260 \times 220$ cells.}
				\label{fig:Diff_two}
			\end{figure}
		\end{example}
		\begin{figure}[htbp]
			\centering
			\begin{minipage}{0.49\linewidth}
				\centering
				\includegraphics[width=0.9\linewidth]{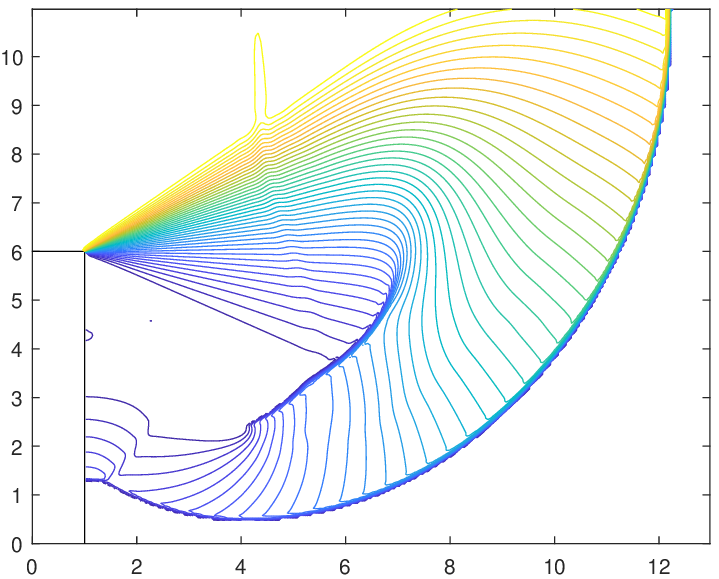}
			\end{minipage}
			\begin{minipage}{0.49\linewidth}
				\centering
				\includegraphics[width=0.9\linewidth]{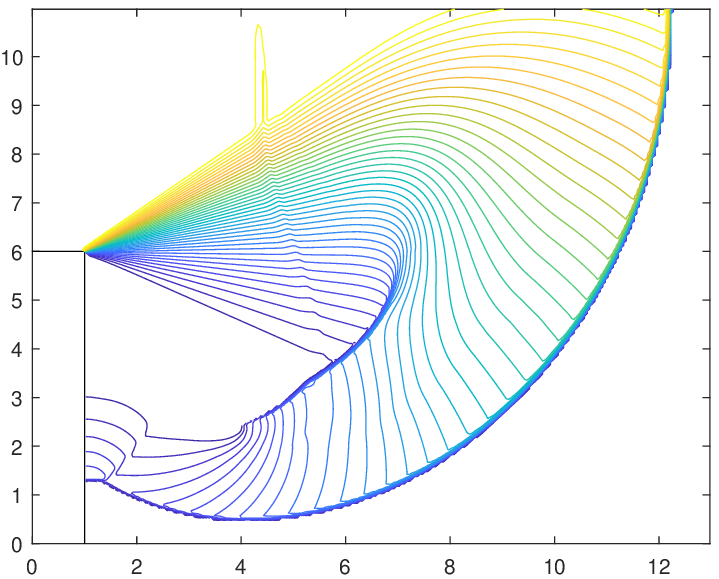}
			\end{minipage}
			\caption{Numerical results of DG scheme with the jump filter for Example \ref{ex:diff_rec}. Left: $k=1$. Right: $k=3$. 40 equally spaced pressure contours from $1.1$ to $111$. $260 \times 220$ cells.}
			\label{fig:Diff_minmod}
		\end{figure}
		\begin{example}[High Mach number astrophysical jets]\label{ex:highmach}
			We consider the high Mach number astrophysical jets without the radiative cooling \cite{ZHANG20108918}.
			A Mach $2000$ problem is considered to show the robustness of the scheme with the PP limiter \cite{ZHANG20108918}. The computational domain is taken as $[0, 1] \times [-0.25, 0.25]$, initially
			full of the ambient gas with $(\rho , u, v, p,\gamma)^{T} = (0.5, 0, 0, 0.4127,5/3)^{T}$. The right, top, and bottom boundary are outflows. For the left boundary, $(\rho , u, v, p,\gamma)^{T} = (5, 800, 0, 0.4127,5/3)^{T}$ if
			$y\in[-0.05, 0.05]$ and $(\rho , u, v, p,\gamma)^{T} = (5, 0, 0, 0.4127,5/3)^{T}$ otherwise. The numerical solution at a mesh of $320 \times 160$ is shown in Fig. \ref{fig:2000mach}. We can observe that the numerical results closely resemble the corresponding outcomes in \cite{ZHANG20108918}.
			\begin{figure}[ht]
				\centering
				\begin{minipage}[b]{0.3\linewidth}
					\centering
					\includegraphics[width=\linewidth]{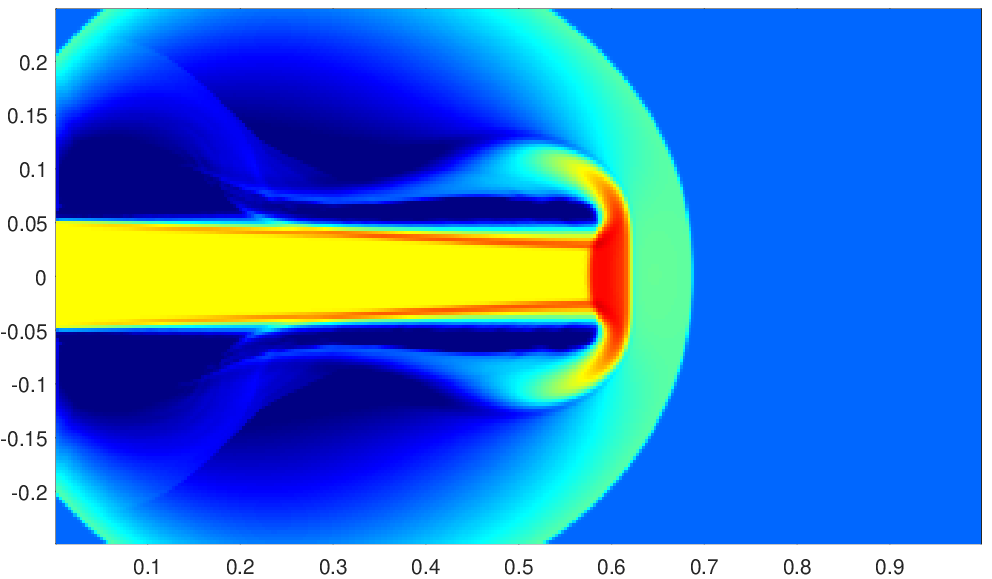}
				\end{minipage}
				\begin{minipage}[b]{0.3\linewidth}
					\centering
					\includegraphics[width=\linewidth]{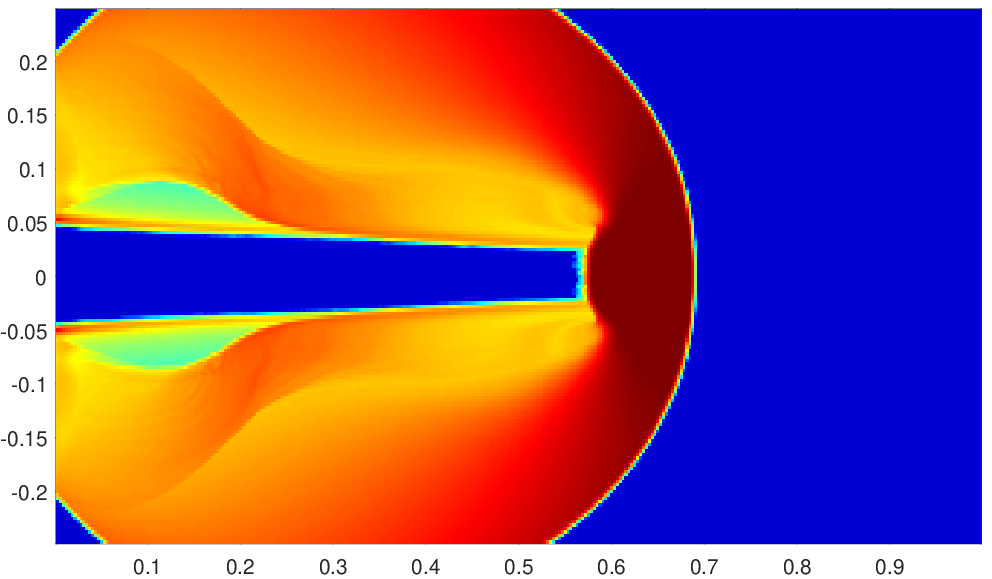}
				\end{minipage}
				\begin{minipage}[b]{0.3\linewidth}
					\centering
					\includegraphics[width=\linewidth]{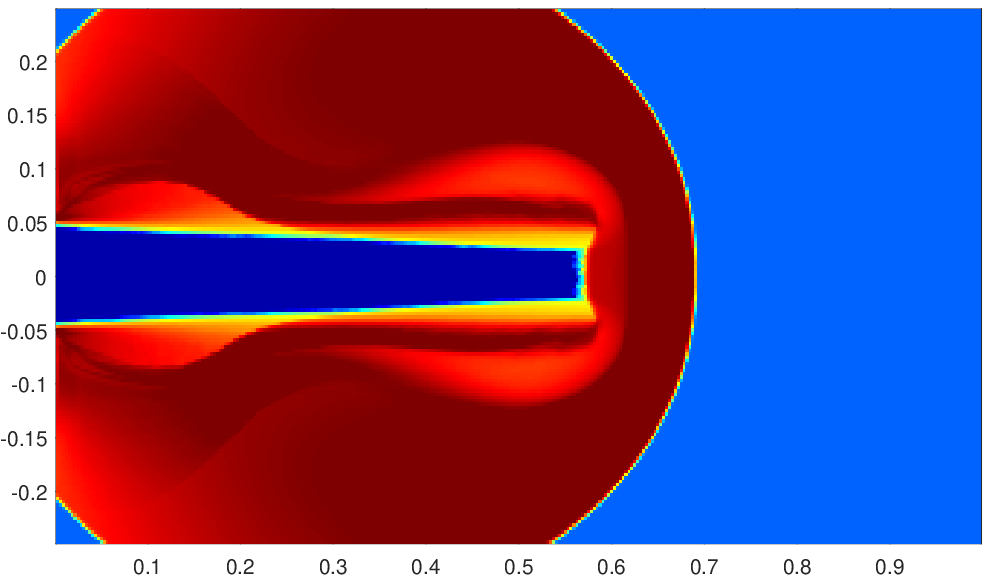}
				\end{minipage}
				\caption{Numerical results of DG scheme with the jump filter for Example \ref{ex:highmach}. From left to right: density; pressure; temperature. Scales are logarithmic. $P^3$ basis. $320\times 160$ cells.}
				\label{fig:2000mach}
			\end{figure}
			
		\end{example}

		\begin{example}[Shock-bubble interaction]\label{ex:shockbubble} In this experiment we simulate a strong rightwards moving shock wave over a low density gas bubble \cite{MR2524513}. The computational domain is $[-0.1, 1.6] \times [-0.5, 0.5]$. At the initial time, the region exhibits three different states: (A) $x < 0$; (B) a circular disk centered at $(0.3, 0.0)$ with a radius of $0.2$; (C) other regions. In region (A),  $(\rho, u, v, p)^{T}=(\frac{11}{3}, 2.7136021011998722, 0.0, 10.0)^{T}$. In region (B), $(\rho, u, v, p)^{T}=(0.1, 0.0, 0.0, 1.0)^T$. In region (C), $(\rho, u, v, p)^{T}=(1.0, 0.0, 0.0, 1.0)^T$. The left boundary has the Dirichlet boundary condition, the right has the outflow boundary condition, and the top and bottom boundaries have the solid wall boundary condition.
			In this example, a moving shockwave propagates to the right and compresses the bubble in the region (B). On the left of Fig. \ref{fig:shockbubble}, we present numerical results obtained with different schemes at $T = 0.4$, and on the right of Fig. \ref{fig:shockbubble}, we display the density distribution around the bubble. It can be observed that the higher-order schemes with the jump filter capture more unstable and intricate structures, providing a more detailed representation of the flow field.
			\begin{figure}[ht]
				\centering
				\begin{minipage}[b]{0.49\linewidth}
					\centering
					\includegraphics[width=0.9\linewidth]{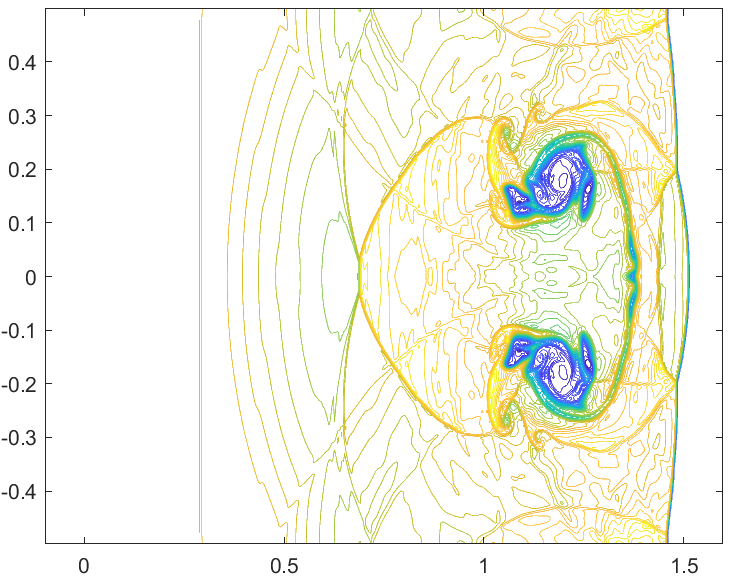}
				\end{minipage}
				\begin{minipage}[b]{0.49\linewidth}
					\centering
					\includegraphics[width=0.9\linewidth]{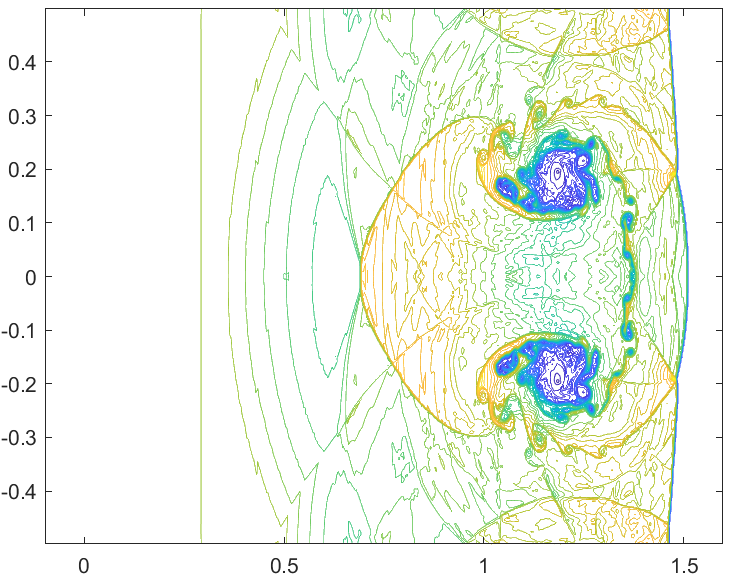}
				\end{minipage}
				\begin{minipage}[b]{0.49\linewidth}
					\centering
					\includegraphics[width=0.9\linewidth]{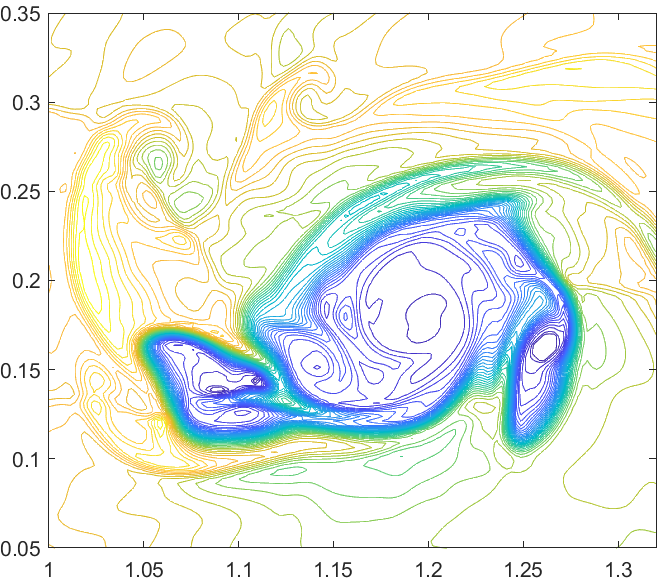}
				\end{minipage}
				\begin{minipage}[b]{0.49\linewidth}
					\centering
					\includegraphics[width=0.9\linewidth]{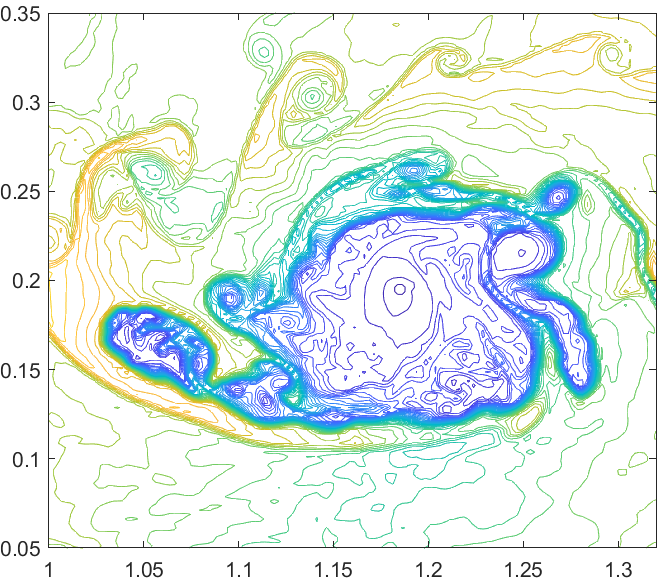}
				\end{minipage}
				\caption{Numerical results of DG scheme with the jump filter for Example \ref{ex:shockbubble}. 30 equally spaced density contours from 0.192 to 5.618. From left to right and top to bottom: $P^1$; $P^3$; $P^1$ (Zoomed-in); $P^3$ (Zoomed-in). $680 \times 400$ cells.}
				\label{fig:shockbubble}
			\end{figure}

		\end{example}

		\begin{example}[Kelvin-Helmholtz (KH) instability]\label{ex:KHinsability}
			The KH instability problem is considered in a computational domain of $[-0.5, 0.5]^2$. The initial values within this domain are set as follows:
			\[
			(\rho, u, v, p)^T =
			\begin{cases}
				(1, -0.5 + 0.5e^{(y+0.25)/L}, 0.01 \sin(4\pi x), 1.5)^T, & \text{for } y \in [-0.5, -0.25), \\
				(2, 0.5 - 0.5e^{(-y-0.25)/L}, 0.01 \sin(4\pi x), 1.5)^T, & \text{for } y \in [-0.25, 0), \\
				(2, 0.5 - 0.5e^{(y-0.25)/L}, 0.01 \sin(4\pi x), 1.5)^T, & \text{for } y \in [0, 0.25), \\
				(1, -0.5 + 0.5e^{(0.25-y)/L}, 0.01 \sin(4\pi x), 1.5)^T, & \text{for } y \in [0.25, 0.5],
			\end{cases}
			\]
			where $L=0.00625$. Periodic boundary conditions are applied and the simulation concludes at $T = 4$. Fig. \ref{fig:kh} depicts the density distribution obtained using the jump filter with $P^1$ and $P^3$ basis functions. Noticeably, the higher-order schemes generate more intricate vortex structures and turbulent mixing layers.
			
			\begin{figure}[htbp]
				\centering
				\begin{minipage}{0.49\linewidth}
					\centering
					\includegraphics[width=0.9\linewidth]{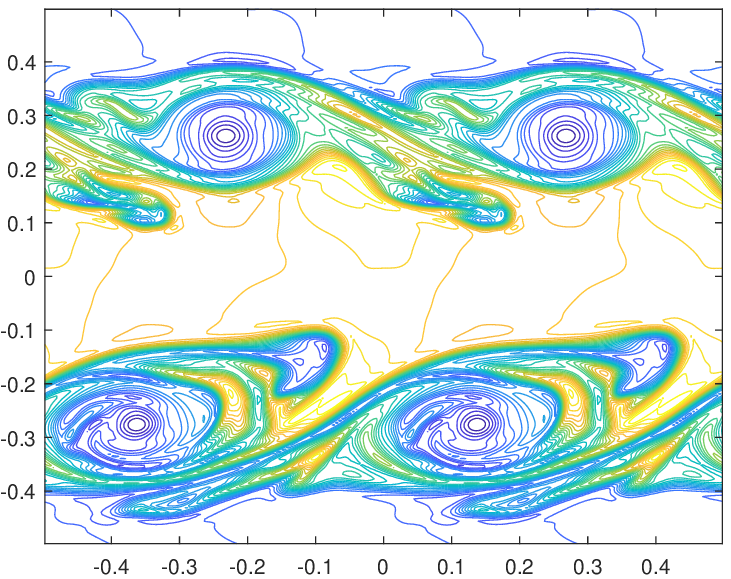}
				\end{minipage}
				\begin{minipage}{0.49\linewidth}
					\centering
					\includegraphics[width=0.9\linewidth]{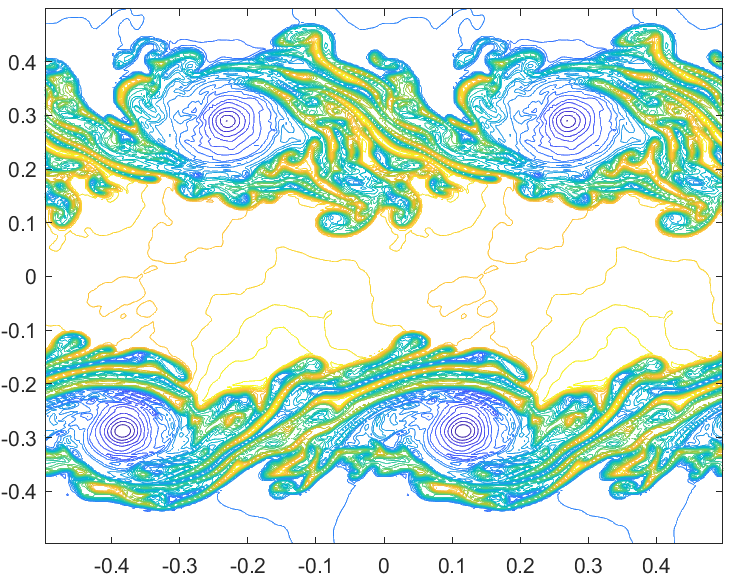}
				\end{minipage}
				\caption{Numerical results of DG scheme with the jump filter for Example \ref{ex:KHinsability}. 30 equally spaced density contours from 0.7 to 2.3. Left: $P^1$. Right: $P^3$. $240\times 240$ cells. }
				\label{fig:kh}
			\end{figure}
		\end{example}
		
		\section{Conclusions}\label{sec:conclusions}
		By incorporating a local viscosity term based solely on cell interface jumps, this paper introduces a novel shock capturing approach to the DG scheme. This scheme is efficiently implemented through a splitting fashion without impacting the chosen time step. Its favorable properties are ensured by rigorous mathematical theory, guaranteeing conservation, $L^2$ stability, and optimal error estimation. Extensive numerical experiments, primarily focusing on various test cases of the Euler equations, validate the proposed method, showcasing its efficacy in controlling numerical spurious oscillations, particularly in scenarios involving strong shock waves.  Furthermore, by integrating the novel jump filter into a hybrid limiter framework, the scheme effectively suppresses numerical spurious oscillations while further reducing dissipation. The jump filter maintains the compactness of the DG scheme, facilitating easy implementation and enabling efficient parallel computations. An important advantage of this approach is its simplicity and computational efficiency. For each time step, only the multiplication of the coefficients by a precomputed factor is necessary. This factor is calculated in physical space without the need for characteristic decomposition. It appears that the jump filter can be seamlessly implemented on unstructured meshes, and we plan to explore this capability and its application in steady-state problems in our future work to enhance its versatility and applicability.

		\bibliographystyle{plain}

\begin{thebibliography}{10}
			
			\bibitem{MR3209963}
			H.~Abbassi, F.~Mashayek, and G.B. Jacobs.
			\newblock Shock capturing with entropy-based artificial viscosity for staggered
			grid discontinuous spectral element method.
			\newblock {\em Comput. \& Fluids}, 98:152--163, 2014.
			
			\bibitem{barter2007shock}
			G.~Barter and D.~Darmofal.
			\newblock {Shock capturing with higher-order, PDE-based artificial viscosity}.
			\newblock In {\em 18th AIAA Computational Fluid Dynamics Conference}, page
			3823, 2007.
			
			\bibitem{bassi2009high}
			F.~Bassi, A.~Crivellini, A.~Ghidoni, and S.~Rebay.
			\newblock {High-order discontinuous Galerkin discretization of transonic
				turbulent flows}.
			\newblock In {\em 47th AIAA aerospace sciences meeting including the new
				horizons forum and aerospace exposition}, page 180, 2009.
			
			\bibitem{bhagatwala2009modified}
			A.~Bhagatwala and S.K. Lele.
			\newblock A modified artificial viscosity approach for compressible turbulence
			simulations.
			\newblock {\em J. Comput. Phys.}, 228(14):4965--4969, 2009.
			
			\bibitem{biswas1994parallel}
			R.~Biswas, K.D. Devine, and J.~Flaherty.
			\newblock Parallel, adaptive finite element methods for conservation laws.
			\newblock {\em Appl. Numer. Math.}, 14(1-3):255--283, 1994.
			
			\bibitem{burbeau2001problem}
			A.~Burbeau, P.~Sagaut, and Ch.-H. Bruneau.
			\newblock A problem-independent limiter for high-order {R}unge-{K}utta
			discontinuous {G}alerkin methods.
			\newblock {\em J. Comput. Phys.}, 169(1):111--150, 2001.
			
			\bibitem{burman2007nonlinear}
			E.~Burman.
			\newblock On nonlinear artificial viscosity, discrete maximum principle and
			hyperbolic conservation laws.
			\newblock {\em BIT Numerical Mathematics}, 47:715--733, 2007.
			
			\bibitem{MR2524513}
			M.~Cada and M.~Torrilhon.
			\newblock Compact third-order limiter functions for finite volume methods.
			\newblock {\em J. Comput. Phys.}, 228(11):4118--4145, 2009.
			
			\bibitem{chaudhuri2017explicit}
			A.~Chaudhuri, G.B. Jacobs, W.-S. Don, H~Abbassi, and F.~Mashayek.
			\newblock Explicit discontinuous spectral element method with entropy
			generation based artificial viscosity for shocked viscous flows.
			\newblock {\em J. Comput. Phys.}, 332:99--117, 2017.
			
			\bibitem{MR1930132}
			P.G. Ciarlet.
			\newblock {\em The finite element method for elliptic problems}, volume~40 of
			{\em Classics in Applied Mathematics}.
			\newblock Society for Industrial and Applied Mathematics (SIAM), Philadelphia,
			PA.
			
			\bibitem{cockburn1989tvb2}
			B.~Cockburn, S.-Y. Lin, and C.-W. Shu.
			\newblock {TVB Runge-Kutta local projection discontinuous Galerkin finite
				element method for conservation laws III: one-dimensional systems}.
			\newblock {\em J. Comput. Phys.}, 84(1):90--113, 1989.
			
			\bibitem{cockburn1989tvb1}
			B.~Cockburn and C.-W. Shu.
			\newblock {TVB Runge-Kutta local projection discontinuous Galerkin finite
				element method for conservation laws. II. General framework}.
			\newblock {\em Math. Comput.}, 52(186):411--435, 1989.
			
			\bibitem{cockburn1998runge}
			B.~Cockburn and C.-W. Shu.
			\newblock {The Runge-Kutta discontinuous Galerkin method for conservation laws
				V: multidimensional systems}.
			\newblock {\em J. Comput. Phys.}, 141(2):199--224, 1998.
			
			\bibitem{cook2004high}
			A.W. Cook and W.H. Cabot.
			\newblock A high-wavenumber viscosity for high-resolution numerical methods.
			\newblock {\em J. Comput. Phys.}, 195(2):594--601, 2004.
			
			\bibitem{cook2005hyperviscosity}
			A.W. Cook and W.H. Cabot.
			\newblock Hyperviscosity for shock-turbulence interactions.
			\newblock {\em J. Comput. Phys.}, 203(2):379--385, 2005.
			
			\bibitem{don2003multidomain}
			W.-S. Don, D.~Gottlieb, and J.-H. Jung.
			\newblock A multidomain spectral method for supersonic reactive flows.
			\newblock {\em J. Comput. Phys.}, 192(1):325--354, 2003.
			
			\bibitem{don1994numerical}
			W.S. Don.
			\newblock Numerical study of pseudospectral methods in shock wave applications.
			\newblock {\em J. Comput. Phys.}, 110(1):103--111, 1994.
			
			\bibitem{fernandez2018physics}
			P.~Fernandez, C.~Nguyen, and J.~Peraire.
			\newblock A physics-based shock capturing method for unsteady laminar and
			turbulent flows.
			\newblock In {\em 2018 AIAA Aerospace Sciences Meeting}, page 0062, 2018.
			
			\bibitem{fiorina2007artificial}
			B.~Fiorina and S.K. Lele.
			\newblock An artificial nonlinear diffusivity method for supersonic reacting
			flows with shocks.
			\newblock {\em J. Comput. Phys.}, 222(1):246--264, 2007.
			
			\bibitem{gottlieb2001spectral}
			D.~Gottlieb and J.S. Hesthaven.
			\newblock Spectral methods for hyperbolic problems.
			\newblock {\em J. Comput. Appl. Math.}, 128(1-2):83--131, 2001.
			
			\bibitem{Gottlieb2001Strong}
			S.~Gottlieb, C.-W. Shu, and E.~Tadmor.
			\newblock Strong stability-preserving high-order time discretization methods.
			\newblock {\em SIAM Rev.}, 43:89--112, 2001.
			
			\bibitem{guermond2008entropy}
			J.-L. Guermond and R.~Pasquetti.
			\newblock {Entropy-based nonlinear viscosity for Fourier approximations of
				conservation laws}.
			\newblock {\em C. R. Math. Acad. Sci. Paris}, 346(13-14):801--806, 2008.
			
			\bibitem{MR2787948}
			J.-L. Guermond, R.~Pasquetti, and B.~Popov.
			\newblock Entropy viscosity method for nonlinear conservation laws.
			\newblock {\em J. Comput. Phys.}, 230(11):4248--4267, 2011.
			
			\bibitem{harten1987uniformly}
			A.~Harten, B.~Engquist, S.~Osher, and S.~Chakravarthy.
			\newblock Uniformly high order accurate essentially non-oscillatory schemes,
			iii.
			\newblock {\em J. Comput. Phys.}, 131(1):3--47, 1987.
			
			\bibitem{harten1983upstream}
			A.~Harten, P.D. Lax, and B.~Van~Leer.
			\newblock On upstream differencing and godunov-type schemes for hyperbolic
			conservation laws.
			\newblock {\em SIAM Rev}, 25(1):35--61, 1983.
			
			\bibitem{hartmann2006adaptive}
			R.~Hartmann.
			\newblock Adaptive discontinuous {G}alerkin methods with shock-capturing for
			the compressible {N}avier-{S}tokes equations.
			\newblock {\em Internat. J. Numer. Methods Fluids}, 51(9-10):1131--1156, 2006.
			
			\bibitem{hartmann2013higher}
			R.~Hartmann.
			\newblock Higher-order and adaptive discontinuous {G}alerkin methods with
			shock-capturing applied to transonic turbulent delta wing flow.
			\newblock {\em Internat. J. Numer. Methods Fluids}, 72(8):883--894, 2013.
			
			\bibitem{hartmann2002adaptive}
			R.~Hartmann and P.~Houston.
			\newblock {Adaptive discontinuous Galerkin finite element methods for the
				compressible Euler equations}.
			\newblock {\em J. Comput. Phys.}, 183(2):508--532, 2002.
			
			\bibitem{hesthaven2017numerical}
			J.S. Hesthaven.
			\newblock {\em {Numerical methods for conservation laws: From analysis to
					algorithms}}.
			\newblock SIAM, 2017.
			
			\bibitem{hesthaven2008filtering}
			J.S. Hesthaven and R.~Kirby.
			\newblock {Filtering in Legendre spectral methods}.
			\newblock {\em Math. Comput.}, 77(263):1425--1452, 2008.
			
			\bibitem{hesthaven2007nodal}
			J.S. Hesthaven and T.~Warburton.
			\newblock {\em {Nodal discontinuous Galerkin methods: algorithms, analysis, and
					applications}}.
			\newblock Springer Science \& Business Media, 2007.
			
			\bibitem{hiltebrand2014entropy}
			A.~Hiltebrand and S.~Mishra.
			\newblock Entropy stable shock capturing space--time discontinuous galerkin
			schemes for systems of conservation laws.
			\newblock {\em Numerische Mathematik}, 126:103--151, 2014.
			
			\bibitem{hou2007solutions}
			S.~Hou and X.D. Liu.
			\newblock {Solutions of multi-dimensional hyperbolic systems of conservation
				laws by square entropy condition satisfying discontinuous Galerkin method}.
			\newblock {\em J. Sci. Comput.}, 31:127--151, 2007.
			
			\bibitem{MR0759810}
			T.~J.~R. Hughes and T.~E. Tezduyar.
			\newblock Finite element methods for first-order hyperbolic systems with
			particular emphasis on the compressible {E}uler equations.
			\newblock {\em Comput. Methods Appl. Mech. Engrg.}, 45(1-3):217--284, 1984.
			
			\bibitem{jameson1993artificial}
			A.~Jameson.
			\newblock {Artificial diffusion, upwind biasing, limiters and their effect on
				accuracy and multigrid convergence in transonic and hypersonic flows}.
			\newblock In {\em 11th Computational Fluid Dynamics Conference}, page 3359,
			1993.
			
			\bibitem{jameson2001perspective}
			A.~Jameson.
			\newblock A perspective on computational algorithms for aerodynamic analysis
			and design.
			\newblock {\em Progress in Aerospace Sciences}, 37(2):197--243, 2001.
			
			\bibitem{jameson1981numerical}
			A.~Jameson, W.~Schmidt, and E.~Turkel.
			\newblock {Numerical solution of the Euler equations by finite volume methods
				using Runge Kutta time stepping schemes}.
			\newblock In {\em 14th fluid and plasma dynamics conference}, page 1259, 1981.
			
			\bibitem{jiang1994cell}
			G.-S. Jiang and C.-W. Shu.
			\newblock {On a cell entropy inequality for discontinuous Galerkin methods}.
			\newblock {\em Math. Comput.}, 62(206):531--538, 1994.
			
			\bibitem{jiang1996efficient}
			G.-S. Jiang and C.-W. Shu.
			\newblock {Efficient implementation of weighted ENO schemes}.
			\newblock {\em J. Comput. Phys.}, 126(1):202--228, 1996.
			
			\bibitem{karamanos2000spectral}
			G-S. Karamanos and G.E. Karniadakis.
			\newblock A spectral vanishing viscosity method for large-eddy simulations.
			\newblock {\em J. Comput. Phys.}, 163(1):22--50, 2000.
			
			\bibitem{kawai2008localized}
			S.~Kawai and S.K. Lele.
			\newblock Localized artificial diffusivity scheme for discontinuity capturing
			on curvilinear meshes.
			\newblock {\em J. Comput. Phys.}, 227(22):9498--9526, 2008.
			
			\bibitem{klockner2009nodal}
			A.~Kl{\"o}ckner, T.~Warburton, J.~Bridge, and J.S. Hesthaven.
			\newblock {Nodal discontinuous Galerkin methods on graphics processors}.
			\newblock {\em J. Comput. Phys.}, 228(21):7863--7882, 2009.
			
			\bibitem{krivodonova2007limiters}
			L.~Krivodonova.
			\newblock {Limiters for high-order discontinuous Galerkin methods}.
			\newblock {\em J. Comput. Phys.}, 226(1):879--896, 2007.
			
			\bibitem{kuzmin2010vertex}
			D.~Kuzmin.
			\newblock {A vertex-based hierarchical slope limiter for p-adaptive
				discontinuous Galerkin methods}.
			\newblock {\em J. Comput. Appl. Math.}, 233(12):3077--3085, 2010.
			
			\bibitem{kuzmin2013slope}
			D.~Kuzmin.
			\newblock {Slope limiting for discontinuous Galerkin approximations with a
				possibly non-orthogonal Taylor basis}.
			\newblock {\em Internat. J. Numer. Methods Fluids}, 71(9):1178--1190, 2013.
			
			\bibitem{kuzmin2014hierarchical}
			D.~Kuzmin.
			\newblock {Hierarchical slope limiting in explicit and implicit discontinuous
				Galerkin methods}.
			\newblock {\em J. Comput. Phys.}, 257:1140--1162, 2014.
			
			\bibitem{leveque1992numerical}
			Randall~J. LeVeque.
			\newblock {\em Numerical methods for conservation laws}.
			\newblock Lectures in Mathematics ETH Z\"{u}rich. Birkh\"{a}user Verlag, Basel,
			second edition, 1992.
			
			\bibitem{liu1994weighted}
			X.-D. Liu, S.~Osher, and T.~Chan.
			\newblock Weighted essentially non-oscillatory schemes.
			\newblock {\em J. Comput. Phys.}, 115(1):200--212, 1994.
			
			\bibitem{liu2022essentially}
			Y.~Liu, J.~Lu, and C.-W. Shu.
			\newblock {An essentially oscillation-free discontinuous Galerkin method for
				hyperbolic systems}.
			\newblock {\em SIAM J. Sci. Comput.}, 44(1):A230--A259, 2022.
			
			\bibitem{liu2006spectral}
			Y.~Liu, M.~Vinokur, and Z.~Wang.
			\newblock {Spectral difference method for unstructured grids I: Basic
				formulation}.
			\newblock {\em J. Comput. Phys.}, 216(2):780--801, 2006.
			
			\bibitem{lodato2019characteristic}
			G.~Lodato.
			\newblock Characteristic modal shock detection for discontinuous finite element
			methods.
			\newblock {\em Comput. \& Fluids}, 179:309--333, 2019.
			
			\bibitem{lu2021oscillation}
			J.~Lu, Y.~Liu, and C.-W. Shu.
			\newblock {An oscillation-free discontinuous Galerkin method for scalar
				hyperbolic conservation laws}.
			\newblock {\em SIAM J. Numer. Anal.}, 59(3):1299--1324, 2021.
			
			\bibitem{LuoHWENO2007}
			H.~Luo, J.D. Baum, and R.~L\"{o}hner.
			\newblock A {H}ermite {WENO}-based limiter for discontinuous {G}alerkin method
			on unstructured grids.
			\newblock {\em J. Comput. Phys.}, 225(1):686--713, 2007.
			
			\bibitem{MR3534873}
			Y.~Lv, Y.~See, and M.~Ihme.
			\newblock An entropy-residual shock detector for solving conservation laws
			using high-order discontinuous {G}alerkin methods.
			\newblock {\em J. Comput. Phys.}, 322:448--472, 2016.
			
			\bibitem{maday1993legendre}
			Y.~Maday, S.M.O. Kaber, and E.~Tadmor.
			\newblock Legendre pseudospectral viscosity method for nonlinear conservation
			laws.
			\newblock {\em SIAM J. Numer. Anal.}, 30(2):321--342, 1993.
			
			\bibitem{MR4583858}
			J.~Markert, G.~Gassner, and S.~Walch.
			\newblock A sub-element adaptive shock capturing approach for discontinuous
			{G}alerkin methods.
			\newblock {\em Commun. Appl. Math. Comput.}, 5(2):679--721, 2023.
			
			\bibitem{meister2012application}
			A.~Meister, S.~Ortleb, and Th. Sonar.
			\newblock Application of spectral filtering to discontinuous {G}alerkin methods
			on triangulations.
			\newblock {\em Numer. Methods Partial Differential Equations},
			28(6):1840--1868, 2012.
			
			\bibitem{messai2024artificial}
			N.-A. Messa\"{\i}, G.~Daviller, and J.-F. Boussuge.
			\newblock Artificial viscosity-based shock capturing scheme for the {S}pectral
			{D}ifference method on simplicial elements.
			\newblock {\em J. Comput. Phys.}, 504:Paper No. 112864, 31, 2024.
			
			\bibitem{moro2016dilation}
			David Moro, Ngoc~Cuong Nguyen, and Jaime Peraire.
			\newblock Dilation-based shock capturing for high-order methods.
			\newblock {\em Internat. J. Numer. Methods Fluids}, 82(7):398--416, 2016.
			
			\bibitem{MR3008293}
			M.~Nazarov and J.~Hoffman.
			\newblock Residual-based artificial viscosity for simulation of turbulent
			compressible flow using adaptive finite element methods.
			\newblock {\em Internat. J. Numer. Methods Fluids}, 71(3):339--357, 2013.
			
			\bibitem{peng2023oedg}
			M.~Peng, Z.~Sun, and K.~Wu.
			\newblock {OEDG: Oscillation-eliminating discontinuous Galerkin method for
				hyperbolic conservation laws}.
			\newblock {\em arXiv preprint arXiv:2310.04807}, 2023.
			
			\bibitem{persson2013shock}
			P.O. Persson.
			\newblock {Shock capturing for high-order discontinuous Galerkin simulation of
				transient flow problems}.
			\newblock In {\em 21st AIAA computational fluid dynamics conference}, page
			3061, 2013.
			
			\bibitem{persson2006sub}
			P.O. Persson and J.~Peraire.
			\newblock {Sub-cell shock capturing for discontinuous Galerkin methods}.
			\newblock In {\em 44th AIAA aerospace sciences meeting and exhibit}, page 112,
			2006.
			
			\bibitem{premasuthan2014computation}
			S.~Premasuthan, C.~Liang, and A.~Jameson.
			\newblock {Computation of flows with shocks using the spectral difference
				method with artificial viscosity, I: basic formulation and application}.
			\newblock {\em Comput. \& Fluids}, 98:111--121, 2014.
			
			\bibitem{premasuthan2014computation2}
			S.~Premasuthan, C.~Liang, and A.~Jameson.
			\newblock {Computation of flows with shocks using the spectral difference
				method with artificial viscosity, II: modified formulation with local mesh
				refinement}.
			\newblock {\em Comput. \& Fluids}, 98:122--133, 2014.
			
			\bibitem{qiu2005rungelimiter}
			J.~Qiu and C.-W. Shu.
			\newblock {Runge-Kutta discontinuous Galerkin method using WENO limiters}.
			\newblock {\em SIAM J. Sci. Comput.}, 26(3):907--929, 2005.
			
			\bibitem{MR4135390}
			V.~Shankar, G.B. Wright, and A.~Narayan.
			\newblock A robust hyperviscosity formulation for stable {RBF}-{FD}
			discretizations of advection-diffusion-reaction equations on manifolds.
			\newblock {\em SIAM J. Sci. Comput.}, 42(4):A2371--A2401, 2020.
			
			\bibitem{shu2014discontinuous}
			C.-W. Shu.
			\newblock {Discontinuous Galerkin method for time-dependent problems: survey
				and recent developments}.
			\newblock {\em Recent developments in discontinuous Galerkin finite element
				methods for partial differential equations}, pages 25--62, 2014.
			
			\bibitem{MR3608339}
			M.~Sonntag and C.~Munz.
			\newblock Efficient parallelization of a shock capturing for discontinuous
			{G}alerkin methods using finite volume sub-cells.
			\newblock {\em J. Sci. Comput.}, 70(3):1262--1289, 2017.
			
			\bibitem{tadmor1989convergence}
			E.~Tadmor.
			\newblock Convergence of spectral methods for nonlinear conservation laws.
			\newblock {\em SIAM J. Numer. Anal.}, 26(1):30--44, 1989.
			
			\bibitem{MR1305618}
			E.~Tadmor.
			\newblock Super-viscosity and spectral approximations of nonlinear conservation
			laws.
			\newblock In {\em Numerical methods for fluid dynamics, 4 ({R}eading, 1992)},
			pages 69--81. Oxford Univ. Press, New York, 1993.
			
			\bibitem{MR4517287}
			I.~Tominec and M.~Nazarov.
			\newblock Residual viscosity stabilized {RBF}-{FD} methods for solving
			nonlinear conservation laws.
			\newblock {\em J. Sci. Comput.}, 94(1):Paper No. 14, 31, 2023.
			
			\bibitem{tonicello2020entropy}
			N.~Tonicello, G.~Lodato, and L.~Vervisch.
			\newblock Entropy preserving low dissipative shock capturing with
			wave-characteristic based sensor for high-order methods.
			\newblock {\em Comput. \& Fluids}, 197:104357, 17, 2020.
			
			\bibitem{toro2013riemann}
			E.F. Toro.
			\newblock {\em Riemann solvers and numerical methods for fluid dynamics: a
				practical introduction}.
			\newblock Springer Science \& Business Media, 2013.
			
			\bibitem{vandeven1991family}
			H.~Vandeven.
			\newblock Family of spectral filters for discontinuous problems.
			\newblock {\em J. Sci. Comput.}, 6(2):159--192, 1991.
			
			\bibitem{MR0037613}
			J.~Von~Neumann and R.~D. Richtmyer.
			\newblock A method for the numerical calculation of hydrodynamic shocks.
			\newblock {\em J. Appl. Phys.}, 21:232--237, 1950.
			
			\bibitem{wan2021new}
			Y.~Wan and Y.~Xia.
			\newblock A new hybrid {WENO} scheme with the high-frequency region for
			hyperbolic conservation laws.
			\newblock {\em Commun. Appl. Math. Comput.}, 5(1):199--234, 2023.
			
			\bibitem{MR4673601}
			L.~Wei and Y.~Xia.
			\newblock An indicator-based hybrid limiter in discontinuous {G}alerkin methods
			for hyperbolic conservation laws.
			\newblock {\em J. Comput. Phys.}, 498:Paper No. 112676, 25, 2024.
			
			\bibitem{MR4125673}
			Y.~Xu, X.~Meng, C.-W. Shu, and Q.~Zhang.
			\newblock Superconvergence analysis of the {R}unge-{K}utta discontinuous
			{G}alerkin methods for a linear hyperbolic equation.
			\newblock {\em J. Sci. Comput.}, 84(1):Paper No. 23, 40, 2020.
			
			\bibitem{MR4161755}
			Y.~Xu, C.-W. Shu, and Q.~Zhang.
			\newblock Error estimate of the fourth-order {R}unge-{K}utta discontinuous
			{G}alerkin methods for linear hyperbolic equations.
			\newblock {\em SIAM J. Numer. Anal.}, 58(5):2885--2914, 2020.
			
			\bibitem{MR3977110}
			Y.~Xu, Q.~Zhang, C.-W. Shu, and H.~Wang.
			\newblock The {$\rm L^2$}-norm stability analysis of {R}unge-{K}utta
			discontinuous {G}alerkin methods for linear hyperbolic equations.
			\newblock {\em SIAM J. Numer. Anal.}, 57(4):1574--1601, 2019.
			
			\bibitem{yu2020study}
			J.~Yu and J.S. Hesthaven.
			\newblock {A study of several artificial viscosity models within the
				discontinuous Galerkin framework}.
			\newblock {\em Commun. Comput. Phys.}, 27(5):1309--1343, 2020.
			
			\bibitem{ZHANG20108918}
			X.~Zhang and C.-W. Shu.
			\newblock {On positivity-preserving high order discontinuous Galerkin schemes
				for compressible Euler equations on rectangular meshes}.
			\newblock {\em J. Comput. Phys.}, 229(23):8918--8934, 2010.
			
			\bibitem{zhong2013simple}
			X.~Zhong and C.-W. Shu.
			\newblock {A simple weighted essentially nonoscillatory limiter for Runge-Kutta
				discontinuous Galerkin methods}.
			\newblock {\em J. Comput. Phys.}, 232(1):397--415, 2013.
			
			\bibitem{zhu2020high}
			J.~Zhu, J.~Qiu, and C.-W. Shu.
			\newblock High-order {R}unge-{K}utta discontinuous {G}alerkin methods with a
			new type of multi-resolution {WENO} limiters.
			\newblock {\em J. Comput. Phys.}, 404:109105, 18, 2020.
			
			\bibitem{zhu2008runge}
			J.~Zhu, J.~Qiu, C.-W. Shu, and M.~Dumbser.
			\newblock {Runge--Kutta discontinuous Galerkin method using WENO limiters II:
				unstructured meshes}.
			\newblock {\em J. Comput. Phys.}, 227(9):4330--4353, 2008.
			
			\bibitem{zingan2013implementation}
			V.~Zingan, J.-L. Guermond, J.~Morel, and B.~Popov.
			\newblock Implementation of the entropy viscosity method with the discontinuous
			{G}alerkin method.
			\newblock {\em Comput. Methods Appl. Mech. Engrg.}, 253:479--490, 2013.
			
		\end{thebibliography}

	\end{document}